\RequirePackage{fix-cm}
\documentclass[smallextended]{svjour3}       
\smartqed  
\usepackage{graphicx}
\journalname{Machine Learning}

\usepackage{algorithm}
\usepackage{algorithmic}
\usepackage{amsmath,amssymb,amsfonts}
\usepackage{natbib}
\usepackage{multirow}
\usepackage{hyperref}
\newcommand{\invisible}[1]{{}}

\renewcommand{\dfrac}[2]{\frac{\displaystyle #1}{\displaystyle  #2}}
\renewcommand{\a}{\mathbf{a}}
\newcommand{\x}{\mathbf{x}}
\newcommand{\y}{\mathbf{y}}
\newcommand{\e}{\mathbf{e}}
\newcommand{\db}{\mathbf{b}}
\newcommand{\w}{\mathbf{w}}
\newcommand{\z}{\mathbf{z}}
\newcommand{\bx}{\mathbf{X}}
\newcommand{\by}{\mathbf{Y}}
\newcommand{\bz}{\mathbf{Z}}
\newcommand{\be}{\mathbf{E}}
\newcommand{\bl}{\mathbf{L}}
\newcommand{\bs}{\mathbf{S}}

\renewcommand{\u}{\mathbf{u}}
\renewcommand{\v}{\mathbf{v}}
\newcommand{\p}{\mathbf{p}}
\renewcommand{\b}{\mathbf{b}}
\newcommand{\bflambda}{\mathbf{\lambda}}
\newcommand{\A}{\mathcal{A}}
\newcommand{\B}{\mathcal{B}}
\newcommand{\X}{\mathbf{X}}

\newcommand{\sgn}{\mbox{sgn}}
\newcommand{\bzero}{\mathbf{0}}

\newcommand{\prox}{\mbox{prox}}
\newcommand{\<}{\left\langle}
\renewcommand{\>}{\right\rangle}
\newcommand{\lbar}{\left\|}
\newcommand{\rbar}{\right\|}

\DeclareMathOperator*{\argmin}{argmin}

\begin{document}

\title{Linearized Alternating Direction Method with Parallel Splitting and Adaptive Penalty
for Separable Convex Programs in Machine Learning\thanks{This
paper is an extension of our prior work
\cite{Lin-2011-LADMAP} and \cite{Liu-2013-LADMPSAP}.}
}

\titlerunning{LADMPSAP for Separable Convex Programs}        

\author{Zhouchen Lin        \and
        Risheng Liu  \and
        Huan Li 
}


\institute{Z. Lin \at
              Key Lab. of Machine Perception (MOE), School of EECS, Peking
              University.
              \email{zlin@pku.edu.cn}           
           \and
           R. Liu (Corresponding author)\at
              Faculty of Electronic Information and Electrical Engineering, Dalian University of Technology,\\
              School of Software Technology, Dalian University of
              Technology.
              \email{rsliu@dlut.edu.cn}
           \and
           H. Li \at
           School of Software and Microelectronics, Peking
           University.
           \email{lihuan\_ss@pku.edu.cn}
}

\date{Received: date / Accepted: date}

\maketitle

\begin{abstract}
Many problems in machine learning and other fields can be
(re)for-mulated as linearly constrained separable convex programs.
In most of the cases, there are multiple blocks of variables.
However, the traditional alternating direction method (ADM) and
its linearized version (LADM, obtained by linearizing the
quadratic penalty term) are for the two-block case and cannot be
naively generalized to solve the multi-block case. So there is
great demand on extending the ADM based methods for the
multi-block case. In this paper, we propose LADM with parallel
splitting and adaptive penalty (LADMPSAP) to solve multi-block
separable convex programs efficiently. When all the component
objective functions have bounded subgradients, we obtain
convergence results that are stronger than those of ADM and LADM,
e.g., allowing the penalty parameter to be \emph{unbounded} and
proving the \emph{sufficient and necessary conditions} for global
convergence. We further propose a simple optimality measure and
reveal the convergence rate of LADMPSAP in an ergodic sense. For
programs with extra convex set constraints, with refined parameter
estimation we devise a practical version of LADMPSAP for faster
convergence. Finally, we generalize LADMPSAP to handle programs
with more difficult objective functions by linearizing part of the
objective function as well. LADMPSAP is particularly suitable for
sparse representation and low-rank recovery problems because its
subproblems have closed form solutions and the sparsity and
low-rankness of the iterates can be preserved during the
iteration. It is also highly parallelizable and hence fits for
parallel or distributed computing. Numerical experiments testify
to the advantages of LADMPSAP in speed and numerical accuracy.
\end{abstract}

\section{Introduction}
In recent years, convex programs have become increasingly popular
for solving a wide range of problems in machine learning and other
fields, ranging from theoretical modeling, e.g., latent variable
graphical model selection \citep{Chand-2012-GMS}, low-rank feature
extraction (e.g., matrix decomposition \citep{Candes-2011-RPCA}
and matrix completion \citep{Candes-2009-matrix}), subspace
clustering \citep{Liu-2012-LRR}, and kernel discriminant analysis
\citep{Ye-2008-MDKL}, to real-world applications, e.g., face
recognition \citep{Wright-2009-face}, saliency detection
\citep{Wu-2012-saliency}, and video denoising
\citep{Ji-2010-VideoDenois}. Most of the problems can be
(re)formulated as the following linearly constrained separable
convex program\footnote{If the objective function is not separable
or there are extra convex set constraints, $\x_i\in X_i$,
$i=1,\cdots,n$, where $X_i$'s are convex sets, the program can be
transformed into (\ref{eq:model_problem_multivar}) by introducing
auxiliary variables, c.f.
(\ref{eq:model_problem_multivar_convex_sets})-(\ref{eq:redefine_equiv}).}:
\begin{equation}
\min\limits_{\x_1,\cdots,\x_n} \sum\limits_{i=1}^n f_i(\x_i),\quad
s.t.\quad
\sum\limits_{i=1}^n\A_i(\x_i)=\mathbf{b},\label{eq:model_problem_multivar}
\end{equation}
where $\x_i$ and $\mathbf{b}$ could be either vectors or
matrices\footnote{In this paper we call each $\x_i$ a ``block" of
variables because it may consist of multiple scalar variables. We
will use bold capital letters if a block is known to be a
matrix.}, $f_i$ is a closed proper convex function, and
$\A_i:\mathbb{R}^{d_i}\rightarrow \mathbb{R}^{m}$ is a linear
mapping. Without loss of generality, we may assume that none of
the $\A_i$'s is a zero mapping, the solution to
$\sum\limits_{i=1}^n \A_i(\x_i)=\mathbf{b}$ is non-unique, and the
mapping $\A(\x_1,\cdots,\x_n)\equiv\sum\limits_{i=1}^n\A_i(\x_i)$
is onto\footnote{The last two assumptions are equivalent to that
the matrix
$\mathbf{A}\equiv(\mathbf{A}_1\,\,\cdots\,\,\mathbf{A}_n)$ is not
full column rank but full row rank, where $\mathbf{A}_i$ is the
matrix representation of $\A_i$.}.

\subsection{Exemplar Problems in Machine Learning}

In this subsection, we present some examples of machine learning
problems that can be formulated as the model problem
\eqref{eq:model_problem_multivar}.

\subsubsection{Latent Low-Rank Representation}
Low-Rank Representation (LRR)~\citep{Liu-2010-LRR,Liu-2012-LRR} is
a recently proposed technique for robust subspace clustering and
has been applied to many machine learning and computer vision
problems. However, LRR works well only when the number of samples
is more than the dimension of the samples, which may not be
satisfied when the data dimension is high. So Liu et
al.~\citep{Liu-2011-LLRR} proposed latent LRR to overcome this
difficulty. The mathematical model of latent LRR is as follows:
\begin{equation}
\min\limits_{\bz,\bl,\be}\|\bz\|_* + \|\bl\|_* + \mu\|\be\|_{1},
\quad s.t. \quad \bx = \bx\bz + \bl\bx+ \be,\label{eq:llrr}
\end{equation}
where $\bx$ is the data matrix, each column being a sample vector,
$\|\cdot\|_*$ is the nuclear norm~\citep{Fazel-2002-nuclear},
i.e., the sum of singular values, and $\|\cdot\|_1$ is the
$\ell_1$ norm~\citep{Candes-2011-RPCA}, i.e., the sum of absolute
values of all entries. Latent LRR is to decompose data into
principal feature $\bx\bz$ and salient feature $\bl\bx$, up to
sparse noise $\be$.

\subsubsection{Nonnegative Matrix Completion}
Nonnegative matrix completion (NMC)~\citep{Xu-2011-NMF} is a novel
technique for dimensionality reduction, text mining, collaborative
filtering, and clustering, etc. It can be formulated as:
\begin{equation}
\min\limits_{\bx,\e}\|\bx\|_* + \frac{1}{2\mu}\|\e\|^2,\quad s.t.
\quad \mathbf{b} = \mathcal{P}_{\Omega}(\bx) + \e, \ \bx \geq
0,\label{eq:nmc}
\end{equation}
where $\b$ is the observed data in the matrix $\bx$ contaminated
by noise $\e$, $\Omega$ is an index set, $\mathcal{P}_{\Omega}$ is
a linear mapping that selects those elements whose indices are in
$\Omega$, and $\|\cdot\|$ is the Frobenius norm. NMC is to recover
the nonnegative low-rank matrix $\bx$ from the observed noisy data
$\b$.

To see that the NMC problem can be reformulated as
(\ref{eq:model_problem_multivar}), we introduce an auxiliary
variable $\by$ and rewrite \eqref{eq:nmc} as
\begin{equation}
\min\limits_{\bx,\by,\e}\|\bx\|_* + \chi_{\geq 0}(\by) +
\frac{1}{2\mu}\|\e\|^2,\quad s.t. \quad
\begin{pmatrix}
\mathcal{P}_{\Omega}(\bx)\\
\bx
\end{pmatrix}
-
\begin{pmatrix}
\bzero\\
\by
\end{pmatrix}
+
\begin{pmatrix}
\e\\
\bzero
\end{pmatrix}
 =\begin{pmatrix}
\mathbf{b}\\
\bzero
\end{pmatrix},\label{eq:nmc-equiv}
\end{equation}
where $\chi_{\geq 0}(\by)=\left\{
\begin{array}{ll}
0, & \mbox{if } \by\geq 0,\\
+\infty, & \mbox{otherwise},
\end{array}
\right.$ is the characteristic function of the set of nonegative
matrices.

\subsubsection{Group Sparse Logistic Regression with
Overlap} 

Besides unsupervised learning models shown above, many supervised
machine learning problems can also be written in the form of
\eqref{eq:model_problem_multivar}. For example, using logistic
function as the loss function in the group LASSO with
overlap~\citep{Jacob-2009-Group-LASSO,Deng-2011-ADM}, one obtains
the following model:
\begin{equation}
\min\limits_{\w,b}\frac{1}{s}\sum\limits_{i=1}^{s}\log\left(1+\exp\left(-y_i(\w^T\x_i+b)\right)\right)+\mu\sum\limits_{j=1}^t\|\bs_j\w\|,\label{eq:logit}
\end{equation}
where $\x_i$ and $y_i$, $i=1,\cdots,s$, are the training data and
labels, respectively, and $\w$ and $b$ parameterize the linear
classifier. $\bs_j$, $j=1,\cdots,t$, are the selection matrices,
with only one 1 at each row and the rest entries are all zeros.
The groups of entries, $\bs_j\w$, $j=1,\cdots,t$, may overlap each
other. This model can also be considered as an extension of the
group sparse logistic regression
problem~\citep{Meier-2008-Group-Logistic} to the case of
overlapped groups.

Introducing $\bar{\w}=(\w^T,b)^T$, $\bar{\x}_i=(\x_i^T,1)^T$,
$\z=(\z_1^T,\z_2^T,\cdots,\z_t^T)^T$, and
$\bar{\bs}=(\bs,\mathbf{0})$, where
$\bs=(\bs_1^T,\cdots,\bs_t^T)^T$, \eqref{eq:logit} can be
rewritten as
\begin{equation}
\min\limits_{\bar{\w},\z}\frac{1}{s}\sum\limits_{i=1}^{s}\log\left(1+\exp\left(-y_i(\bar{\w}^T\bar{\x}_i)\right)\right)+\mu\sum\limits_{j=1}^t\|\z_j\|,
\quad s.t. \quad \z = \bar{\bs}\bar{\w},\label{eq:logit'}
\end{equation}
which is a special case of \eqref{eq:model_problem_multivar}.

\subsection{Related Work}

Although general theories on convex programs are fairly complete
nowadays, e.g., most of them can be solved by the interior point
method~\citep{Boyd-Convex-Optimization}, when faced with large
scale problems, which are typical in machine learning, the general
theory may not lead to efficient algorithms. For example, when
using CVX\footnote{Available at
\url{http://stanford.edu/~boyd/cvx}}, an interior point based
toolbox, to solve nuclear norm minimization problems (i.e., one of
the $f_i$'s is the nuclear norm of a matrix, e.g., \eqref{eq:llrr}
and \eqref{eq:nmc}), such as matrix
completion~\citep{Candes-2009-matrix}, robust principal component
analysis~\citep{Candes-2011-RPCA}, and low-rank
representation~\citep{Liu-2010-LRR,Liu-2012-LRR}, the complexity
of each iteration is $O(q^6)$, where $q\times q$ is the matrix
size. Such a complexity is unbearable for large scale computing.

To address the scalability issue, first order methods are often
preferred. The accelerated proximal gradient (APG)
algorithm~\citep{Beck2009,Toh-2009-APG} is popular due to its
guaranteed $O(K^{-2})$ convergence rate, where $K$ is the
iteration number. However, APG is basically for unconstrained
optimization. For constrained optimization, the constraints have
to be added to the objective function as penalties, resulting in
approximated solutions only. The alternating direction method
(ADM)\footnote{Also called the alternating direction method of
multipliers (ADMM) in some literatures, e.g.,
\citep{Boyd-2011-Distributed,zhang2011unified,deng2012global}.
}~\citep{Fortin-1983-ADM,Boyd-2011-Distributed,Lin09} has regained
a lot of attention recently and is also widely used. It is
especially suitable for separable convex programs like
(\ref{eq:model_problem_multivar}) because it fully utilizes the
separable structure of the objective function. Unlike APG, ADM can
solve (\ref{eq:model_problem_multivar}) exactly. Another first
order method is the split Bregman method
\citep{Goldstein-2009-SB,zhang2011unified}, which is closely
related to ADM~\citep{Esser-2009-SB} and is influential in image
processing.

An important reason that first order methods are popular for
solving large scale convex programs in machine learning is that
the convex functions $f_i$'s are often matrix or vector norms or
characteristic functions of convex sets, which enables the
following subproblems (called the proximal operation of
$f_i$~\citep{Rockafellar})
\begin{equation}
\prox_{f_i,\sigma}(\w)=\argmin\limits_{\x_i}
f_i(\x_i)+\dfrac{\sigma}{2}\|\x_i-\mathbf{w}\|^2 \label{eq:proxy}
\end{equation}
to have closed form solutions. For example, when $f_i$ is the
$\ell_1$ norm,
$\prox_{f_i,\sigma}(\w)=\mathcal{T}_{\sigma^{-1}}(\mathbf{w})$,
where
$\mathcal{T}_{\varepsilon}(x)=\sgn(x)\max(|x|-\varepsilon,0)$ is
the soft-thresholding operator~\citep{Goldstein-2009-SB}; when
$f_i$ is the nuclear norm, the optimal solution is:
$\prox_{f_i,\sigma}(\mathbf{W})=\mathbf{U}\mathcal{T}_{\sigma^{-1}}(\mathbf{\Sigma})\mathbf{V}^T$,
where $\mathbf{U}\mathbf{\Sigma}\mathbf{V}^T$ is the singular
value decomposition (SVD) of $\mathbf{W}$~\citep{Cai-2008-SVT};
and when $f_i$ is the characteristic function of the nonnegative
cone, the optimal solution is
$\prox_{f_i,\sigma}(\w)=\max(\mathbf{w},0)$. Since subproblems
like (\ref{eq:proxy}) have to be solved in each iteration when
using first order methods to solve separable convex programs, that
they have closed form solutions greatly facilitates the
optimization.

However, when applying ADM to solve
(\ref{eq:model_problem_multivar}) with non-unitary linear mappings
(i.e., $\A_i^{\dag}\A_i$ is not the identity mapping, where
$\A_i^{\dag}$ is the adjoint operator of $\A_i$), the resulting
subproblems may not have closed form solutions\footnote{Because
$\|\x_i-\mathbf{w}\|^2$ in (\ref{eq:proxy}) becomes
$\|\A_i(\x_i)-\mathbf{w}\|^2$, which cannot be reduced to
$\|\x_i-\tilde{\mathbf{w}}\|^2$.}, hence need to be solved
iteratively, making the optimization process awkward. Some work
\citep{Yang-2011-LADM,Lin-2011-LADMAP} has considered this issue
by linearizing the quadratic term $\|\A_i(\x_i)-\mathbf{w}\|^2$ in
the subproblems, hence such a variant of ADM is called the
linearized ADM (LADM). \citet{deng2012global} further propose the
generalized ADM that makes both ADM and LADM as its special cases
and prove its globally linear convergence by imposing strong
convexity on the objective function or full-rankness on some
linear operators.

Nonetheless, most of the existing theories on ADM and LADM are for
the \emph{two-block} case, i.e., $n=2$ in
(\ref{eq:model_problem_multivar})~\citep{Fortin-1983-ADM,Boyd-2011-Distributed,Lin-2011-LADMAP,deng2012global}.
The number of blocks is restricted to two because the proofs of
convergence for the two-block case are not applicable for the
multi-block case, i.e., $n>2$ in
(\ref{eq:model_problem_multivar}). Actually, a naive
generalization of ADM or LADM to the multi-block case may diverge
(see (\ref{eq:parallel_BP}) and \citep{Chen-2013-Divergent}).
Unfortunately, in practice multi-block convex programs often
occur, e.g., robust principal component analysis with dense
noise~\citep{Candes-2011-RPCA}, latent low-rank
representation~\citep{Liu-2011-LLRR} (see \eqref{eq:llrr}), and
when there are extra convex set constraints (see (\ref{eq:nmc})
and
(\ref{eq:model_problem_multivar_convex_sets})-(\ref{eq:model_problem_multivar_equiv})).
So it is desirable to design practical algorithms for the
multi-block case.

Recently \citet{He-11-LADM_Gauss} and \citet{Tao-14-ADM_Parallel}
considered the multi-block LADM and ADM, respectively. To
safeguard convergence, \citet{He-11-LADM_Gauss} proposed LADM with
Gaussian back substitution (LADMGB), which destroys the sparsity
or low-rankness of the iterates during iterations when dealing
with sparse representation and low-rank recovery problems, while
\citet{Tao-14-ADM_Parallel} proposed ADM with parallel splitting,
whose subproblems may not be easily solvable. Moreover, they all
developed their theories with the penalty parameter being fixed,
resulting in difficulty of tuning an optimal penalty parameter
that fits for different data and data sizes. This has been
identified as an important issue~\citep{deng2012global}.

\invisible{ The work in \citet{deng2012global} extended ADM for
the two variable case of our model and proved global linear
convergence under the assumptions that one of the two functions
are strict convex and has Lipschitz gradient, along with certain
rank constraints on the operators in the linear system. They also
mentioned that by introducing another new variable, the liner
convergence results can also be extended to multiple variable case
with the same strict convexity and Lipschitz gradient assumptions
on $f_i$. It should be noticed that the strict convexity,
Lipschitz gradient and rank constraints on operators $\A_i$ cannot
always be guaranteed in real problems, thus their results is not
suitable for our general convex program
(\ref{eq:model_problem_multivar}). }

\subsection{Contributions and Differences from Prior Work}

To propose an algorithm that is more suitable for convex programs
in machine learning, in this paper we aim at combining the
advantages of \citep{He-11-LADM_Gauss},
\citep{Tao-14-ADM_Parallel}, and \citep{Lin-2011-LADMAP}, i.e.,
combining LADM, parallel splitting, and adaptive penalty. Hence we
call our method LADM with parallel splitting and adaptive penalty
(LADMPSAP). With LADM, the subproblems will have forms like
(\ref{eq:proxy}) and hence can be easily solved. With parallel
splitting, the sparsity and low-rankness of iterates can be
preserved during iterations when dealing with sparse
representation and low-rank recovery problems, saving both the
storage and the computation load. With adaptive penalty, the
convergence can be faster and it is unnecessary to tune an optimal
penalty parameter. Parallel splitting also makes the algorithm
highly parallelizable, making LADMPSAP suitable for parallel or
distributed computing, which is important for large scale machine
learning. When all the component objective functions have bounded
subgradients, we prove convergence results that are stronger than
the existing theories on ADM and LADM. For example, the penalty
parameter can be \emph{unbounded} and the \emph{sufficient and
necessary} conditions of the global convergence of LADMPSAP can be
obtained as well. We also propose a simple optimality measure and
prove the convergence rate of LADMPSAP in an ergodic sense under
this measure. Our proof is simpler than those in
\citep{He-2012-Rate} and ~\citep{Tao-14-ADM_Parallel} which relied
on a complex optimality measure. When a convex program has extra
convex set constraints, we further devise a practical version of
LADMPSAP that converges faster thanks to better parameter
analysis. Finally, we generalize LADMPSAP to cope with more
difficult $f_i$'s, whose proximal operation \eqref{eq:proxy} is
not easily solvable, by further linearizing the smooth components
of $f_i$'s. Experiments testify to the advantage of LADMPSAP in
speed and numerical accuracy.

Note that \citet{Goldfarb-10-fast} also proposed a multiple
splitting algorithm for convex optimization. However, they only
considered a special case of our model problem
(\ref{eq:model_problem_multivar}), i.e., all the linear mappings
$\A_i$'s are identity mappings\footnote{The multi-block problems
introduced in \citep{Boyd-2011-Distributed} also fall within this
category.}. With their simpler model problem, linearization is
unnecessary and a faster convergence rate, $O(K^{-2})$, can be
achieved. In contrast, in this paper we aim at proposing a
practical algorithm for efficiently solving more general problems
like (\ref{eq:model_problem_multivar}).

We also note that \citet{Luo-2012-LinearADM} used the same
linearization technique for the smooth components of $f_i$'s as
well, but they only considered a special class of $f_i$'s. Namely,
the non-smooth component of $f_i$ is a sum of $\ell_1$ and
$\ell_2$ norms or its epigraph is polyhedral. Moreover, for
parallel splitting (Jacobi update) \citet{Luo-2012-LinearADM} has
to incorporate a postprocessing to guarantee convergence, by
interpolating between an intermediate iterate and the previous
iterate. Third, \citet{Luo-2012-LinearADM} still focused on a
fixed penalty parameter. Again, our method can handle more general
$f_i$'s, does not require postprocessing, and allows for an
adaptive penalty parameter.

A more general splitting/linearization technique can be founded in
\citep{zhang2011unified}. However, the authors only proved that
any accumulation point of the iteration is a Kuhn-Karush-Tucker
(KKT) point and did not investigate the convergence rate. There
was no evidence that the iteration could converge to a unique
point. Moreover, the authors only studied the case of fixed
penalty parameter.

Although dual ascent with dual
decomposition~\citep{Boyd-2011-Distributed} can also solve
(\ref{eq:model_problem_multivar}) in a parallel way, it may break
down when some $f_i$'s are not strictly
convex~\citep{Boyd-2011-Distributed}, which typically happens in
sparse or low-rank recovery problems where $\ell_1$ norm or
nuclear norm are used. Even if it works, since $f_i$ is not
strictly convex, dual ascent becomes dual \emph{subgradient}
ascent~\citep{Boyd-2011-Distributed}, which is known to converge
at a rate of $O(K^{-1/2})$ -- slower than our $O(K^{-1})$ rate.
Moreover, dual ascent requires choosing a good step size for each
iteration, which is less convenient than ADM based methods.

\subsection{Organization}
The remainder of this paper is organized as follows. We first
review LADM with adaptive penalty (LADMAP) for the two-block case
in Section~\ref{sec:LADMAP-2var}. Then we present LADMPSAP for the
multi-block case in Section~\ref{sec:LADMPSAP}. Next, we propose a
practical version of LADMPSAP for separable convex programs with
convex set constraints in Section~\ref{sec:PPLAMDAP}. We further
extend LADMPSAP to proximal LADMPSAP for programs with more
difficult objective functions in Section~\ref{sec:G-LADMPSAP}. We
compare the advantage of LADMPSAP in speed and numerical accuracy
with other first order methods in Section~\ref{sec:exp}. Finally,
we conclude the paper in Section~\ref{sec:con}.

\section{Review of LADMAP for the Two-Block Case}\label{sec:LADMAP-2var}

We first review LADMAP~\citep{Lin-2011-LADMAP} for the two-block
case of (\ref{eq:model_problem_multivar}). It consists of four
steps:
\begin{enumerate}
\item Update $\x_1$:
\begin{equation}
\x_1^{k+1}=\argmin\limits_{\x_1} f_1(\x_1)
+\dfrac{\sigma_1^{(k)}}{2}\lbar\x_1-\x_1^{k}+\A_1^{\dag}(\tilde{\bflambda}_1^{k})/\sigma_1^{(k)}\rbar^2,
\label{eq:update_x1}
\end{equation}
\item Update $\x_2$:
\begin{equation}
\x_2^{k+1}=\argmin\limits_{\x_2} f_2(\x_2)
+\dfrac{\sigma_2^{(k)}}{2}\lbar\x_2-\x_2^{k}+\A_2^{\dag}(\tilde{\bflambda}_2^{k})/\sigma_2^{(k)}\rbar^2,
\label{eq:update_x2}
\end{equation}
\item Update $\bflambda$:
\begin{equation}
\bflambda^{k+1} = \bflambda^k +
\beta_k\left(\sum\limits_{i=1}^2\A_i(\x_i^{k+1})-\mathbf{b}\right),
\label{eq:update_lambda}
\end{equation}
\item Update $\beta$:
\begin{equation}
\beta_{k+1} = \min(\beta_{\max},\rho\beta_k),
\label{eq:update_beta}
\end{equation}
\end{enumerate}
where $\lambda$ is the Lagrange multiplier, $\beta_k$ is the
penalty parameter, $\sigma_i^{(k)}=\eta_i\beta_k$ with $\eta_i
> \|\A_i\|^2$ ($\|\A_i\|$ is the operator norm of $\A_i$),
\begin{eqnarray}
\tilde{\bflambda}_1^{k}&=&\bflambda^k +
\beta_k(\A_1(\x_1^{k})+\A_2(\x_2^{k})-\mathbf{b}),\\
\tilde{\bflambda}_2^{k}&=&\bflambda^k +
\beta_k(\A_1(\x_1^{k+1})+\A_2(\x_2^{k})-\mathbf{b}),\label{eq:tilde_lambda12}
\end{eqnarray}
and $\rho$ is an adaptively updated parameter (see
\eqref{eq:update_rho}). Please refer to~\citep{Lin-2011-LADMAP}
for details. Note that the latest $\x_1^{k+1}$ is immediately used
to compute $\x_2^{k+1}$ (see~\eqref{eq:tilde_lambda12}). So $\x_1$
and $\x_2$ have to be updated alternately, hence the name
alternating direction method.

\invisible{
The iteration terminates when the following two conditions are
met:
\begin{equation}
\left\|\sum\limits_{i=1}^2\mathcal{A}_i(\mathbf{x}_i^{k+1}) -
\mathbf{b}\right\|/\|\mathbf{b}\| <
\varepsilon_1,\label{eq:stopping1}
\end{equation}
\begin{equation}
\max\left(\left\{\sqrt{\beta_{k}\sigma_i^{(k)}}\lbar\mathbf{x}_i^{k+1}-\mathbf{x}_i^k\rbar,i=1,2\right\}\right)/\|\mathbf{b}\|
< \varepsilon_2.\label{eq:stopping2}
\end{equation}}

\section{LADMPSAP for the Multi-Block Case}\label{sec:LADMPSAP}

In this section, we extend LADMAP for multi-block separable convex
programs (\ref{eq:model_problem_multivar}). We also provide the
\emph{sufficient and necessary conditions} for global convergence
when subgradients of the objective functions are all bounded. We
further prove the convergence rate in an ergodic sense.

\subsection{LADM with Parallel Splitting and Adaptive
Penalty}\label{sec:LADM+PSAP}

Contrary to our intuition, the multi-block case is actually
fundamentally different from the two-block one. For the
multi-block case, it is very natural to generalize LADMAP for the
two-block case in a straightforward way, with
\begin{equation}
\tilde{\bflambda}_i^{k}=\bflambda^k +
\beta_k\left(\sum\limits_{j=1}^{i-1}\A_j(\x_j^{k+1})
+\sum\limits_{j=i}^{n}\A_j(\x_j^{k})-\mathbf{b}\right),\ i=1,\cdots,n.\label{eq:tilde_lambda_i}
\end{equation}
Unfortunately, we were unable to prove the convergence of such a
naive LADMAP using the same proof for the two-block case. This is
because their Fej\'{e}r monotone inequalities (see
Remark~\ref{rem:different}) cannot be the same. That is why He et
al. has to introduce an extra Gaussian back
substitution~\citep{He-2011-Gaussian,He-11-LADM_Gauss} for
correcting the iterates. Actually, the above naive generalization
of LADMAP may be divergent (which is even worse than converging to
a wrong solution), e.g., when applied to the following problem:
\begin{equation}
\min\limits_{\x_1,\cdots,\x_n} \sum\limits_{i=1}^n
\|\x_i\|_1,\quad s.t.\quad \sum\limits_{i=1}^n
\mathbf{A}_i\x_i=\mathbf{b},\label{eq:parallel_BP}
\end{equation}
where $n\geq 5$ and $\mathbf{A}_i$ and $\b$ are Gaussian random
matrix and vector, respectively, whose entries fulfil the standard
Gaussian distribution independently. \citet{Chen-2013-Divergent}
also analyzed the naively generalized ADM for the multi-block case
and showed that even for three blocks the iteration could still be
divergent. They also provided sufficient conditions, which
basically require that the linear mappings $\A_i$ should be
orthogonal to each other ($\A_i^{\dag}\A_j=0$, $i\neq j$), to
ensure the convergence of naive ADM.

Fortunately, by modifying $\tilde{\bflambda}_i^k$ slightly we are
able to prove the convergence of the corresponding algorithm. More
specifically, our algorithm for solving
(\ref{eq:model_problem_multivar}) consists of the following steps:
\begin{enumerate}
\item Update $\x_i$'s in parallel:\\
\begin{equation}
\x_i^{k+1}=\argmin\limits_{\x_i} f_i(\x_i)
+\dfrac{\sigma_i^{(k)}}{2}\lbar\x_i-\x_i^{k}+\A_i^{\dag}(\hat{\bflambda}^{k})/\sigma_i^{(k)}\rbar^2,\quad
i=1,\cdots,n, \label{eq:update_xi}
\end{equation}
\item Update $\bflambda$:\\
\begin{equation}
\bflambda^{k+1} = \bflambda^k +
\beta_k\left(\sum\limits_{i=1}^n\A_i(\x_i^{k+1})-\b\right),
\label{eq:update_lambda_multivar}
\end{equation}
\item Update $\beta$:\\
\begin{equation}
\beta_{k+1} = \min(\beta_{\max},\rho\beta_k),
\label{eq:update_beta_multivar}
\end{equation}
\end{enumerate}
where $\sigma_i^{(k)}=\eta_i\beta_k$,
\begin{equation}
\hat{\bflambda}^{k}=\bflambda^k +
\beta_k\left(\sum\limits_{i=1}^n\A_i(\x_i^{k})-\b\right),\label{eq:hat_lambda}
\end{equation}
and
\begin{equation}
\rho = \left\{\begin{array}{ll} \rho_0, & \mbox{if} \
\beta_{k}\max\left(\left\{\sqrt{\eta_i}\lbar\mathbf{x}_i^{k+1}-\mathbf{x}_i^k\rbar,i=1,\cdots,n\right\}\right)/\lbar\mathbf{b}\rbar < \varepsilon_2,\\
1, & \mbox{otherwise},
\end{array}\right.\label{eq:update_rho}
\end{equation}
with $\rho_0 > 1$ being a constant and $0 < \varepsilon_2\ll 1$
being a threshold. Indeed, we replace $\tilde{\bflambda}_i^k$ with
$\hat{\bflambda}^k$ as (\ref{eq:hat_lambda}), which is independent
of $i$, and the rest procedures of the algorithm, including the
scheme \eqref{eq:update_beta_multivar} and \eqref{eq:update_rho}
to update the penalty parameter, are all inherited
from~\citep{Lin-2011-LADMAP}, except that $\eta_i$'s have to be
made larger (see Theorem~\ref{thm:converge_multivar}). As now
$\x_i$'s are updated in parallel and $\beta_k$ changes adaptively,
we call the new algorithm LADM with \emph{parallel splitting} and
\emph{adaptive penalty} (LADMPSAP).

\subsection{Stopping Criteria}

Some existing work (e.g., \citep{Liu-2010-LRR,Favaro-2011-Closed})
proposed stopping criteria out of intuition only, which may not
guarantee that the correct solution is approached. Recently,
\citet{Lin09} and \citet{Boyd-2011-Distributed} suggested that the
stopping criteria can be derived from the KKT conditions of a
problem. Here we also adopt such a strategy. Specifically, the
iteration terminates when the following two conditions are met:
\begin{equation}
\left\|\sum\limits_{i=1}^n\mathcal{A}_i(\mathbf{x}_i^{k+1}) -
\mathbf{b}\right\|/\|\mathbf{b}\| <
\varepsilon_1,\label{eq:stopping1}
\end{equation}
\begin{equation}
\beta_{k}\max\left(\left\{\sqrt{\eta_i}\lbar\mathbf{x}_i^{k+1}-\mathbf{x}_i^k\rbar,i=1,\cdots,n\right\}\right)/\|\mathbf{b}\|
< \varepsilon_2.\label{eq:stopping2}
\end{equation}
The first condition measures the feasibility error. The second
condition is derived by comparing the KKT conditions of problem
\eqref{eq:model_problem_multivar} and the optimality condition of
subproblem~\eqref{eq:LADMPSAP_update_xi}. The rules
(\ref{eq:update_beta_multivar}) and (\ref{eq:update_rho}) for
updating $\beta$ are actually hinted by the above stopping
criteria such that the two errors are well balanced.

For better reference, we summarize the proposed LADMPSAP algorithm
in Algorithm~\ref{alg:LADMPSAP-multivar}. For fast convergence, we
suggest that $\beta_0=\alpha m\varepsilon_2$ and $\alpha>0$ and
$\rho_0 > 1$ should be chosen such that $\beta_k$ increases
steadily along with iterations.

\begin{algorithm}[tb]
   \caption{LADMPSAP for Solving (\ref{eq:model_problem_multivar})}
   \label{alg:LADMPSAP-multivar}
\begin{algorithmic}
   \STATE {\bfseries Initialize:}
     Set $\rho_0 > 1$, $\varepsilon_1>0$, $\varepsilon_2 > 0$, $\beta_{\max} \gg 1 \gg \beta_0 > 0$, $\bflambda^0$,
   $\eta_i > n\|\mathcal{A}_i\|^2$, $\mathbf{x}_i^0$, $i=1,\cdots,n$.
   \WHILE {(\ref{eq:stopping1}) or (\ref{eq:stopping2}) is not satisfied}
   \STATE {\bfseries Step 1:} Compute $\hat{\bflambda}^k$ as (\ref{eq:hat_lambda}).
   \STATE {\bfseries Step 2:} Update $\mathbf{x}_i$'s in parallel by solving
   \begin{equation}
   \x_i^{k+1}=\argmin\limits_{\x_i} f_i(\x_i)+\dfrac{\eta_i\beta_k}{2}\left\|\x_i-\x_i^k+\A_i^{\dag}(\hat{\bflambda}^k)/(\eta_i\beta_k)\right\|^2,\ i=1,\cdots,n.\label{eq:LADMPSAP_update_xi}
   \end{equation}
   \STATE {\bfseries Step 3:} Update $\bflambda$ by (\ref{eq:update_lambda_multivar}) and $\beta$ by (\ref{eq:update_beta_multivar}) and (\ref{eq:update_rho}).
   \ENDWHILE
\end{algorithmic}
\end{algorithm}

\subsection{Global Convergence}
In the following, we always use
$(\x_1^*,\cdots,\x_n^*,\bflambda^*)$ to denote the KKT point of
problem (\ref{eq:model_problem_multivar}). For the global
convergence of LADMPSAP, we have the following theorem, where we
denote $\{\mathbf{x}_i^k\}=\{\mathbf{x}_1^k,\cdots,\x_n^k\}$ for
simplicity.
\begin{theorem}{\bf(Convergence of LADMPSAP)}\footnote{Please see Appendix for all the proofs of our theoretical results hereafter.}
If $\{\beta_k\}$ is non-decreasing and upper bounded, $\eta_i >
n\|\mathcal{A}_i\|^2$, $i=1,\cdots,n$, then
$\{(\{\mathbf{x}_i^k\},\bflambda^k)\}$ generated by LADMPSAP
converge to a KKT point of problem
(\ref{eq:model_problem_multivar}).\label{thm:converge_multivar}
\end{theorem}

\subsection{Enhanced Convergence Results}

Theorem~\ref{thm:converge_multivar} is a convergence result for
general convex programs~(\ref{eq:model_problem_multivar}), where
$f_i$'s are general convex functions and hence $\{\beta_k\}$ needs
to be bounded. Actually, almost all the existing theories on ADM
and LADM even assumed a fixed $\beta$. For adaptive $\beta_k$, it
will be more convenient if a user needs not to specify an upper
bound on $\{\beta_k\}$ because imposing a large upper bound
essentially equals to allowing $\{\beta_k\}$ to be unbounded.
Since many machine learning problems choose $f_i$'s as
matrix/vector norms, which result in bounded subgradients, we find
that the boundedness assumption can be removed. Moreover, we can
further prove the \emph{sufficient and necessary} condition for
global convergence.

\begin{theorem}\label{thm:convergence_unbounded}{\bf (Sufficient Condition for Global Convergence)}
If $\{\beta_k\}$ is non-decreasing and
$\sum\limits_{k=1}^{+\infty}\beta_k^{-1}=+\infty$, $\eta_i >
n\|\A_i\|^2$, $\partial f_i(\x)$ is bounded, $i=1,\cdots,n$, then
the sequence $\{\x_i^k\}$ generated by LADMPSAP converges to an
optimal solution to (\ref{eq:model_problem_multivar}).
\end{theorem}

\begin{remark}\label{rem:convergence_of_lambda}
Theorem~\ref{thm:convergence_unbounded} does not claim that
$\{\bflambda^k\}$ converges to a point $\bflambda^\infty$.
However, as we are more interested in $\{\x_i^k\}$, such a
weakening is harmless.
\end{remark}
We also have the following result on the necessity of
$\sum\limits_{k=1}^{+\infty}\beta_k^{-1}=+\infty$.

\begin{theorem}\label{thm:convergence_unbounded_necessary}{\bf (Necessary Condition for Global Convergence)}
If $\{\beta_k\}$ is non-decreasing, $\eta_i > n\|\A_i\|^2$,
$\partial f_i(\x)$ is bounded, $i=1,\cdots,n$, then
$\sum\limits_{k=1}^{+\infty}\beta_k^{-1}=+\infty$ is also a
necessary condition for the global convergence of $\{\x_i^k\}$
generated by LADMPSAP to an optimal solution to
(\ref{eq:model_problem_multivar}).
\end{theorem}
With the above analysis, when all the subgradients of the
component objective functions are bounded we can remove
$\beta_{\max}$ in Algorithm~\ref{alg:LADMPSAP-multivar}.

\subsection{Convergence Rate}
The convergence rate of ADM and LADM in the traditional sense is
an open problem~\citep{Goldfarb-10-fast}. Although
\cite{Luo-2012-LinearADM} claimed that they proved the linear
convergence rate of ADM, their assumptions are actually quite
strong. They assumed that the non-smooth part of $f_i$ is a sum of
$\ell_1$ and $\ell_2$ norms or its epigraph is polyhedral.
Moreover, the convex constraint sets should all be polyhedral and
bounded. So although their results are encouraging, for general
convex programs the convergence rate is still a mystery. Recently,
\cite{He-2012-Rate} and \cite{Tao-14-ADM_Parallel} proved an
$O(1/K)$ convergence rate of ADM and ADM with parallel splitting
in an ergodic sense, respectively. Namely
$\frac{1}{K}\sum\limits_{k=1}^{K}\x_i$ violates an optimality
measure in $O(1/K)$. Their proof is lengthy and is for fixed
penalty parameter only.

In this subsection, based on a simple optimality measure we give a
simple proof for the convergence rate of LADMPSAP. For simplicity,
we denote $\x=(\x_1^T,\cdots,\x_n^T)^T$,
$\x^*=((\x_1^*)^T,\cdots,(\x_2^*)^T)^T$, and
$f(\x)=\sum\limits_{i=1}^n f_i(\x_i)$. We first have the following
proposition.
\begin{proposition}\label{prop:optimality}
$\tilde{\x}$ is an optimal solution to
(\ref{eq:model_problem_multivar}) if and only if there exists
$\alpha > 0$, such that
\begin{equation}
f(\tilde{\x})-f(\x^*)+\sum\limits_{i=1}^n\<\A_i^{\dag}(\bflambda^*),\tilde{\x}_i-\x_i^*\>+\alpha\lbar\sum\limits_{i=1}^n
\A_i(\tilde{\x}_i)-\mathbf{b}\rbar^2
=0.\label{eq:constrained_optimality}
\end{equation}
\end{proposition}

Since the left hand side of (\ref{eq:constrained_optimality}) is
always nonnegative and it becomes zero only when $\tilde{\x}$ is
an optimal solution, we may use its magnitude to measure how far a
point $\tilde{\x}$ is from an optimal solution. Note that in the
unconstrained case, as in APG~\citep{Beck2009}, one may simply use
$f(\tilde{\x})-f(\x^*)$ to measure the optimality. But here we
have to deal with the constraints. Our criterion is simpler than
that in \citep{He-2012-Rate,Tao-14-ADM_Parallel}, which has to
compare $(\{\x_i^k\},\lambda^k)$ with all
$(\x_1,\cdots,\x_n,\bflambda) \in \mathbb{R}^{d_1}\times\cdots
\times\mathbb{R}^{d_n}\times \mathbb{R}^{m}$.

Then we have the following convergence rate theorem for LADMPSAP
in an ergodic sense.
\begin{theorem}\label{thm:convergence_rate_2var}{\bf (Convergence Rate of LADMPSAP)}
Define $\bar{\x}^K=\sum\limits_{k=0}^K \gamma_k \x^{k+1}$, where
$\gamma_k=\beta_k^{-1}/\sum\limits_{j=0}^K \beta_j^{-1}$. Then the following inequality holds for $\bar{\x}^K$:
\begin{eqnarray}
\begin{array}{rl}
&f(\bar{\x}^K) - f(\x^*) + \sum\limits_{i=1}^n
\<\A_i^{\dag}(\bflambda^*),\bar{\x}_i^K-\x_i^*\>
+\dfrac{\alpha\beta_0}{2}\lbar \sum\limits_{i=1}^n
\A_i(\bar{\x}_i^K) - \mathbf{b}\rbar^2 \\
\leq & C_0/\left(2\sum\limits_{k=0}^K\beta_k^{-1}\right),
\end{array}\label{eq:convergence_rate_2var}
\end{eqnarray}
where $\alpha^{-1}=(n+1) \max\Bigg(1,
\Bigg\{\dfrac{\|\A_i\|^2}{\eta_i-n\|\A_i\|^2}, i=1,\cdots,$
$n\Bigg\}\Bigg)$ and
$C_0=\sum\limits_{i=1}^n\eta_i\lbar\x_i^{0}-\x_i^*\rbar^2$
$+\beta_0^{-2}\lbar\bflambda^{0}-\bflambda^*\rbar^2$.
\end{theorem}

Theorem~\ref{thm:convergence_rate_2var} means that $\bar{\x}^K$ is
by $O\left(1/\sum\limits_{k=0}^K\beta_k^{-1}\right)$ from being an
optimal solution. This theorem holds for both bounded and
unbounded $\{\beta_k\}$. In the bounded case,
$O\left(1/\sum\limits_{k=0}^K\beta_k^{-1}\right)$ is simply
$O(1/K)$. Theorem~\ref{thm:convergence_rate_2var} also hints that
$\sum\limits_{k=0}^K\beta_k^{-1}$ should approach infinity to
guarantee the convergence of LADMPSAP, which is consistent with
Theorem~\ref{thm:convergence_unbounded_necessary}.

\section{Practical LADMPSAP for Convex Programs with Convex Set Constraints}\label{sec:PPLAMDAP}
In real applications, we are often faced with convex programs with
convex set constraints:
\begin{equation}
\min\limits_{\x_1,\cdots,\x_n} \sum\limits_{i=1}^n f_i(\x_i), \
s.t. \ \sum\limits_{i=1}^n\A_i(\x_i)=\mathbf{b}, \ \x_i\in X_i,
i=1,\cdots,n,\label{eq:model_problem_multivar_convex_sets}
\end{equation}
where $X_i\subseteq\mathbb{R}^{d_i}$ is a closed convex set. In
this section, we consider to extend LADMPSAP to solve the more
complex convex set constraint model
\eqref{eq:model_problem_multivar_convex_sets}. We assume that the
projections onto $X_i$'s are all easily computable. For many
convex sets used in machine learning, such an assumption is valid,
e.g., when $X_i$'s are nonnegative cones or positive semi-definite
cones. In the following, we discuss how to solve
(\ref{eq:model_problem_multivar_convex_sets}) efficiently. For
simplicity, we assume $X_i\neq \mathbb{R}^{d_i}$, $\forall i$.
Finally, we assume that $\db$ is an interior point of
$\sum\limits_{i=1}^n\A_i(X_i)$.

We introduce auxiliary variables $\x_{n+i}$ to convert $\x_i\in
X_i$ into $\x_i=\x_{n+i}$ and $\x_{n+i}\in X_i$, $i=1,\cdots,n$.
Then (\ref{eq:model_problem_multivar_convex_sets}) can be
reformulated as:
\begin{equation}
\min\limits_{\x_1,\cdots,\x_{2n}} \sum\limits_{i=1}^{2n}
f_i(\x_i),\quad s.t.\quad
\sum\limits_{i=1}^{2n}\hat{\A}_i(\x_i)=\hat{\mathbf{b}},\label{eq:model_problem_multivar_equiv}
\end{equation}
where $$f_{n+i}(\x)\equiv \chi_{X_i}(\x)=\left\{
\begin{array}{ll}
0,& \mbox{if }\x\in X_i,\\
+\infty, &\mbox{otherwise},
\end{array}
\right.$$ is the characteristic function of $X_i$,
\begin{eqnarray}
\begin{array}{c}
\hat{\A}_i(\x_i)= \left(
\begin{array}{c}
\A_i(\x_i)\\
0\\
\vdots\\
\x_i\\
\vdots\\
0
\end{array}
\right), \hat{\A}_{n+i}(\x_{n+i})=\left(
\begin{array}{c}
0\\
0\\
\vdots\\
-\x_{n+i}\\
\vdots\\
0
\end{array}
\right), \mbox{ and } \hat{\b}=\left(
\begin{array}{c}
\b\\
0\\
\vdots\\
0\\
\vdots\\
0
\end{array}
\right),
\end{array}\label{eq:redefine_equiv}
\end{eqnarray}
where $i=1,\cdots,n$.

\invisible{
\begin{equation}
\begin{array}{c}
\hat{\A}_i(\x_i) = \left( \A_i(\x_i), 0,\cdots,\x_i,\cdots,0
\right)^T,
\hat{\A}_{n+i}(\x_{n+i}) =\left( 0,0,\cdots,-\x_{n+i},\cdots,0
\right)^T,
i=1,\cdots,n, \\
\hat{\mathbf{b}}=\left(\mathbf{b},0,\cdots,0,\cdots,0 \right)^T.
\end{array}\label{eq:redefine_equiv}
\end{equation}
}

The adjoint operator $\hat{\A}_i^{\dag}$ is
\begin{equation}
\hat{\A}_i^{\dag}(\y)=\A_i^{\dag}(\y_{1})+\y_{i+1},\
\hat{\A}_{n+i}^{\dag}(\y)=-\y_{i+1}, \
i=1,\cdots,n,\label{eq:adjoint_hat_A}
\end{equation}
where $\y_{i}$ is the $i$-th sub-vector of $\y$, partitioned
according to the sizes of $\mathbf{b}$ and $\x_i$, $i=1,\cdots,n$.

Then LADMPSAP can be applied to solve problem
(\ref{eq:model_problem_multivar_equiv}). The Lagrange
multiplier $\bflambda$ and the auxiliary multiplier
$\hat{\bflambda}$ are respectively updated as
\begin{eqnarray}
&\bflambda^{k+1}_{1}=\bflambda^{k}_{1} + \beta_k\left(\sum\limits_{i=1}^n \A_i(\x_i^{k+1})-\mathbf{b}\right),\
\bflambda^{k+1}_{i+1}=\bflambda^{k}_{i+1} +
\beta_k(\x_i^{k+1}-\x_{n+i}^{k+1}), \ \label{eq:update_lambda_multivar-equiv}\\
&\hat{\bflambda}^{k}_{1}=\bflambda^{k}_{1} + \beta_k\left(\sum\limits_{i=1}^n \A_i(\x_i^{k})-\mathbf{b}\right),\
\hat{\bflambda}^{k}_{i+1}=\bflambda^{k}_{i+1} +
\beta_k(\x_i^{k}-\x_{n+i}^{k}), \label{eq:update_hatlambda_multivar-equiv}
\end{eqnarray}
and $\x_i$ is updated as (see (\ref{eq:update_xi}))
\begin{eqnarray}
\x_i^{k+1}&=&\argmin\limits_{\x} f_i(\x)+\dfrac{\eta_i\beta_k}{2}\left\|\x-\x_i^k +[\A_i^{\dag}(\hat{\bflambda}^k_{1})+\hat{\bflambda}^k_{i+1}]/(\eta_i\beta_k)\right\|^2,\ \label{eq:LADMPSAP_update_xi1}\\
\x_{n+i}^{k+1}&=&\argmin\limits_{\x\in X_i} \dfrac{\eta_{n+i}\beta_k}{2}\left\|\x-\x_{n+i}^k-\hat{\bflambda}^k_{i+1}/(\eta_{n+i}\beta_k)\right\|^2\nonumber\\
&=&\pi_{X_i}
\left(\x_{n+i}^k+\hat{\bflambda}^k_{i+1}/(\eta_{n+i}\beta_k)\right),\label{eq:LADMPSAP_update_xi2}
\end{eqnarray}
where $\pi_{X_i}$ is the projection onto $X_i$ and $i=1,\cdots,n$.

As for the choice of $\eta_i$'s, although we can simply apply
Theorem~\ref{thm:converge_multivar} to assign their values as
$\eta_i
> 2n (\|\A_i\|^2+1)$ and $\eta_{n+i} > 2n$, $i=1,\cdots,n$, such
choices are too pessimistic. As $\eta_i$'s are related to the
magnitudes of the differences in $\x_i^{k+1}$ from $\x_i^k$, we
had better provide tighter estimate on $\eta_i$'s in order to
achieve faster convergence. Actually, we have the following better
result.
\begin{theorem}\label{thm:better_eta}
For problem (\ref{eq:model_problem_multivar_equiv}), if
$\{\beta_k\}$ is non-decreasing and upper bounded and $\eta_i$'s
are chosen as $\eta_i > n \|\A_i\|^2 + 2$ and $\eta_{n+i} > 2$,
$i=1,\cdots,n$, then the sequence
$\{(\{\mathbf{x}_i^k\},\bflambda^k)\}$ generated by LADMPSAP
converge to a KKT point of problem
(\ref{eq:model_problem_multivar_equiv}).
\end{theorem}

Finally, we summarize LADMPSAP for problem
(\ref{eq:model_problem_multivar_equiv}) in
Algorithm~\ref{alg:LADMPSAP-multivar_equiv}, which is a practical
algorithm for solving
(\ref{eq:model_problem_multivar_convex_sets}).

\begin{algorithm}[tb]
   \caption{LADMPSAP for (\ref{eq:model_problem_multivar_equiv}), also a Practical Algorithm for (\ref{eq:model_problem_multivar_convex_sets}).}
   \label{alg:LADMPSAP-multivar_equiv}
\begin{algorithmic}
   \STATE {\bfseries Initialize:}
     Set $\rho_0>1$, $\varepsilon_1>0$, $\varepsilon_2 > 0$, $\beta_{\max} \gg 1 \gg \beta_0 > 0$, $\bflambda^0=((\bflambda_{1}^0)^T,\cdots,(\bflambda_{n+1}^0)^T)^T$,
   $\eta_i > n\|\mathcal{A}_i\|^2+2$, $\eta_{n+i} > 2$, $\mathbf{x}_i^0$, $\mathbf{x}_{n+i}^0=\mathbf{x}_{i}^0$, $i=1,\cdots,n$.
   \WHILE {(\ref{eq:stopping1}) or (\ref{eq:stopping2}) is not satisfied}
   \STATE {\bfseries Step 1:} Compute $\hat{\bflambda}^k$ as (\ref{eq:update_hatlambda_multivar-equiv}).
   \STATE {\bfseries Step 2:} Update $\mathbf{x}_i$, $i=1,\cdots,2n$, in parallel as (\ref{eq:LADMPSAP_update_xi1})-(\ref{eq:LADMPSAP_update_xi2}).
   \STATE {\bfseries Step 3:} Update $\bflambda$ by (\ref{eq:update_lambda_multivar-equiv}) and $\beta$ by (\ref{eq:update_beta_multivar}) and (\ref{eq:update_rho}).
   \ENDWHILE
   \STATE (Note that in (\ref{eq:update_rho}), (\ref{eq:stopping1}), and (\ref{eq:stopping2}), $n$ and $\mathcal{A}_i$ should be replaced by $2n$ and $\hat{\mathcal{A}}_i$, respectively.)
\end{algorithmic}
\end{algorithm}

\begin{remark}
Analogs of Theorems~\ref{thm:convergence_unbounded} and
\ref{thm:convergence_unbounded_necessary} are also true for
Algorithm~\ref{alg:LADMPSAP-multivar_equiv} although $\partial
f_{n+i}$'s are unbounded, thanks to our assumptions that all
$\partial f_i$, $i=1,\cdots,n$, are bounded and $\db$ is an
interior point of $\sum\limits_{i=1}^n\A_i(X_i)$, which result in
an analog of Proposition~\ref{prop:unbounded_beta}. Consequently,
$\beta_{\max}$ can also be removed if all $\partial f_i$,
$i=1,\cdots,n$, are bounded.
\end{remark}

\begin{remark}
Since Algorithm~\ref{alg:LADMPSAP-multivar_equiv} is an
application of Algorithm~\ref{alg:LADMPSAP-multivar} to problem
(\ref{eq:model_problem_multivar_equiv}), only with refined
parameter estimation, its convergence rate in an ergodic sense is
also $O\left(1/\sum\limits_{k=0}^K \beta_k^{-1}\right)$, where $K$
is the number of iterations.
\end{remark}

\invisible{ Finally, we summarize LADMPSAP for problem
(\ref{eq:model_problem_multivar_equiv}) in
Algorithm~\ref{alg:LADMPSAP-multivar_equiv}.
Algorithm~\ref{alg:LADMPSAP-multivar_equiv} is the practical
algorithm to solve (\ref{eq:model_problem_multivar_convex_sets}).

\begin{algorithm}[tb]
   \caption{LADMPSAP for Problem (\ref{eq:model_problem_multivar_equiv}). It Is Also the Practical Algorithm for Solving Problem (\ref{eq:model_problem_multivar_convex_sets}).}
   \label{alg:LADMPSAP-multivar_equiv}
\begin{algorithmic}
   \STATE {\bfseries Initialize:}
     Set $\varepsilon_1>0$, $\varepsilon_2 > 0$, $\beta_{\max} \gg 1 \gg \beta_0 > 0$,
   $\eta_i > n\|\mathcal{A}_i\|^2+2$, $\eta_{n+i} > 2$, $\mathbf{x}_i^0$, $\mathbf{x}_{n+i}^0=\mathbf{x}_{i}^0$, $i=1,\cdots,n$, $\bflambda^0=((\bflambda_{1}^0)^T,\cdots,(\bflambda_{n+1}^0)^T)^T$, and $k\leftarrow 0$.
   \WHILE {a multi-block version of (\ref{eq:stopping1}) or (\ref{eq:stopping2}) is not satisfied}
   \STATE {\bfseries Step 1:} Compute $\hat{\bflambda}^k$ as (\ref{eq:update_hatlambda_multivar-equiv}).
   \STATE {\bfseries Step 2:} Update $\mathbf{x}_i$, $i=1,\cdots,2n$, in parallel as (\ref{eq:LADMPSAP_update_xi1})-(\ref{eq:LADMPSAP_update_xi2}).
   \STATE {\bfseries Step 3:} Update $\bflambda$ by a multi-block version of (\ref{eq:update_lambda_multivar-equiv}).
   \STATE {\bfseries Step 4:} Update $\beta$ by a multi-block version of (\ref{eq:update_beta}) and (\ref{eq:update_rho}).
   \STATE {\bfseries Step 5:} $k\leftarrow k+1$.
   \ENDWHILE
\end{algorithmic}
\end{algorithm}
}

\section{Proximal LADMPSAP for Even More General Convex Programs}\label{sec:G-LADMPSAP}

In LADMPSAP we have assumed that the subproblems
(\ref{eq:update_xi}) are easily solvable. In many machine learning
problems, the functions $f_i$'s are often matrix or vector norms
or characteristic functions of convex sets. So this assumption
often holds. Nonetheless, this assumption is not always true,
e.g., when $f_i$ is the logistic loss function
(see~\eqref{eq:logit'}). So in this section we aim at generalizing
LADMPSAP to solve even more general convex programs
\eqref{eq:model_problem_multivar}.

We are interested in the case that $f_i$ can be decomposed into
two components:
\begin{equation}\label{eq:decompose_f}
f_i(\x_i)=g_i(\x_i)+h_i(\x_i),
\end{equation}
where both $g_i$ and $h_i$ are convex, $g_i$ is $C^{1,1}$:
\begin{equation}
\lbar \nabla g_i(\x)- \nabla g_i(\y)\rbar \leq L_i \lbar
\x-\y\rbar,\quad \forall x,y\in\mathbb{R}^{d_i},
\label{eq:Lipschitz-continue}
\end{equation}
and $h_i$ may not be differentiable but its proximal operation is
easily solvable. For brevity, we call $L_i$ the Lipschitz constant
of $\nabla g_i$.

Recall that in each iteration of LADMPSAP, we have to solve
subproblem \eqref{eq:update_xi}. Since now we do not assume that
the proximal operation of $f_i$ \eqref{eq:proxy} is easily
solvable, we may have difficulty in solving subproblem
\eqref{eq:update_xi}. By \eqref{eq:decompose_f}, we write down
\eqref{eq:update_xi} as
\begin{equation}
\x_i^{k+1}=\argmin\limits_{\x_i} h_i(\x_i)+g_i(\x_i)
+\dfrac{\sigma_i^{(k)}}{2}\lbar\x_i-\x_i^{k}+\A_i^\dag(\hat{\bflambda}^{k})/\sigma_i^{(k)}\rbar^2,\quad
i=1,\cdots,n, \label{eq:update_xi_new}
\end{equation}
Since
$g_i(\x_i)+\dfrac{\sigma_i^{(k)}}{2}\lbar\x_i-\x_i^{k}+\A_i^\dag(\hat{\bflambda}^{k})/\sigma_i^{(k)}\rbar^2$
is $C^{1,1}$, we may also linearize it at $\x_i^k$ and add a
proximal term. Such an idea leads to the following updating scheme
of ${\x}_i$:
\begin{eqnarray}
\begin{array}{rcl}
\x_i^{k+1}&=&\argmin \limits_{\x_i} h_i(\x_i)+g_i(\x_i^k)+\dfrac{\sigma_i^{(k)}}{2}\lbar\A_i^\dag(\hat{\bflambda}^{k})/\sigma_i^{(k)}\rbar^2\\
             &&\quad +\langle \nabla g_i(\x_i^k)+\A_i^\dag(\hat{\bflambda}^{k}),\x_i-\x_i^k \rangle+\dfrac{\tau_i^{(k)}}{2}\lbar\x_i-\x_i^{k}\rbar^2\\
          &=&\argmin \limits_{\x_i} h_i(\x_i)+\frac{\tau_i^{(k)}}{2} \lbar \x_i-\x_i^k+\frac{1}{\tau_i^{(k)}}[\A_i^{\dag}(\hat \bflambda^k)+\nabla g_i(\x_i^k)] \rbar^2,
\end{array}\label{eq:update_xi_proximal}
\end{eqnarray}
where $i=1,\cdots,n$. The choice of $\tau_i^{(k)}$ is presented in
Theorem~\ref{thm:general_convergence_unbounded}, i.e.
$\tau_i^{(k)}=T_i+\beta_k \eta_i$, where $T_i \geq L_i$ and
$\eta_i
> n\|\A_i\|^2$ are both positive constants.

By our assumption on $h_i$, the above subproblems are easily
solvable. The update of Lagrange multiplier $\lambda$ and $\beta$
are still respectively goes as \eqref{eq:update_lambda_multivar}
and \eqref{eq:update_beta_multivar} but with
\begin{equation}
\rho = \left\{\begin{array}{ll} \rho_0, & \mbox{if} \
\max\left(\left\{\|\A_i\|^{-1}\lbar \nabla
g_i(\mathbf{x}_i^{k+1})-\nabla g_i(\mathbf{x}_i^k)-
\tau_{i}^{(k)}(\mathbf{x}_i^{k+1}-\mathbf{x}_i^k)\rbar,\right.\right.\\
&\quad i=1,\cdots,n\Big\}\Big)/\|\mathbf{b}\|
< \varepsilon_2,\\
1, & \mbox{otherwise}.
\end{array}\right.\label{eq:update_rho_proximal}
\end{equation}
The iteration terminates when the following two conditions are
met:
\begin{equation}
\left\|\sum\limits_{i=1}^n\mathcal{A}_i(\mathbf{x}_i^{k+1}) -
\mathbf{b}\right\|/\|\mathbf{b}\| <
\varepsilon_1,\label{eq:stopping3}
\end{equation}
\begin{eqnarray}
\begin{array}{l}
\max\left(\left\{\|\A_i\|^{-1}\lbar \nabla
g_i(\mathbf{x}_i^{k+1})-\nabla g_i(\mathbf{x}_i^k)-
\tau_{i}^{(k)}(\mathbf{x}_i^{k+1}-\mathbf{x}_i^k)\rbar,\right.\right.\\
\quad i=1,\cdots,n\Big\}\Big)/\|\mathbf{b}\| < \varepsilon_2.
\end{array}\label{eq:stopping4}
\end{eqnarray}
These two conditions are also deduced from the KKT conditions.

We call the above algorithm as proximal LADMPSAP and summarize it
in Algorithm~\ref{alg:general-LADMPSAP-multivar}.

\begin{algorithm}[tb]
   \caption{Proximal LADMPSAP for Solving (\ref{eq:model_problem_multivar}) with $f_i$ Satisfying \eqref{eq:decompose_f}.}
   \label{alg:general-LADMPSAP-multivar}
\begin{algorithmic}
   \STATE {\bfseries Initialize:}
     Set $\rho_0>1$, $\beta_0 > 0$, $\bflambda^0$, $T_i \geq L_i$, $\eta_i > n\|\A_i\|^2$, $\mathbf{x}_i^0$,
$i=1,\cdots,n$.
   \WHILE {(\ref{eq:stopping3}) or (\ref{eq:stopping4}) is not satisfied}
   \STATE {\bfseries Step 1:} Compute $\hat{\bflambda}^k$ as (\ref{eq:hat_lambda}).
   \STATE {\bfseries Step 2:} Update $\mathbf{x}_i$'s in parallel by solving
   \begin{equation}
   \x_i^{k+1}=\argmin \limits_{\x_i} h_i(\x_i)+\frac{\tau_i^{(k)}}{2} \lbar \x_i-\x_i^k+\frac{1}{\tau_i^{(k)}}[\A_i^{\dag}(\hat \bflambda^k)+\nabla g_i(\x_i^k)] \rbar^2,\ i=1,\cdots,n,\label{eq:general_LADMPSAP_update_xi}
   \end{equation}
   where $\tau_i^{(k)}=T_i+\beta_k \eta_i$.
   \STATE {\bfseries Step 3:} Update $\bflambda$ by (\ref{eq:update_lambda_multivar}) and $\beta$ by \eqref{eq:update_beta_multivar}
   with $\rho$ defined in \eqref{eq:update_rho_proximal}.
   \ENDWHILE
\end{algorithmic}
\end{algorithm}

As for the convergence of proximal LADMPSAP, we have the following
theorem.

\begin{theorem}\label{thm:general_convergence_unbounded}{\bf(Convergence of Proximal LADMPSAP)}
If $\beta_k$ is non-decreasing and upper bounded,
$\tau_i^{(k)}=T_i+\beta_k \eta_i$, where $T_i \geq L_i$ and
$\eta_i
> n\|\A_i\|^2$ are both positive constants, $i=1,\cdots,n$, then
$\{(\{\mathbf{x}_i^k\},\bflambda^k)\}$ generated by proximal
LADMPSAP converge to a KKT point of problem
(\ref{eq:model_problem_multivar}).
\end{theorem}

We further have the following convergence rate theorem for
proximal LADMPSAP in an ergodic sense.
\begin{theorem}\label{thm:general_convergence_unbounded_rate}{\bf (Convergence Rate of Proximal LADMPSAP)}
Define $\bar\x_i^K=\sum\limits_{k=0}^K \gamma_k \x_i^{k+1}$, where
$\gamma_k=\beta_k^{-1}/\sum\limits_{j=0}^K \beta_j^{-1}$. Then the following inequality holds for $\bar{\x}_i^K$:
\begin{eqnarray}
&&\sum\limits_{i=1}^n \left(f_i(\bar\x_i^{K})-f_i(\x_i^*)+\<\A_i^{\dag} (\lambda^*),\bar\x_i^{K}-\x_i^*\>\right)+\frac{\alpha\beta_0}{2}\lbar\sum\limits_{i=1}^n\A_i(\bar\x_i^{K})-\b\rbar^2\notag\\
&\leq&C_0/\sum\limits_{k=0}^K2\beta_k^{-1},
\end{eqnarray}
where $\alpha^{-1}=(n+1) \max\left(1,
\left\{\dfrac{\|\A_i\|^2}{\eta_i-n\|\A_i\|^2},
i=1,\cdots,n\right\}\right)$ and
$C_0=\sum\limits_{i=1}^n\beta_0^{-1}\tau_i^{(0)}
\|\x_i^{0}-\x_i^*\|^2+\beta_0^{-2}\|\lambda^{0}-\lambda^*\|^2$.
\end{theorem}

When there are extra convex set constraints, $\x_i\in X_i$,
$i=1,\cdots,n$, we can also introduce auxiliary variables as in
Section~\ref{sec:PPLAMDAP} and have an analogy of
Theorems~\ref{thm:better_eta} and \ref{thm:convergence_rate_2var}.
\begin{theorem}\label{thm:better_eta_proximal}
For problem (\ref{eq:model_problem_multivar_equiv}), where $f_i$
is described at the beginning of Section~\ref{sec:G-LADMPSAP}, if
$\beta_k$ is non-decreasing and upper bounded and
$\tau_i^{(k)}=T_i+\eta_i\beta_k$, where $T_i \geq L_i$,
$T_{n+i}=0$, $\eta_i>n\|\A_i\|^2 + 2$, and $\eta_{n+i}
> 2$, $i=1,\cdots,n$, then $\{(\{\mathbf{x}_i^k\},\bflambda^k)\}$
generated by proximal LADMPSAP converge to a KKT point of problem
(\ref{eq:model_problem_multivar_equiv}). The convergence rate in
an ergodic sense is also $O\left(1/\sum\limits_{k=0}^K
\beta_k^{-1}\right)$, where $K$ is the number of iterations.
\end{theorem}

\section{Numerical Results}\label{sec:exp}
In this section, we test the performance of LADMPSAP on three
specific examples of problem (\ref{eq:model_problem_multivar}),
i.e., Latent Low-Rank Representation (see \eqref{eq:llrr}),
Nonnegative Matrix Completion (see \eqref{eq:nmc}), and Group
Sparse Logistic Regression with Overlap (see \eqref{eq:logit'}).

\subsection{Solving Latent Low-Rank Representation}\label{sec:solve_LLRR}
We first solve the latent LRR
problem~\cite{Liu-2011-LLRR}~\eqref{eq:llrr}. In order to test
LADMPSAP and related algorithms with data whose characteristics
are controllable, we follow \citep{Liu-2010-LRR} to generate
synthetic data, which are parameterized as ($s$, $p$, $d$,
$\tilde{r}$), where $s$, $p$, $d$, and $\tilde{r}$ are the number
of independent subspaces, points in each subspace, and ambient and
intrinsic dimensions, respectively. The number of scale variables
and constraints is $(sp)\times d$.

As first order methods are popular for solving convex programs in
machine learning~\citep{Boyd-2011-Distributed}, here we compare
LADMPSAP with several conceivable first order algorithms,
including APG~\citep{Beck2009}, naive ADM, naive LADM, LADMGB, and
LADMPS. Naive ADM and naive LADM are generalizations of ADM and
LADM, respectively, which are straightforwardly generalized from
two variables to multiple variables, as discussed in
Section~\ref{sec:LADM+PSAP}. Naive ADM is applied to solve
 (\ref{eq:llrr}) after rewriting the constraint of (\ref{eq:llrr}) as $\bx =
\bx\mathbf{P} + \mathbf{Q}\bx+ \be, \mathbf{P}=\bz,
\mathbf{Q}=\bl$. For LADMPS, $\beta_k$ is fixed in order to show
the effectiveness of adaptive penalty.
The parameters of APG and ADM are the same as those
in~\citep{Lin-09-APG} and~\citep{Liu-2011-LLRR}, respectively. For
LADM, we follow the suggestions in \citep{Yang-2011-LADM} to fix
its penalty parameter $\beta$ at $2.5/\min(d,sp)$, where $d\times
sp$ is the size of $\bx$. For LADMGB, as there is no suggestion
in~\cite{He-11-LADM_Gauss} on how to choose a fixed $\beta$, we
simply set it the same as that in LADM. The rest of the parameters are
the same as those suggested in \citep{He-2011-Gaussian}. We fix $\beta =
\sigma_{\max}(\mathbf{X})\min(d,sp)\varepsilon_2$ in LADMPS and
set $\beta_0 = \sigma_{\max}(\mathbf{X})\min(d,sp)\varepsilon_2$
and $\rho_0 = 10$ in LADMPSAP. For LADMPSAP, we also set
$\eta_Z=\eta_L=1.02\times 3\sigma_{\max}^2(\mathbf{X})$, where
$\eta_Z$ and $\eta_L$ are the parameters $\eta_i$'s in
Algorithm~\ref{alg:LADMPSAP-multivar} for $\bz$ and $\bl$,
respectively. For the stopping criteria, $\|\bx\bz^k+\bl^k\bx +
\be^k-\bx\|/\|\bx\|\leq\varepsilon_1$ and
$\max(\|\bz^{k}-\bz^{k-1}\|,\|\bl^{k}-\bl^{k-1}\|,\|\be^{k}-\be^{k-1}\|)/\|\bx\|\leq\varepsilon_2$,
with $\varepsilon_1=10^{-3}$ and $\varepsilon_2=10^{-4}$ are used
for all the algorithms. For the parameter $\mu$ in
(\ref{eq:llrr}), we empirically set it as $\mu = 0.01$. To measure
the relative errors in the solutions we run LADMPSAP 2000
iterations with $\rho_0 = 1.01$ to obtain the estimated ground
truth solution ($\mathbf{Z}^*, \mathbf{L}^*, \mathbf{E}^*$). The
experiments are run and timed on a notebook computer with an Intel
Core i7 2.00 GHz CPU and 6GB memory, running Windows 7 and Matlab 7.13.

Table~\ref{tab:latlrr} shows the results of related algorithms. We
can see that LADMPS and LADMPSAP are faster and more numerically
accurate than LADMGB, and LADMPSAP is even faster than LADMPS
thanks to the adaptive penalty. Moreover, naive ADM and naive LADM
have relatively poorer numerical accuracy, possibly due to
converging to wrong solutions. The numerical accuracy of APG is
also worse than those of LADMPS and LADMPSAP because it only
solves an approximate problem by adding the constraint to the
objective function as penalty. Note that although we do not
require $\{\beta_k\}$ to be bounded, this does not imply that
$\beta_k$ will grow infinitely. As a matter of fact, when LADMPSAP
terminates the final values of $\beta_k$ are $21.1567$, $42.2655$,
and $81.4227$ for the three data settings, respectively.

We then test the performance of the above six algorithms on the
Hopkins155 database \citep{Tron-2007-Hopkins}, which consists of
156 sequences, each having 39 to 550 data vectors drawn from two
or three motions. For computational efficiency, we preprocess the
data by projecting them to be 5-dimensional using PCA. We test all
algorithms with $\mu = 2.4$, which is the best parameter for LRR
on this database \citep{Liu-2010-LRR}. \invisible{We set
$\varepsilon_1=10^{-3}$ and $\varepsilon_2=10^{-4}$. For LADMPSAP,
we set $\beta_0 = 0.005/\min(d,sp)$ and $\rho_0 = 1.9$. For LADM,
LADMGB and LADMPS, we set $\beta=0.05/\min(d,sp)$.}Table
{\ref{tab:latlrr_rel}} shows the results on the Hopkins155
database. We can also see that LADMPSAP is faster than other
methods in comparison. In particular, LADMPSAP is faster than
LADMPS, which uses a fixed $\beta$. This testify to the advantage
of using an adaptive penalty.

\invisible{
\begin{table*}[th]
\begin{center}
\caption{Comparison of APG, naive ADM (nADM), naive LADM (nLADM),
LADMGB, LADMPS and LADMPSAP on the latent LRR problem
(\ref{eq:llrr}). The quantities include computing time (in
seconds), number of iterations, relative errors, and clustering
accuracy (in percentage). They are averaged over $10$
runs.}\label{tab:latlrr}
\end{center}
\end{table*}
}

\begin{table*}[t]
\caption{Comparisons of APG, naive ADM (nADM), naive LADM (nLADM),
LADMGB, LADMPS, and LADMPSAP on the latent LRR problem
(\ref{eq:llrr}). The quantities include computing time (in
seconds), number of iterations, relative errors, and clustering
accuracy (in percentage). They are averaged over $10$
runs.}\label{tab:latlrr}
\begin{center}
\begin{tabular}{c|c||c|c|c|c|c||c}\hline
($s$, $p$, $d$, $\tilde{r}$) & Method & Time &  \#Iter. &
$\frac{\|\hat{\mathbf{Z}}-\mathbf{Z}^*\|}{\|\mathbf{Z}^*\|}$ &
$\frac{\|\hat{\mathbf{L}}-\mathbf{L}^*\|}{\|\mathbf{L}^*\|}$ &
$\frac{\|\hat{\mathbf{E}}-\mathbf{E}^*\|}{\|\mathbf{E}^*\|}$ &
Acc.
\\\hline\hline
\multirow{6}{*}{(5, 50, 250, 5)}
& APG       & 18.20         & 236          & 0.3389            & 0.3167          & 0.4500             & 95.6  \\
& nADM       & 16.32         & 172          & 0.3993            & 0.3928          & 0.5592             & 95.6  \\
& nLADM      & 21.34         & 288          & 0.4553            & 0.4408          & 0.5607             & 95.6  \\
& LADMGB     & 24.10         & 290          & 0.4520            & 0.4355          & 0.5610             & 95.6  \\
& LADMPS    & 17.15         & 232          & 0.0163            & 0.0139          & \textbf{0.0446}    & 95.6  \\
& LADMPSAP  & \textbf{8.04} & \textbf{109}  & \textbf{0.0089}   &
\textbf{0.0083}& 0.0464             & 95.6  \\\hline
\multirow{6}{*}{(10, 50, 500, 5)}
& APG       & 85.03          & 234          & 0.1020            & 0.0844         & 0.7161            & 95.8  \\
& nADM       & 78.27          & 170          & 0.0928            & 0.1026         & 0.6636            & 95.8  \\
& nLADM      & 181.42         & 550          & 0.2077            & 0.2056         & 0.6623            & 95.8  \\
& LADMGB     & 214.94         & 550          & 0.1877            & 0.1848         & 0.6621            & 95.8  \\
& LADMPS    & 64.65          & 200          & 0.0167            & 0.0089         & 0.1059            & 95.8  \\
& LADMPSAP  & \textbf{37.85} & \textbf{117} & \textbf{0.0122}   &
\textbf{0.0055}& \textbf{0.0780}   & 95.8  \\\hline
\multirow{6}{*}{(20, 50, 1000, 5)}
& APG       & 544.13          & 233          & 0.0319            & 0.0152            & 0.2126           & 95.2  \\
& nADM       & 466.78          & 166          & 0.0501            & 0.0433            & 0.2676           & 95.2  \\
& nLADM      & 1888.44         & 897          & 0.1783            & 0.1746            & 0.2433           & 95.2  \\
& LADMGB     & 2201.37         & 897          & 0.1774            & 0.1736            & 0.2434           & 95.2  \\
& LADMPS    & 367.68          & 177          & 0.0151            & 0.0105            & 0.0872           & 95.2  \\
& LADMPSAP  & \textbf{260.22} & \textbf{125} & \textbf{0.0106}   &
\textbf{0.0041}   & \textbf{0.0671}  & 95.2  \\\hline
\end{tabular}
\end{center}
\end{table*}

\begin{table*}[t]
\caption{Comparisons of APG, naive ADM (nADM), naive LADM (nLADM),
LADMGB, LADMPS, and LADMPSAP on the Hopkins155 database. The
quantities include average computing time, average number of
iterations, and average classification errors on all 156
sequences.}\label{tab:latlrr_rel}
\begin{center}
\begin{tabular}{c||c|c||c}\hline
 Method & Time (seconds) &  \#Iteration & Error (\%)\\
\hline\hline
 APG        & 10.37         & 67            & 8.33  \\
 nADM       & 24.76         & 144           & 8.33  \\
 nLADM      & 15.50         & 112            & 8.33  \\
 LADMGB     & 16.05         & 113            & 8.36  \\
 LADMPS     & 15.58         & 113           & 8.33  \\
 LADMPSAP   & \textbf{3.80} & \textbf{26}   & 8.33  \\\hline
\end{tabular}
\end{center}
\end{table*}

\subsection{Solving Nonnegative Matrix Completion}

This subsection evaluates the performance of the practical
LADMPSAP proposed in Section~\ref{sec:PPLAMDAP} for solving
nonnegative matrix completion~\citep{Xu-2011-NMF} \eqref{eq:nmc}.


We first evaluate the numerical performance on synthetic data to
demonstrate the superiority of practical LADMPSAP over the
conventional LADM\footnote{Code available at
\url{http://math.nju.edu.cn/~jfyang/IADM_NNLS/index.html}}
\citep{Yang-2011-LADM}. The nonnegative low-rank matrix $\bx_0$ is
generated by truncating the singular values of a randomly
generated matrix. As LADM cannot handle the nonnegativity
constraint, it actually solve the standard matrix completion
problem, i.e., (\ref{eq:nmc}) without the nonnegativity
constraint. For LADMPSAP, we follow the conditions in
Theorem~\ref{thm:better_eta} to set $\eta_i$'s and set the rest of
the parameters the same as those in Section~\ref{sec:solve_LLRR}.
The stopping tolerances are set as $\varepsilon_1 = \varepsilon_2
= 10^{-5}$. The numerical comparison is shown in
Table~\ref{tab:nmc_syn}, where the relative nonnegative
feasibility (FA) is defined as \citep{Xu-2011-NMF}:
\begin{equation*}
\mbox{FA} := \|\min(\hat{\bx},0)\|/\|\bx_0\|,
\end{equation*}
in which $\X_0$ is the ground truth and $\hat{\bx}$ is the
computed solution. It can be seen that the numerical performance
of LADMPSAP is much better than that of LADM, thus again verifies
the efficiency of our proposed parallel splitting and adaptive
penalty scheme for enhancing ADM/LADM type algorithms.
\begin{table}[t]
\begin{center}
\caption{Comparisons on the NMC problem (\ref{eq:nmc}) with
synthetic data, averaged on $10$ runs. $q$, $t$, and $d_r$ denote,
respectively, the sample ratio, the number of measurements $t =
q(mn)$, and the ``degree of freedom'' defined by $d_r = r(m+n-r)$
for an $m\times n$ matrix with rank $r$ and $q$. Here we set $m =
n$ and fix $r = 10$ in all the tests.}\label{tab:nmc_syn}
\vspace{1em}
\begin{tabular}{c|c|c||c|c|c|c||c|c|c|c}\hline
\multicolumn{3}{c||}{$\bx$} & \multicolumn{4}{c||}{LADM} & \multicolumn{4}{c}{LADMPSAP}\\\hline
$n$ & $q$ & $t/d_r$ & \#Iter. & Time (s) & RelErr & FA & \#Iter. & Time (s) & RelErr & FA \\\hline\hline
\multirow{2}{*}{1000}
& 20$\%$ & 10.05 & 375  & 177.92 & 1.35E-5 & 6.21E-4 & 58 & {\bf 24.94} & {\bf 9.67E-6} & {\bf 0} \\
& 10$\%$ & 5.03  & 1000 & 459.70 & 4.60E-5 & 6.50E-4 & 109& {\bf
42.68} & {\bf 1.72E-5} & {\bf 0}\\\hline
\multirow{2}{*}{5000}
& 20$\%$ & 50.05 & 229  & 1613.68 & 1.08E-5 & 1.93E-4 & 49 & {\bf 369.96} & {\bf 9.05E-6} & {\bf 0} \\
& 10$\%$ & 25.03 & 539  & 2028.14 & 1.20E-5 & 7.70E-5 & 89 & {\bf
365.26} & {\bf 9.76E-6} & {\bf 0}\\\hline 10000 & 10$\%$ & 50.03 &
463 & 6679.59 & 1.11E-5 & 4.18E-5 & 89 & {\bf 1584.39} & {\bf
1.03E-5} & {\bf 0}
\\\hline
\end{tabular}
\end{center}
\end{table}

We then consider the image inpainting problem, which is to fill in
the missing pixel values of a corrupted image. As the pixel values
are nonnegative, the image inpainting problem can be formulated as
the NMC problem. To prepare a low-rank image, we also truncate the
singular values of a $1024 \times 1024$ grayscale image
``man''\footnote{Available at \url{http://sipi.usc.edu/database/}}
to obtain an image of rank 40, shown in Fig.~\ref{fig:nmc}
(a)-(b). The corrupted image is generated from the original image (all pixels have been normalized in the range of [0, 1])
by sampling $20\%$ of the pixels uniformly at random and adding
Gaussian noise with mean zero and standard deviation 0.1.

Besides LADM, here we also consider another recently proposed
fixed point continuation with approximate SVD (FPCA
\citep{Ma-2008-FPC}) on this problem. Similar to LADM, the code of
FPCA\footnote{Code available at
\url{http://www1.se.cuhk.edu.hk/~sqma/softwares.html}} can only
solve the standard matrix completion problem without the
nonnegativity constraint. This time we set $\varepsilon_1 =
10^{-3}$ and $\varepsilon_2 = 10^{-1}$ as the thresholds for
stopping criteria. The recovered images are shown in
Fig.~\ref{fig:nmc} (c)-(e) and the quantitative results are in
Table~\ref{tab:nmc}. One can see that on our test image both the
qualitative and the quantitative results of LADMPSAP are better
than those of FPCA and LADM. Note that LADMPSAP is faster than
FPCA and LADM even though they do not handle the nonnegativity
constraint.

\invisible{
Then we run two recently proposed algorithms (i.e., fixed point continuation with approximate SVD (FPCA)\footnote{The
Matlab code is provided by authors of \citep{Ma-2008-FPC}}
and LADM\footnote{The
Matlab code is provided by authors of \citep{Yang-2011-LADM}. })\footnote{As
the code cannot handle the nonnegativity constraint, it actually
solves the standard matrix completion problem, i.e.,
(\ref{eq:nmc}) without the nonnegativity constraint.} and LADMPSAP
on this image to compare their performance on the nonnegative
matrix completion problem. }

\invisible{
\begin{figure}[th]
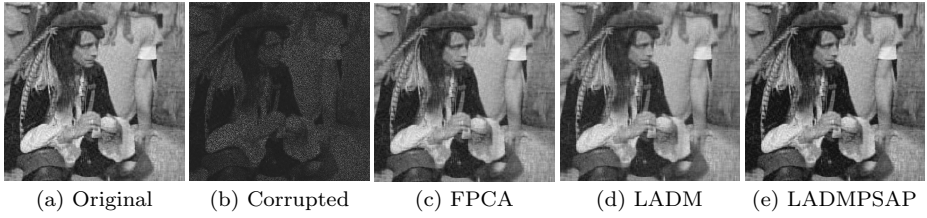

\caption{Image inpainting by FPCA, LADM, and
LADMPSAP.}\label{fig:nmc}
\end{figure}
}
\invisible{
\begin{figure}[th]
\begin{center}
\subfigure[The original and corrupted images]{
\begin{tabular}{c@{\extracolsep{0.2em}}c}
\includegraphics[width=0.3\textwidth,
keepaspectratio]{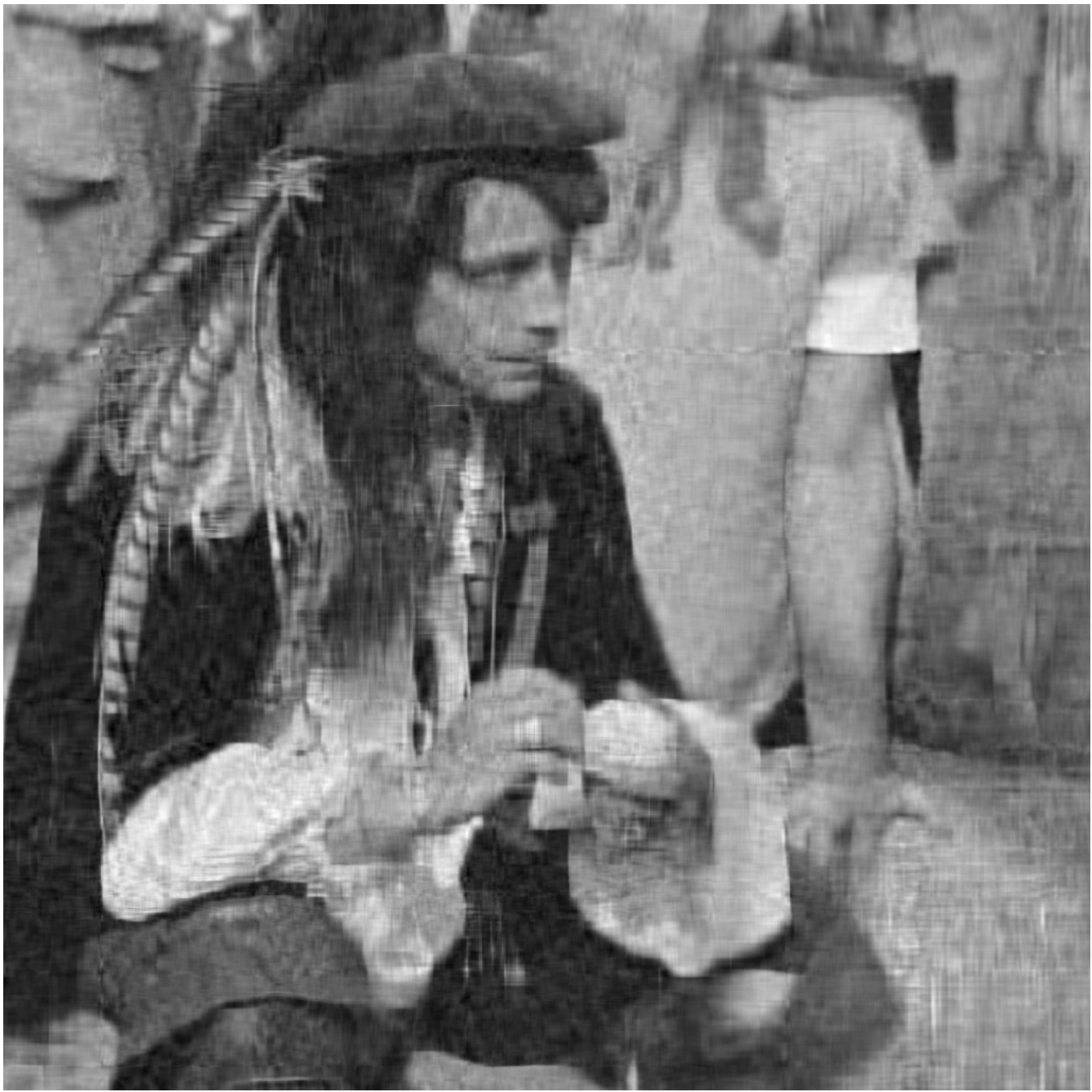}
&\includegraphics[width=0.3\textwidth,
keepaspectratio]{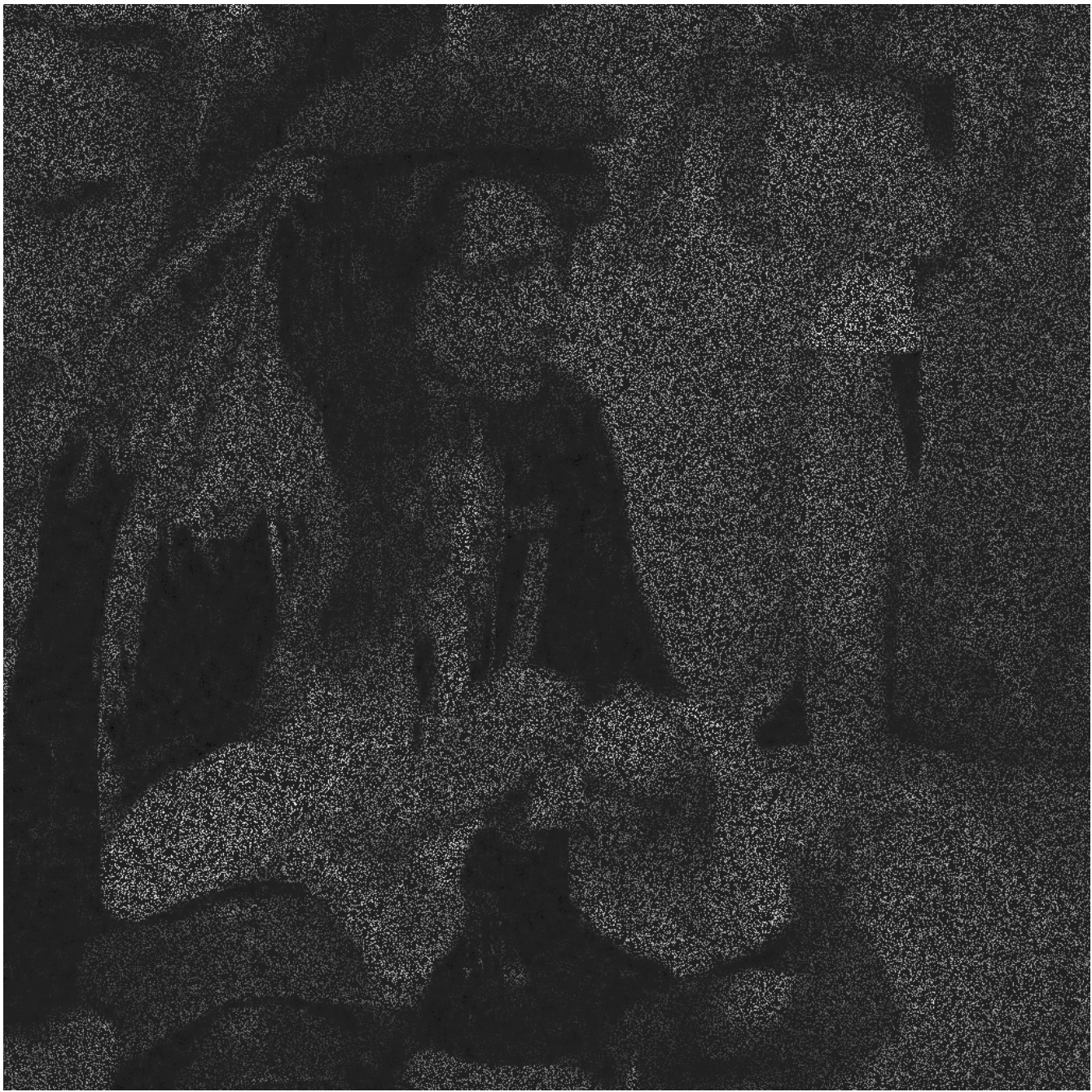}\\
 Original & Corrupted\\
\end{tabular}
}
\subfigure[The inpaining results]{
\begin{tabular}{c@{\extracolsep{0.2em}}c@{\extracolsep{0.2em}}c}
\includegraphics[width=0.3\textwidth,
keepaspectratio]{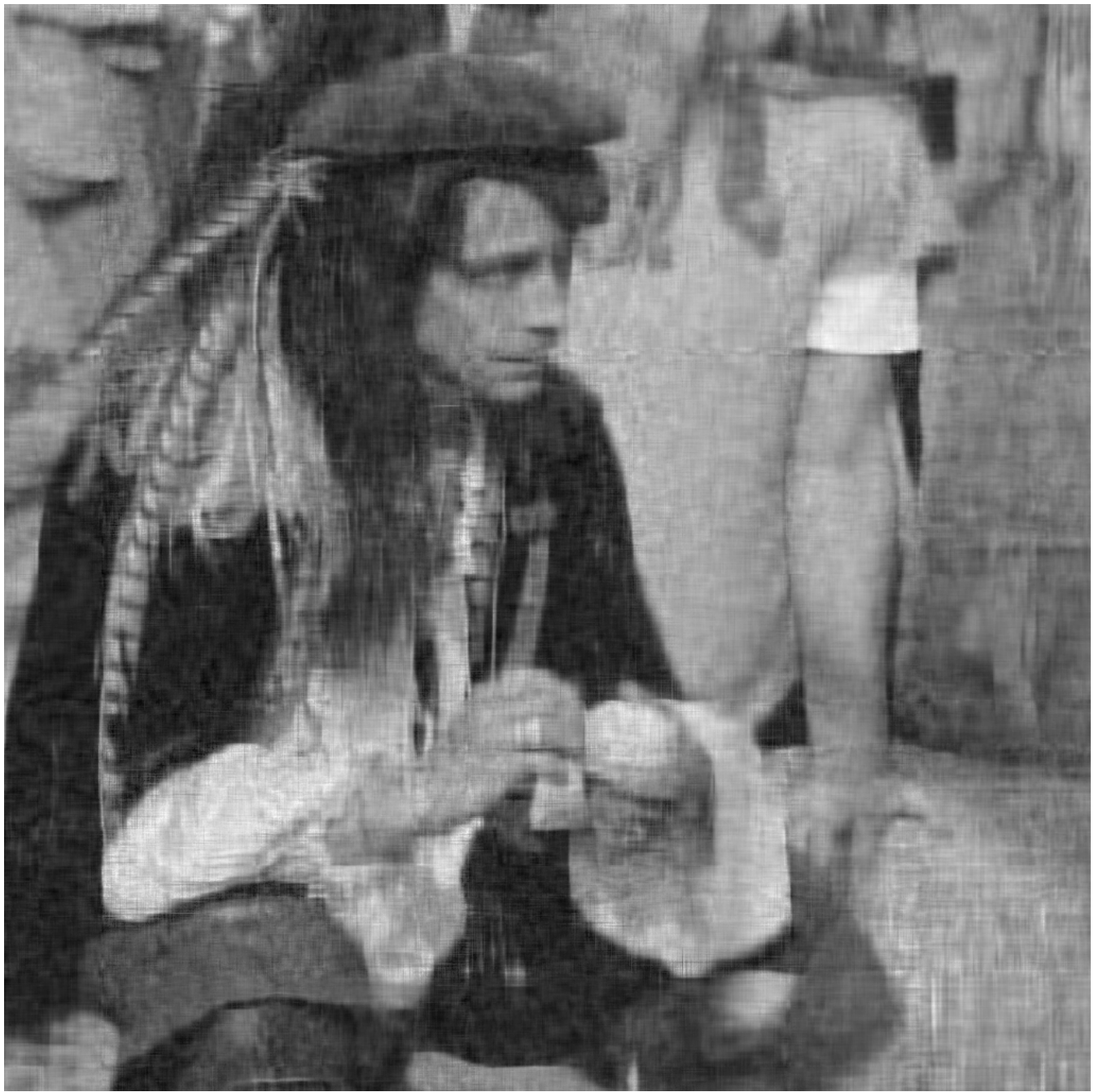}
&\includegraphics[width=0.3\textwidth,
keepaspectratio]{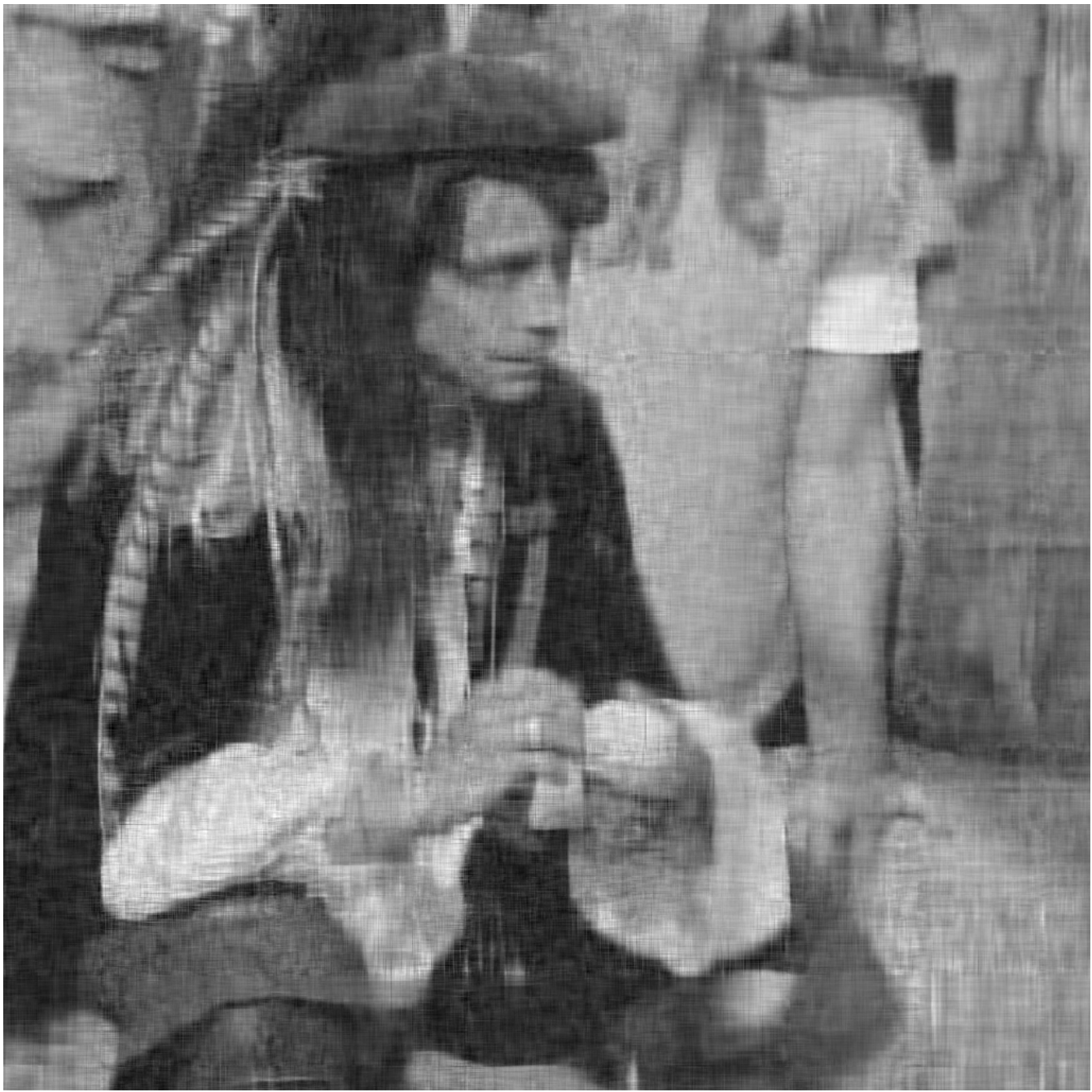}
&\includegraphics[width=0.3\textwidth,
keepaspectratio]{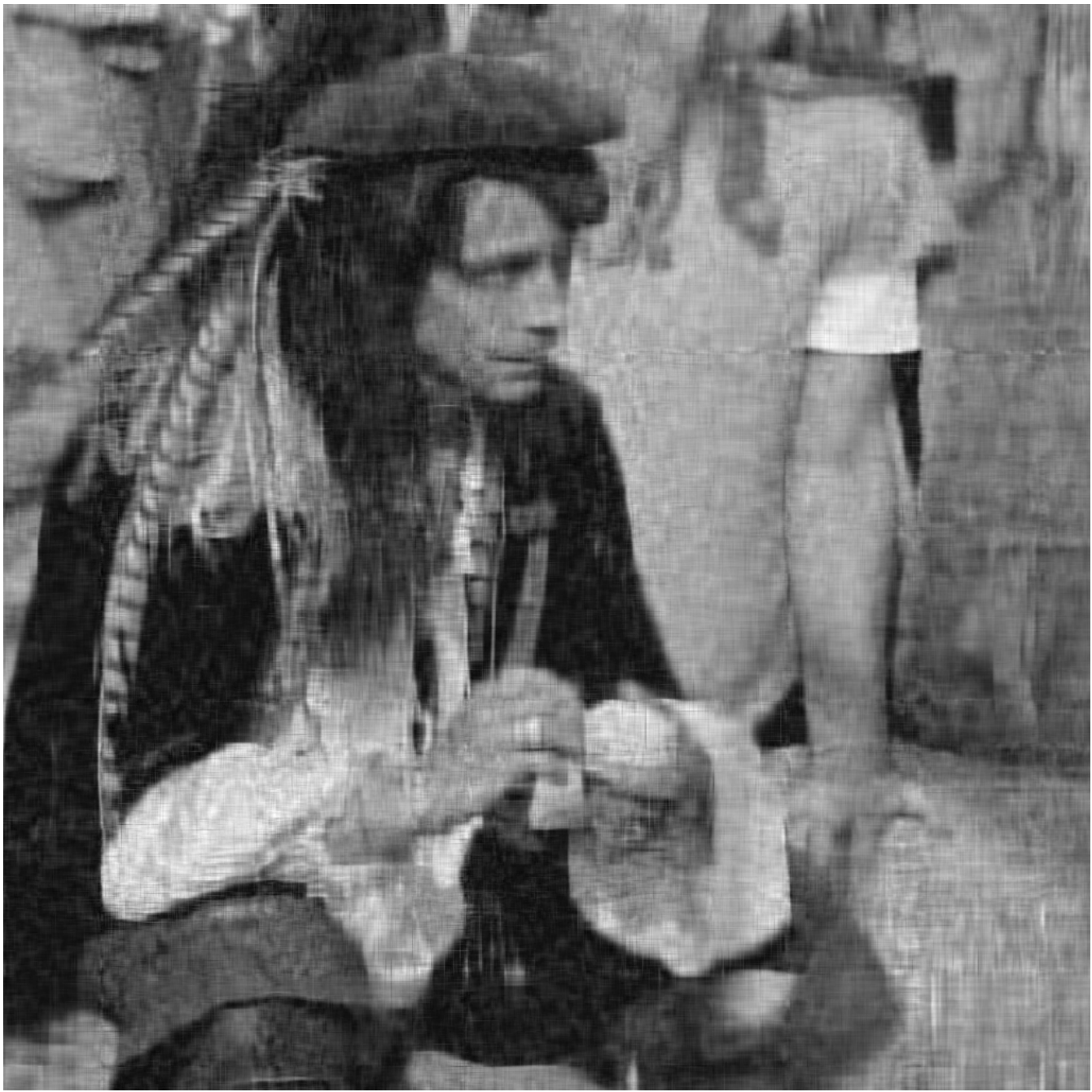}\\
FPCA & LADM & LADMPSAP\\
\end{tabular}
}
\end{center}
\caption{Image inpainting by FPCA, LADM, and
LADMPSAP.}\label{fig:nmc}
\end{figure}
}

\begin{figure}[t]
\centering
\begin{tabular}{c@{\extracolsep{0.2em}}c@{\extracolsep{0.2em}}c@{\extracolsep{0.2em}}c@{\extracolsep{0.2em}}c}
\includegraphics[width=0.2\textwidth,
keepaspectratio]{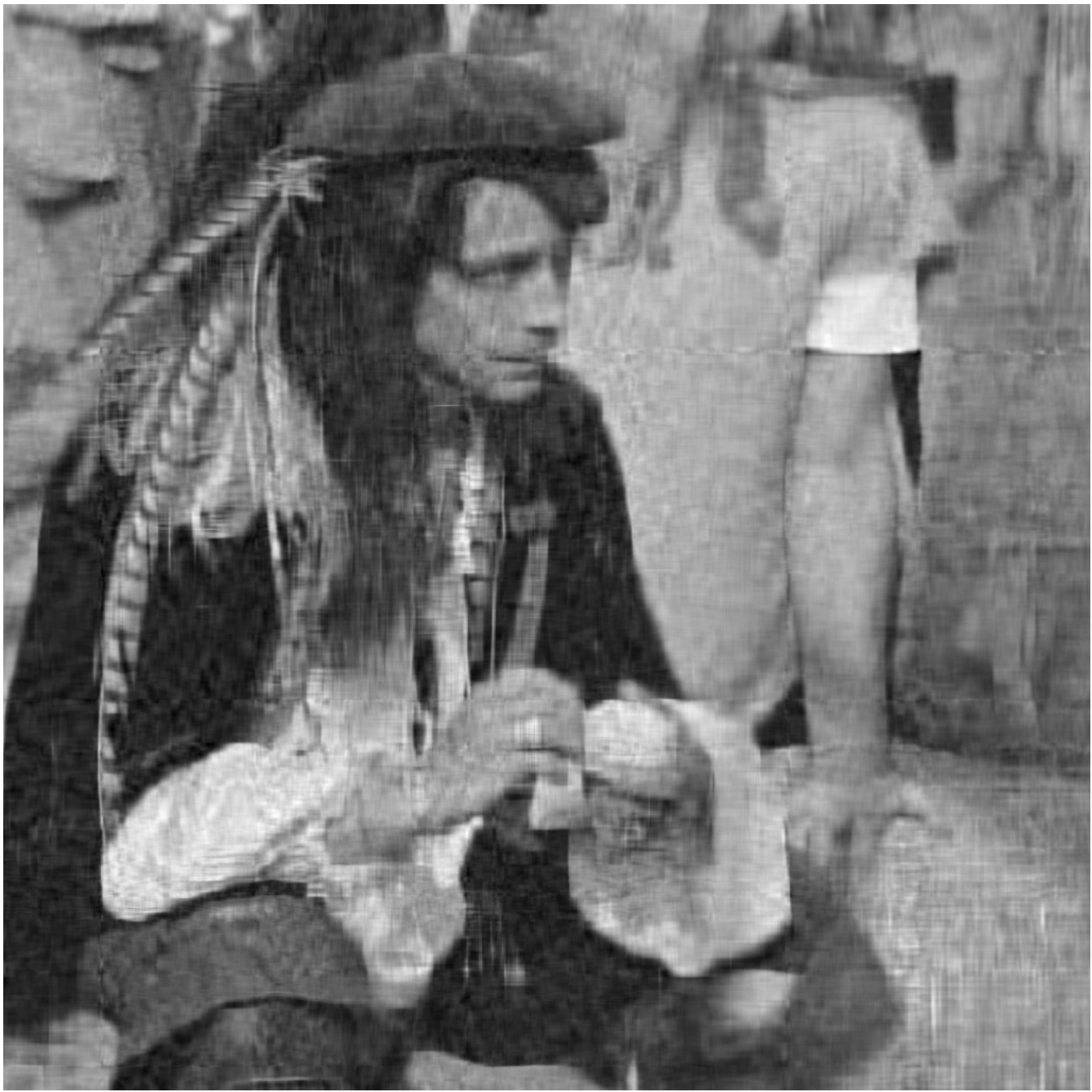}
&\includegraphics[width=0.2\textwidth,
keepaspectratio]{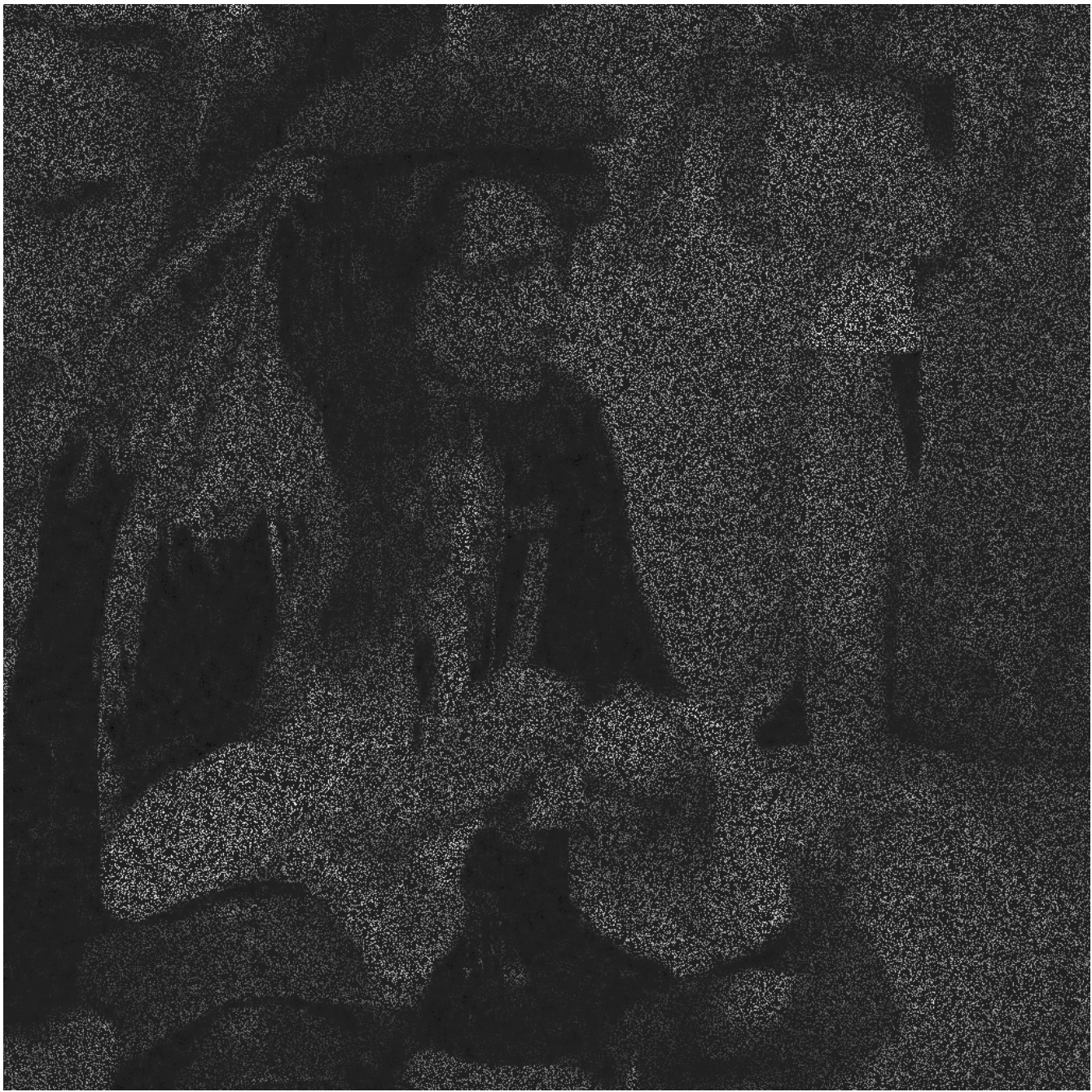}
&\includegraphics[width=0.2\textwidth,
keepaspectratio]{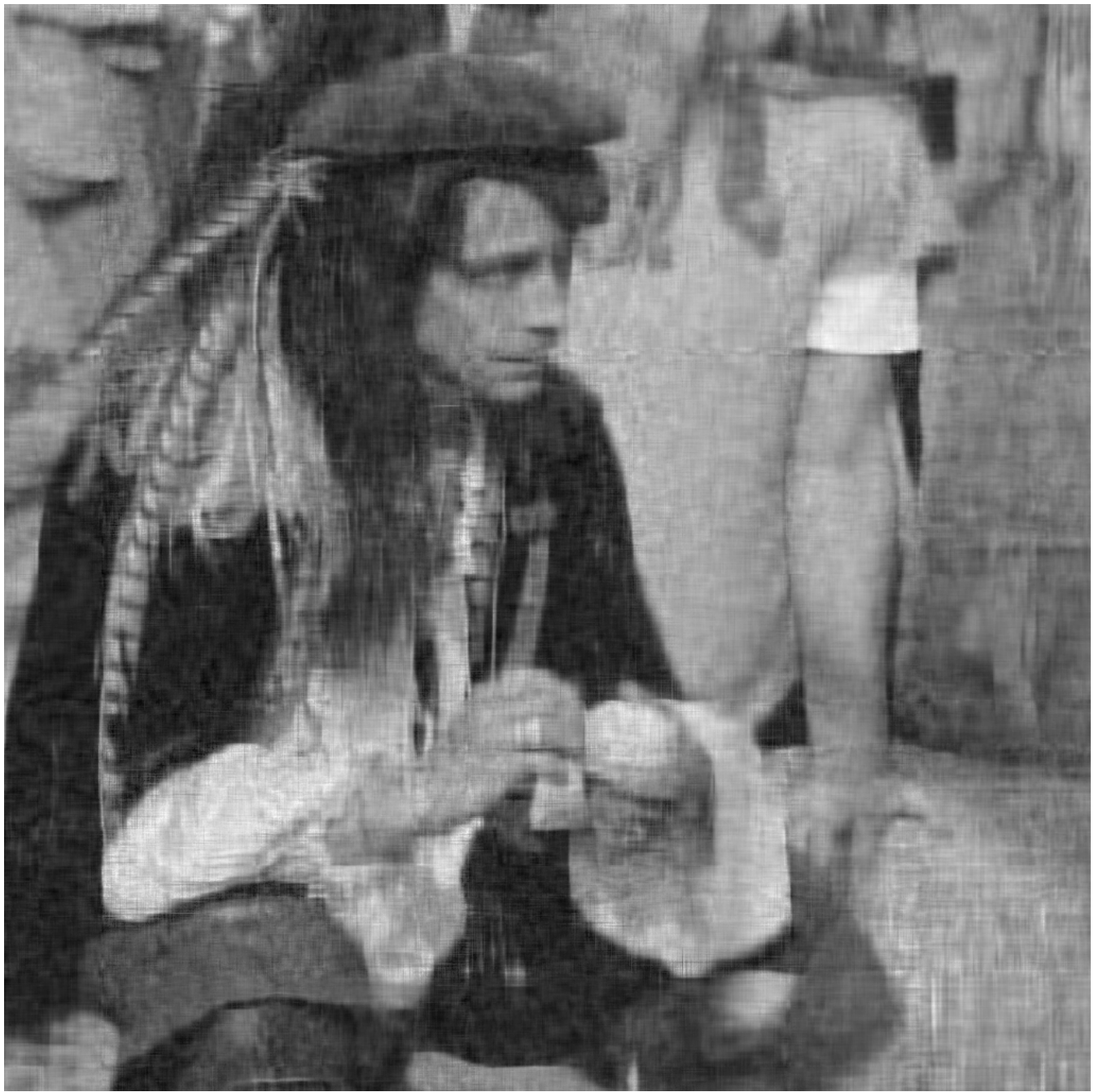}
&\includegraphics[width=0.2\textwidth,
keepaspectratio]{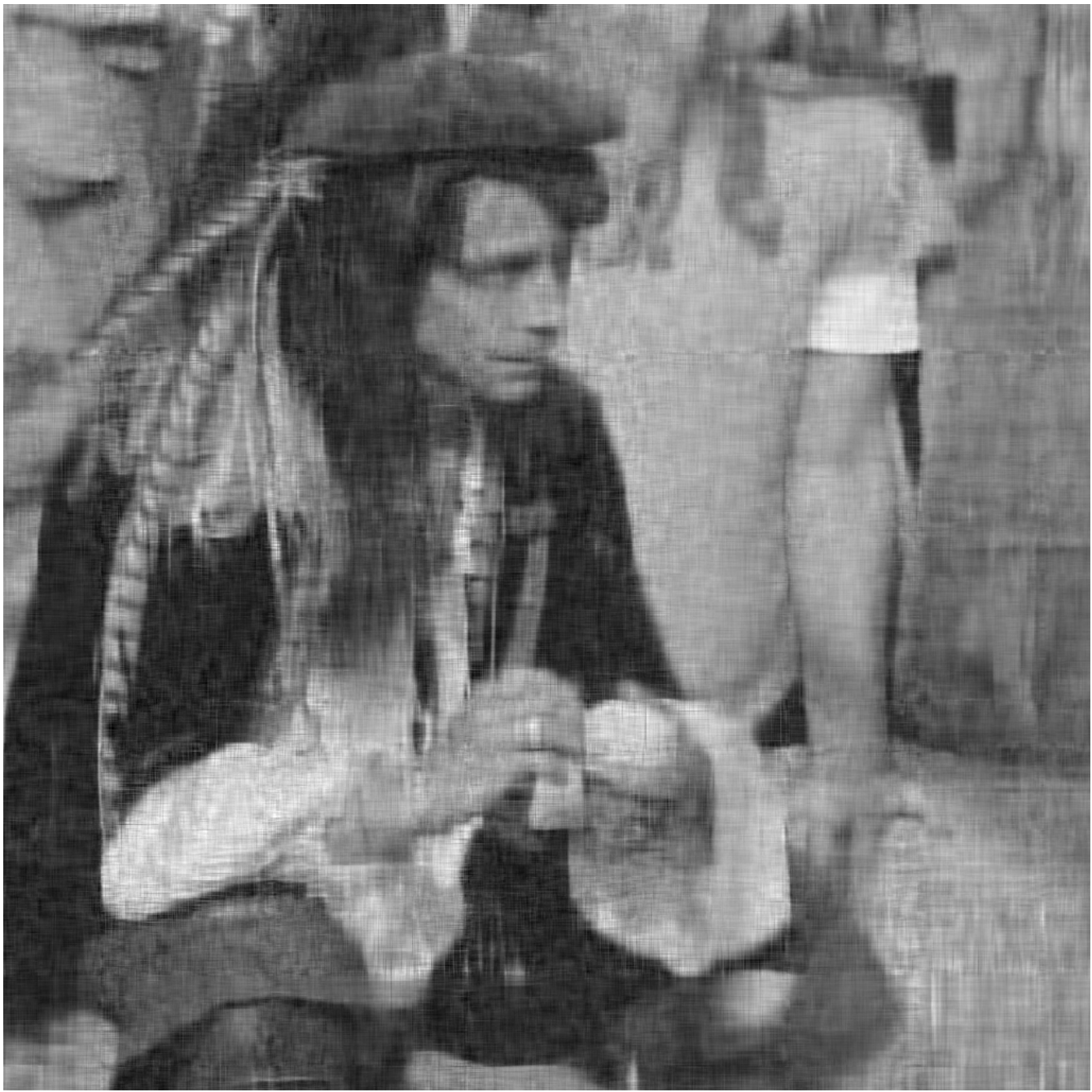}
&\includegraphics[width=0.2\textwidth,
keepaspectratio]{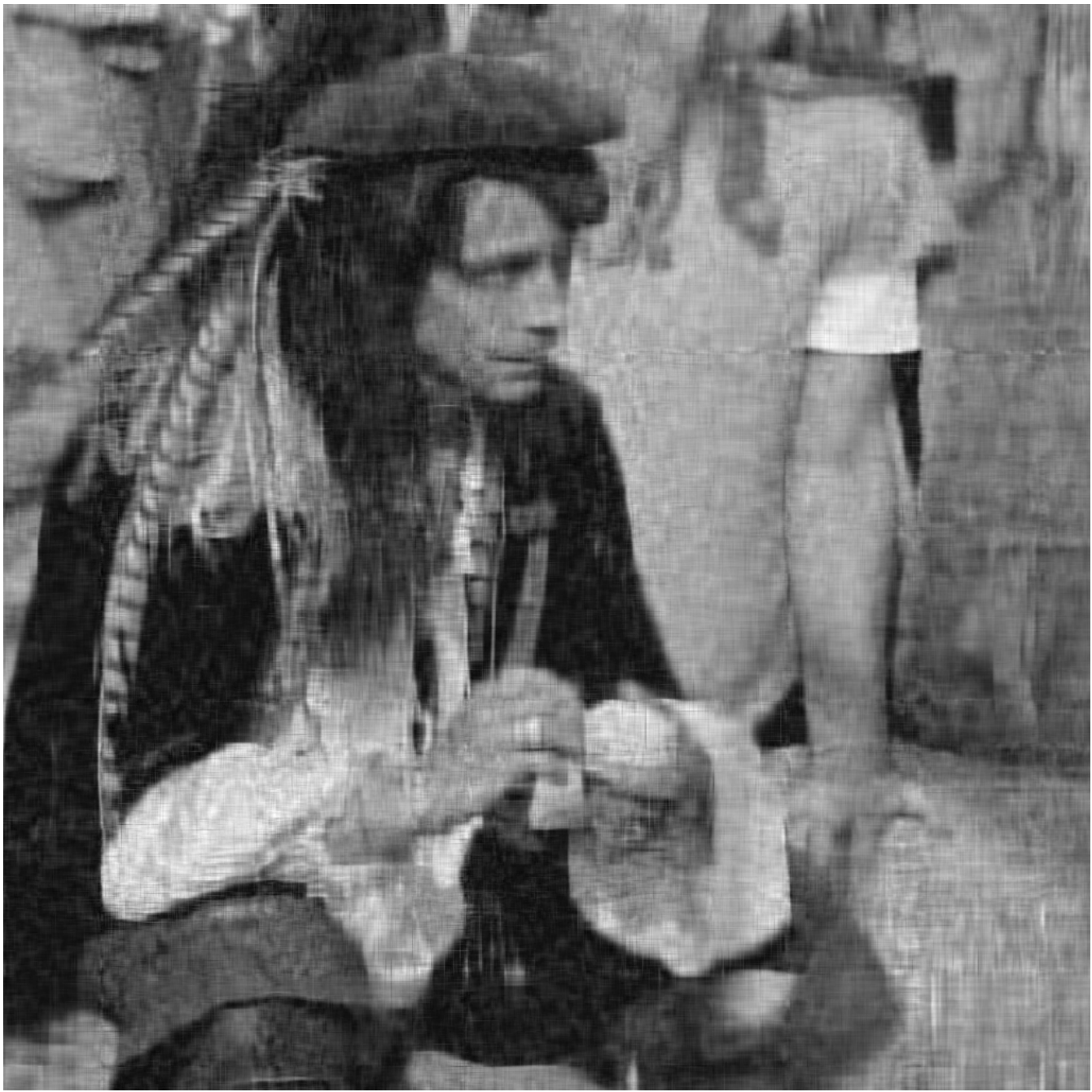}\\
(a) Original & (b) Corrupted & (c) FPCA & (d) LADM & (e) LADMPSAP
\end{tabular}
\caption{Image inpainting by FPCA, LADM and
LADMPSAP.}\label{fig:nmc}
\end{figure}


\invisible{
\begin{table}[th]
\caption{Numerical comparison on the image inpainting problem
(nonnegative matrix completion). }\label{tab:nmc}
\end{table}
}

\begin{table}[t]
\begin{center}
\caption{Comparisons on the image inpainting problem. ``PSNR"
stands for ``Peak Signal to Noise Ratio" measured in decibel
(dB).}\label{tab:nmc} \vspace{1em}
\begin{tabular}{c||c|c|c|c}\hline
Method & \#Iter. & Time (s) & PSNR (dB) & FA\\\hline\hline
{\small FPCA} & 179 & 228.99 &  27.77 & 9.41E-4 \\
{\small LADM} & 228 & 207.95 &  26.98 & 2.92E-3 \\
{\small LADMPSAP} & 143 & {\bf 134.89} & {\bf 31.39} & {\bf
0}\\\hline
\end{tabular}
\end{center}
\end{table}

\subsection{Solving Group Sparse Logistic Regression with
Overlap}\label{sec:exp_pLADMPSAP} In this subsection, we apply
proximal LADMPSAP to solve the problem of group sparse logistic
regression with overlap \eqref{eq:logit}.

The Lipschitz constant of the gradient of logistic function with
respect to $\bar{\w}$ can be proven to be $L_{\bar{w}}\leq
\frac{1}{4s}\|\bar{\bx}\|_2^2$, where
$\bar{\bx}=(\bar{\x}_1,\bar{\x}_2,\cdots,\bar{\x}_s)$. Thus
\eqref{eq:logit} can be directly solved by
Algorithm~\ref{alg:general-LADMPSAP-multivar}.

\begin{table*}[t]
\caption{Comparisons among ADM, LADM, LADMPS, LADMPSAP, and
proximal LADMPSAP (pLADMPSAP) on the group sparse logistic
regression with overlap problem. The quantities include the
computing time (in seconds), number of outer iterations, and
relative errors.}\label{table-glc}
\begin{center}
\begin{tabular}{c|c||c|c|c|c}\hline
($s$, $p$, $t$, $q$) & Method & Time &  \#Iter. &
$\frac{\|\hat{\bar{\w}}-\bar{\w}^*\|}{\|\bar{\w}^*\|}$ &
$\frac{\|\hat{\z}-\z^*\|}{\|\z^*\|}$
\\\hline\hline
\multirow{6}{*}{(300, 901, 100, 10)}
& ADM       & 294.15         & 43          & 0.4800            & 0.4790                 \\
& LADM      & 229.03         &43           & 0.5331            & 0.5320              \\
& LADMPS    & 105.50         & 47          & 0.2088            & 0.2094                 \\
& LADMPSAP  & 57.46          & 39          & 0.0371            & 0.0368               \\
& pLADMPSAP & {\bf 1.97}     & 141         & {\bf 0.0112}      &
{\bf 0.0112}
\\\hline \multirow{6}{*}{(450, 1351, 150, 15)}
& ADM       & 450.96         & 33          & 0.4337            & 0.4343             \\
& LADM      & 437.12         &36           & 0.5126            & 0.5133              \\
& LADMPS    & 201.30         & 39          & 0.1938            & 0.1937            \\
& LADMPSAP  & 136.64         & 37          & 0.0321            & 0.0306            \\
& pLADMPSAP & {\bf 4.16}     & 150         & {\bf 0.0131}      &
{\bf 0.0131}
\\\hline \multirow{6}{*}{(600, 1801, 200, 20)}
& ADM       & 1617.09        & 62          & 1.4299            & 1.4365      \\
& LADM      & 1486.23        &63           & 1.5200            & 1.5279              \\
& LADMPS    & 494.52         & 46          & 0.4915            & 0.4936        \\
& LADMPSAP  & 216.45         & 32          & 0.0787            & 0.0783            \\
& pLADMPSAP & {\bf 5.77}     & 127         & {\bf 0.0276}      &
{\bf 0.0277}
\\\hline
\end{tabular}
\end{center}
\end{table*}

\subsubsection{Synthetic Data}
To assess the performance of proximal LADMPSAP, we simulate data
with $p=9t+1$ variables, covered by $t$ groups of ten variables
with overlap of one variable between two successive groups:
$\{1,\cdots,10\}$, $\{10,\cdots,19\}$, $\cdots$,
$\{p-9,\cdots,p\}$. We randomly choose $q$ groups to be the
support of $\w$. If the chosen groups have overlapping variables
with the unchosen groups, the overlapping variables are removed
from the support of $\w$. So the support of $\w$ may be less than
$10q$. $\y=(y_1,\cdots,y_s)^T$ is chosen as
$(1,-1,1,-1,\cdots)^T$. $\bx\in\mathbb{R}^{p\times s}$ is
generated as follows. For $\bx_{i,j}$, if $i$ is in the support of
$\w$ and $\y_j=1$, then $\bx_{i,j}$ is generated uniformly on
$[0.5, 1.5]$; if $i$ is in the support of $\w$ and $\y_j=-1$, then
$\bx_{i,j}$ is generated uniformly on $[-1.5, -0.5]$; if $i$ is
not in the support of $\w$, then $\bx_{i,j}$ is generated
uniformly on $[-0.5, 0.5]$. Then the rows whose indices are in the
support of $\w$ are statistically different from the remaining
rows in $\bx$, hence can be considered as informative rows. We use
model (\ref{eq:logit'}) to select the informative rows for
classification, where $\mu=0.1$. If the ground truth support of
$\w$ is recovered, then the two groups of data are linearly
separable by considering only the coordinates in the support of
$\w$.

We compare proximal LADMPSAP with a series of ADM based methods,
including ADM, LADM, LADMPS, and LADMPSAP, where the subproblems
for $\w$ and $\b$ have to be solved iteratively, e.g., by
APG~\citep{Beck2009}. We terminate the inner loop by APG when the
norm of gradient of the objective function of the subproblem is
less than $10^{-6}$. As for the outer loop, we choose
$\varepsilon_1=2\times 10^{-4}$ and $\varepsilon_2=2\times
10^{-3}$ as the thresholds to terminate the iterations.

For ADM, LADM, and LADMPS, which use a fixed penalty $\beta$, as
we do not find any suggestion on its choice in the literature (the
choice suggested in \citep{Yang-2011-LADM} is for nuclear norm
regularized least square problem only) we try multiple choices of
$\beta$ and choose the one that results in the fastest
convergence. For LADMPSAP, we set $\beta_0=0.2$ and $\rho_0=5$.
For proximal LADMPSAP we set $T_1=\frac{1}{4s}\|\bar{\bx}\|_2^2$,
$\eta_1=2.01\|\bar{\bs}\|_2^2$, $T_2=0$, $\eta_2=2.01$,
$\beta_0=1$, and $\rho_0=5$. To measure the relative errors in the
solutions we iterate proximal LADMPSAP for 2,000 times and regard
its output as the ground truth solution $(\bar{\w}^*,\z^*)$.

Table \ref{table-glc} shows the comparison among related
algorithms. The ground truth support of $\w$ is recovered by all
the compared algorithms. We can see that ADM, LADM, LADMPS, and
LADMPSAP are much slower than proximal LADMPSAP because of the
time-consuming subproblem computation, although they have much
smaller number of outer iterations. Their numerical accuracies are
also inferior to that of proximal LADMPSAP. We can also see that
LADMPSAP is faster and more numerically accurate than ADM, LADM,
and LADMPS. This again testifies to the effectiveness of using
adaptive penalty.

\begin{table*}[t]
\begin{center}
\caption{Comparisons among the Active Set
method~\citep{Jacob-2009-Group-LASSO}, LADM, LADMPSAP, and
proximal LADMPSAP (pLADMPSAP) on the pathway analysis. We present
the CPU time (in seconds), classification error rate, and number of
pathways. Results are estimated by three-fold cross validation.
\#Pathway gives the number of pathways that the selected genes
belong to in each of the cross validation.\label{table-gene}}
\begin{tabular}{c||c|c|c}\hline
Method & Time & Error & \#Pathway \\\hline\hline Active Set &
2179 & $0.36\pm0.03$ &6, 5, 78\\
LADM & 2433 & $0.315\pm0.049$ &7, 9, 10\\
LADMPSAP & 1593 & $0.329\pm0.011$ &7, 9, 9\\
pLADMPSAP &  {\bf 179} &{\bf $0.312\pm0.026$} &  4, 6, 6\\\hline
\end{tabular}
\end{center}
\end{table*}

\subsubsection{Pathway Analysis on Breast Cancer Data}
Then we consider the pathway analysis problem using the breast
cancer gene expression data set~\citep{Van-2002-Breast}, which
consists of 8141 genes in 295 breast cancer tumors (78 metastatic
and 217 non-metastatic). We follow~\citet{Jacob-2009-Group-LASSO}
and use the canonical pathways from
MSigDB~\citep{Subramanian-2005-MSigDB} to generate the overlapping
gene sets, which contains 639 groups of genes, 637 of which
involve genes from our study. The statistics of the 637 gene
groups are summarized as follows: the average number of genes in
each group is 23.7, the largest gene group has 213 genes, and 3510
genes appear in these 637 groups with an average appearance
frequency of about four. We follow~\citet{Jacob-2009-Group-LASSO}
to restrict the analysis to the 3510 genes and balance the data
set by using three replicates of each metastasis patient in the
training set. We use model (\ref{eq:logit'}) to select genes,
where $\mu=0.08$. We want to predict whether a tumor is metastatic
($y_i=1$) or non-metastatic ($y_i=-1$).

We compare proximal LADMPSAP with the active set method, which was
adopted in~\citep{Jacob-2009-Group-LASSO}\footnote{Code available
at
\url{http://cbio.ensmp.fr/~ljacob/documents/overlasso-package.tgz}.},
LADM, and LADMPSAP. In LADMPSAP and proximal LADMPSAP, we both set
$\beta_0=0.8$ and $\rho_0=1.1$. For LADM, we try multiple choices
of $\beta$ and choose the one that results in the fastest
convergence. In LADM and LADMPSAP, we terminate the inner loop by
APG when the norm of gradient of the objective function of the
subproblem is less than $10^{-6}$. The thresholds for terminating
the outer loop are all chosen as $\varepsilon_1=10^{-3}$ and
$\varepsilon_2=6\times 10^{-3}$. For the three LADM based methods,
we first solve \eqref{eq:logit'} to select genes. Then we use the
selected genes to re-train a traditional logistic regression model
and use the model to predict the test samples. As
in~\citep{Jacob-2009-Group-LASSO} we partition the whole data set
into three subsets to do the experiment three times. Each time we
select one subset as the test set and the other two as the
training set (i.e., there are $(78+217)\times 2/3=197$ samples for
training). It is worth mentioning
that~\citet{Jacob-2009-Group-LASSO} only kept the 300 genes that
are the most correlated with the output in the pre-processing
step. In contrast, we use all the 3510 genes in the training
phase.

Table \ref{table-gene} shows that proximal LADMPSAP is more than
ten times faster than the active set method used
in~\citep{Jacob-2009-Group-LASSO}, although it computes with a
more than ten times larger training set. Proximal LADMPSAP is also
much faster than LADM and LADMPSAP due to the lack of inner loop
to solve subproblems. The prediction error and the sparseness at
the pathway level by proximal LADMPSAP is also competitive with
those of other methods in comparison.

\section{Conclusions}\label{sec:con}

In this paper, we propose linearized alternating direction method
with parallel splitting and adaptive penalty (LADMPSAP) for
efficiently solving linearly constrained multi-block separable
convex programs, which are abundant in machine learning. LADMPSAP
fully utilizes the properties that the proximal operations of the
component objective functions and the projections onto convex sets
are easily solvable, which are usually satisfied by machine
learning problems, making each of its iterations cheap. It is also
highly parallel, making it appealing for parallel or distributed
computing. Numerical experiments testify to the advantages of
LADMPSAP over other possible first order methods.

Although LADMPSAP is inherently parallel, when solving the
proximal operations of component objective functions we will still
face basic numerical algebraic computations. So for particular
large scale machine learning problems, it will be interesting to
integrate the existing distributed computing techniques (e.g.,
parallel incomplete Cholesky
factorization~\citep{Chang-2007-Psvm,Chang-FoLSMM} and caching
factorization techniques~\citep{Boyd-2011-Distributed}) with our
LADMPSAP in order to effectively address the scalability issues.

\begin{acknowledgements}
Z. Lin is supported by NSFC (Nos. 61272341, 61231002, 61121002).
R. Liu is supported by NSFC (No. 61300086), the China Postdoctoral
Science Foundation (No. 2013M530917), the Fundamental Research
Funds for the Central Universities (No. DUT12RC(3)67) and the Open
Project Program of the State Key Lab of CAD\&CG (No.A1404),
Zhejiang University. Z. Lin also thanks Xiaoming Yuan, Wotao Yin,
and Edward Chang for valuable discussions and HTC for financial
support.
\end{acknowledgements}

\appendix

\section{Proof of Theorem~\ref{thm:converge_multivar}}

To prove this theorem, we first have the following lemmas and
propositions.

\begin{lemma}\label{lem:KKT}{\bf (KKT Condition)}
The Kuhn-Karush-Tucker (KKT) condition of problem
(\ref{eq:model_problem_multivar}) is that there exists
$(\x_1^*,\cdots,\x_n^*,\bflambda^*)$, such that
\begin{eqnarray}
&\sum\limits_{i=1}^n \A_i(\x_i^*)=\b,&\label{eq:KKT2}\\
&-\A_i^{\dag}(\bflambda^*)\in \partial f_i(\x_i^*),&
i=1,\cdots,n,\label{eq:KKT1}
\end{eqnarray}
where $\partial f_i$ is the subgradient of $f_i$.
\end{lemma}

The first is the feasibility condition and the second is the
duality condition. Such $(\x_1^*,\cdots,\x_n^*,\bflambda^*)$ is
called a KKT point of problem (\ref{eq:model_problem_multivar}).
\begin{lemma}\label{lem:subgradients}
For $\{(\x_1^{k},\cdots,\x_n^k,\lambda^k)\}$ generated by
Algorithm~\ref{alg:LADMPSAP-multivar}, we have that
\begin{equation}
-\sigma_i^{(k)}(\x_i^{k+1}-\u_i^{k}) \in\partial
f_i(\x_i^{k+1}),\quad
i=1,\cdots,n,\label{eq:subproblem_optimality}
\end{equation}
where $\u_i^k=\x_i^{k}-\A_i^{\dag}(\hat{\bflambda}^{k})/\sigma_i^{(k)}$.
\end{lemma}
This can be easily proved by checking the optimality conditions of
(\ref{eq:update_xi}).
\begin{lemma}\label{lem:weak-monotonicity}
For $\{(\x_1^{k},\cdots,\x_n^k,\lambda^k)\}$ generated by
Algorithm~\ref{alg:LADMPSAP-multivar} and a KKT point
$(\x_1^*,\cdots,\x_n^*,\bflambda^*)$ of problem
(\ref{eq:model_problem_multivar}), the following inequality holds:
\begin{equation}
\<-\sigma_i^{(k)}(\x_i^{k+1}-\u_i^{k})+\A_i^{\dag}(\bflambda^*),\x_i^{k+1}-\x_i^*\>\geq
0,\quad i=1,\cdots,n.\label{eq:weak-monotonicity}
\end{equation}
\end{lemma}
This can be deduced by the monotonicity of subgradient mapping~\citep{Rockafellar}.
\begin{lemma}\label{lem:basic-identity}
For $\{(\x_1^{k},\cdots,\x_n^k,\lambda^k)\}$ generated by
Algorithm~\ref{alg:LADMPSAP-multivar}  and a KKT point
$(\x_1^*,\cdots,\x_n^*,\bflambda^*)$ of problem
(\ref{eq:model_problem_multivar}), we have that
\begin{eqnarray}
&&\beta_k\sum\limits_{i=1}^n
\sigma_i^{(k)}\lbar\x_i^{k+1}-\x_i^*\rbar^2+\lbar\bflambda^{k+1}-\bflambda^*\rbar^2\label{eq:basic_id_line1}\\
&=&\beta_k\sum\limits_{i=1}^n\sigma_i^{(k)}\lbar\x_i^{k}-\x_i^*\rbar^2+\lbar\bflambda^{k}-\bflambda^*\rbar^2\\
&&-2\beta_k\sum\limits_{i=1}^n
\<\x_i^{k+1}-\x_i^*,-\sigma_i^{(k)}(\x_i^{k+1}-\u_i^k)+\A_i^{\dag}(\bflambda^*)\>\label{eq:basic_id_line3}\\
&&-\beta_k\sum\limits_{i=1}^n\sigma_i^{(k)}\lbar\x_i^{k+1}-\x_i^k\rbar^2-\lbar\bflambda^{k+1}-\bflambda^k\rbar^2\\
&&-2\beta_k\sum\limits_{i=1}^n\sigma_i^{(k)}\<\x_i^{k+1}-\x_i^*,\x_i^{k}-\u_i^k\>\label{eq:basic_id_line5}\\
&&+2\<\bflambda^{k+1}-\bflambda^k,\bflambda^{k+1}\>.\label{eq:basic-identity}
\end{eqnarray}
\end{lemma}
\begin{proof} This can be easily checked. First, we add
(\ref{eq:basic_id_line3}) and (\ref{eq:basic_id_line5}) to have
\begin{eqnarray}
\hspace{-2em}&&-2\beta_k\sum\limits_{i=1}^n
\<\x_i^{k+1}-\x_i^*,-\sigma_i^{(k)}(\x_i^{k+1}-\u_i^k)+\A_i^{\dag}(\bflambda^*)\>\\
\hspace{-2em}&&-2\beta_k\sum\limits_{i=1}^n\sigma_i^{(k)}\<\x_i^{k+1}-\x_i^*,\x_i^{k}-\u_i^k\>\\
\hspace{-2em}&=&-2\beta_k\sum\limits_{i=1}^n\<\x_i^{k+1}-\x_i^*,\A_i^{\dag}(\bflambda^*)\>+2\beta_k\sum\limits_{i=1}^n\sigma_i^{(k)}\<\x_i^{k+1}-\x_i^*,\x_i^{k+1}-\x_i^k\>\\
\hspace{-2em}&=&-2\beta_k\sum\limits_{i=1}^n\<\A_i(\x_i^{k+1}-\x_i^*),\bflambda^*\>+2\beta_k\sum\limits_{i=1}^n\sigma_i^{(k)}\<\x_i^{k+1}-\x_i^*,\x_i^{k+1}-\x_i^k\>\\
\hspace{-2em}&=&-2\<\beta_k\sum\limits_{i=1}^n\A_i(\x_i^{k+1}-\x_i^*),\bflambda^*\>+2\beta_k\sum\limits_{i=1}^n\sigma_i^{(k)}\<\x_i^{k+1}-\x_i^*,\x_i^{k+1}-\x_i^k\>\label{eq:basic_id_proof_line3}\\
\hspace{-2em}&=&-2\<\beta_k\left(\sum\limits_{i=1}^n\A_i(\x_i^{k+1})-\b\right),\bflambda^*\>\\
\hspace{-2em}&&\qquad+2\beta_k\sum\limits_{i=1}^n\sigma_i^{(k)}\<\x_i^{k+1}-\x_i^*,\x_i^{k+1}-\x_i^k\>\\
\hspace{-2em}&=&-2\<\bflambda^{k+1}-\bflambda^k,\bflambda^*\>+2\beta_k\sum\limits_{i=1}^n\sigma_i^{(k)}\<\x_i^{k+1}-\x_i^*,\x_i^{k+1}-\x_i^k\>,
\end{eqnarray}
where we have used (\ref{eq:KKT2}) in
(\ref{eq:basic_id_proof_line3}). Then we apply the identity
\begin{equation}
2\<\a_{k+1}-\a^*,\a_{k+1}-\a_k\>=\|\a_{k+1}-\a^*\|^2-\|\a_{k}-\a^*\|^2+\|\a_{k+1}-\a_k\|^2
\label{eq:inner_prod_to_norms}
\end{equation}
to see that (\ref{eq:basic_id_line1})-(\ref{eq:basic-identity})
holds.\qed
\end{proof}

\begin{proposition}\label{prop:basic-identity-multivar}
For $\{(\x_1^{k},\cdots,\x_n^k,\lambda^k)\}$ generated by
Algorithm~\ref{alg:LADMPSAP-multivar} and a KKT point
$(\x_1^*,\cdots,\x_n^*,\bflambda^*)$ of problem
(\ref{eq:model_problem_multivar}), the following inequality holds:
\begin{eqnarray}
&&\beta_k\sum\limits_{i=1}^n\sigma_i^{(k)}\|\x_i^{k+1}-\x_i^*\|^2+\|\bflambda^{k+1}-\bflambda^*\|^2 \label{eq:basic-identity-multivar-line1}\\
&\leq &\beta_k\sum\limits_{i=1}^n\sigma_i^{(k)}\|\x_i^{k}-\x_i^*\|^2+\|\bflambda^{k}-\bflambda^*\|^2\\
&&-2\beta_k\sum\limits_{i=1}^n \<\x_i^{k+1}-\x_i^*,-\sigma_i^{(k)}(\x_i^{k+1}-\u_i^k)+\A_i^{\dag}(\bflambda^*)\>\\
&&-\beta_k\sum\limits_{i=1}^n \left(\sigma_i^{(k)}-n\beta_k\|\A_i\|^2\right)\|\x_i^{k+1}-\x_i^k\|^2-\|\bflambda^k-\hat{\bflambda}^{k}\|^2.\label{eq:basic-identity-multivar}
\end{eqnarray}
\end{proposition}

\begin{proof}
We continue from
(\ref{eq:basic_id_line5})-(\ref{eq:basic-identity}). As
$\sigma_i^{(k)}(\x_i^k-\u_i^k)=\A_i^{\dag}(\hat{\bflambda}^k)$, we have
\begin{eqnarray}
&&-2\beta_k\sum\limits_{i=1}^n\sigma_i^{(k)}\<\x_i^{k+1}-\x_i^*,\x_i^{k}-\u_i^k\>+2\<\bflambda^{k+1}-\bflambda^k,\bflambda^{k+1}\>\\
&=&-2\beta_k\sum\limits_{i=1}^n\<\A_i(\x_i^{k+1}-\x_i^*),\hat{\bflambda}^k\>+2\<\bflambda^{k+1}-\bflambda^k,\bflambda^{k+1}\>\\
&=&-2\beta_k\<\sum\limits_{i=1}^n\A_i(\x_i^{k+1})-\sum\limits_{i=1}^n\A_i(\x_i^*),\hat{\bflambda}^k\>+2\<\bflambda^{k+1}-\bflambda^k,\bflambda^{k+1}\>\\
&=&-2\<\bflambda^{k+1}-\bflambda^k,\hat{\bflambda}^k\>+2\<\bflambda^{k+1}-\bflambda^k,\bflambda^{k+1}\>\\
&=&2\<\bflambda^{k+1}-\bflambda^k,\bflambda^{k+1}-\hat{\bflambda}^k\>\\
&=&\|\bflambda^{k+1}-\bflambda^k\|^2+\|\bflambda^{k+1}-\hat{\bflambda}^k\|^2-\|\bflambda^{k}-\hat{\bflambda}^k\|^2\\
&=&\|\bflambda^{k+1}-\bflambda^k\|^2+\beta_k^2\left\|\sum\limits_{i=1}^n\A_i(\x_i^{k+1}-\x_i^k)\right\|^2-\|\bflambda^{k}-\hat{\bflambda}^k\|^2\label{eq:basic-identity-multivar-line7}\\
&\leq &\|\bflambda^{k+1}-\bflambda^k\|^2+\beta_k^2\left(\sum\limits_{i=1}^n\|\A_i\|\|\x_i^{k+1}-\x_i^k\|\right)^2-\|\bflambda^{k}-\hat{\bflambda}^k\|^2\\
&\leq
&\|\bflambda^{k+1}-\bflambda^k\|^2+n\beta_k^2\sum\limits_{i=1}^n\|\A_i\|^2\|\x_i^{k+1}-\x_i^k\|^2-\|\bflambda^{k}-\hat{\bflambda}^k\|^2
\end{eqnarray}
Plugging the above into
(\ref{eq:basic_id_line5})-(\ref{eq:basic-identity}), we have
(\ref{eq:basic-identity-multivar-line1})-(\ref{eq:basic-identity-multivar}).\qed
\end{proof}

\begin{remark}\label{rem:different}
Proposition~\ref{prop:basic-identity-multivar} shows that the
sequence $\{(\x_1^k,\cdots,\x_n^k,\lambda^k)\}$ is Fej\'{e}r
monotone. Proposition~\ref{prop:basic-identity-multivar} is
different from Lemma 1 in Supplementary Material
of~\citep{Lin-2011-LADMAP} because for $n>2$ we cannot obtain an
(in)equality that is similar to Lemma 1 in Supplementary Material
of~\citep{Lin-2011-LADMAP} such that each term with minus sign
could be made non-positive. Such Fej\'{e}r monotone (in)equalities
are the corner stones for proving the convergence of Lagrange
multiplier based optimization algorithms. As a result, we cannot
prove the convergence of the naively generalized LADM for the
multi-block case.
\end{remark}

Then we have the following proposition.
\begin{proposition}\label{prop:converge_multivar}
Let $\sigma_i^{(k)}=\eta_i\beta_k$, $i=1,\cdots,n$. If
$\{\beta_k\}$ is non-decreasing, $\eta_i > n \|\mathcal{A}_i\|^2$,
$i=1,\cdots,n$, $\{(\x_1^{k},\cdots,\x_n^k,\lambda^k)\}$ is
generated by Algorithm~\ref{alg:LADMPSAP-multivar}, and
$(\x_1^*,\cdots,\x_n^*,\bflambda^*)$ is any KKT point of problem
(\ref{eq:model_problem_multivar}), then
\begin{enumerate}
\item[1)]
$\left\{\sum\limits_{i=1}^n\eta_i\|\mathbf{x}_i^k-\mathbf{x}_i^*\|^2
+\beta_k^{-2}\|\bflambda^{k}-\bflambda^*\|^2\right\}$ is
nonnegative and non-increasing. \item[2)]
$\|\mathbf{x}_i^{k+1}-\mathbf{x}_i^k\|\to 0$, $i=1,\cdots,n$, and
$\beta_k^{-1}\|\bflambda^{k}-\hat{\bflambda}^k\|\to 0$. \item[3)]
$\sum\limits_{k=1}^{+\infty}\beta_k^{-1}\<\x_i^{k+1}-\x_i^*,-\sigma_i^{(k)}(\x_i^{k+1}-\u_i^k)+\A_i^{\dag}(\bflambda^*)\>
<+\infty$, $i=1,\cdots,n$.
\end{enumerate}
\end{proposition}

\begin{proof} We divide both sides of (\ref{eq:basic-identity-multivar-line1})-(\ref{eq:basic-identity-multivar}) by
$\beta_k^2$ to have
\begin{eqnarray}
&&\sum\limits_{i=1}^n\eta_i\|\x_i^{k+1}-\x_i^*\|^2+\beta_k^{-2}\|\bflambda^{k+1}-\bflambda^*\|^2 \label{eq:basic-identity-multivar-line1'}\\
&\leq &\sum\limits_{i=1}^n\eta_i\|\x_i^{k}-\x_i^*\|^2+\beta_k^{-2}\|\bflambda^{k}-\bflambda^*\|^2\\
&&-2\beta_k^{-1}\sum\limits_{i=1}^n \<\x_i^{k+1}-\x_i^*,-\sigma_i^{(k)}(\x_i^{k+1}-\u_i^k)+\A_i^{\dag}(\bflambda^*)\>\\
&&-\sum\limits_{i=1}^n \left(\eta_i-n\|\A_i\|^2\right)\|\x_i^{k+1}-\x_i^k\|^2\\
&&-\beta_k^{-2}\|\bflambda^k-\hat{\bflambda}^{k}\|^2.\label{eq:basic-identity-multivar'}
\end{eqnarray}
Then by (\ref{eq:weak-monotonicity}), $\eta_i > n
\|\mathcal{A}_i\|^2$ and the non-decrement of $\{\beta_k\}$, we
can easily obtain 1). Second, we sum both sides of
(\ref{eq:basic-identity-multivar-line1'})-(\ref{eq:basic-identity-multivar'})
over $k$ to have
\begin{eqnarray}
&&2\sum\limits_{k=0}^{+\infty}\beta_k^{-1}\sum\limits_{i=1}^n \<\x_i^{k+1}-\x^*,-\sigma_i^{(k)}(\x_i^{k+1}-\u_i^k)+\A_i^{\dag}(\bflambda^*)\>\\
&&+\sum\limits_{i=1}^n \left(\eta_i-n\|\A_i\|^2\right)\sum\limits_{k=0}^{+\infty}\|\x_i^{k+1}-\x_i^k\|^2\\
&&+\sum\limits_{k=0}^{+\infty}\beta_k^{-2}\|\bflambda^k-\hat{\bflambda}^{k}\|^2\\
&\leq
&\sum\limits_{i=1}^n\eta_i\|\x_i^{0}-\x^*\|^2+\beta_0^{-2}\|\bflambda^{0}-\bflambda^*\|^2.
\end{eqnarray}
Then 2) and 3) can be easily deduced.\qed
\end{proof}

Now we are ready to prove Theorem~\ref{thm:converge_multivar}. The
proof resembles that in \citep{Lin-2011-LADMAP}.
\begin{proof} {\bf (of Theorem~\ref{thm:converge_multivar})}
By Proposition~\ref{prop:converge_multivar}-1) and the boundedness
of $\{\beta_k\}$,
$\{(\mathbf{x}_1^k,\cdots,\mathbf{x}_n^k,\bflambda^k)\}$ is
bounded, hence has an accumulation point, say $
(\mathbf{x}_1^{k_j},\cdots,\mathbf{x}_n^{k_j},\bflambda^{k_j}) \to
(\mathbf{x}_1^{\infty},\cdots,\mathbf{x}_n^{\infty},\bflambda^{\infty})
$. We accomplish the proof in two steps.

\textbf{1.} We first prove that
$(\mathbf{x}_1^{\infty},\cdots,\mathbf{x}_n^{\infty},\bflambda^{\infty})$
is a KKT point of problem (\ref{eq:model_problem_multivar}).

By Proposition~\ref{prop:converge_multivar}-2), $$
\sum\limits_{i=1}^n\mathcal{A}_i(\mathbf{x}_i^{k})-\mathbf{b} =
\beta_k^{-1}(\hat{\bflambda}^{k}-\bflambda^k) \to 0.$$ So any
accumulation point of
$\{(\mathbf{x}_1^{k},\cdots,\mathbf{x}_n^{k})\}$ is a feasible
solution.

Since $-\sigma_i^{(k_j-1)}(\x_i^{k_j}-\u_i^{k_j-1}) \in\partial
f_i(\x_i^{k_j})$, we have
\begin{eqnarray*}
\begin{array}{rl}
\sum\limits_{i=1}^n f_i(\mathbf{x}_i^{k_j})&\leq
\sum\limits_{i=1}^n f_i(\mathbf{x}_i^*)
+\sum\limits_{i=1}^n \< \mathbf{x}_i^{k_j} - \mathbf{x}_i^*, -\sigma_i^{(k_j-1)}(\x_i^{k_j}-\u_i^{k_j-1}) \> \\
&= \sum\limits_{i=1}^n f_i(\mathbf{x}_i^*) +\sum\limits_{i=1}^n \<
\mathbf{x}_i^{k_j} - \mathbf{x}_i^*,
-\eta_i\beta_{k_j-1}(\x_i^{k_j}-\x_i^{k_j-1})-\A_i^{\dag}(\hat{\bflambda}^{k_j-1})
\>.
\end{array}
\end{eqnarray*}
Let $j\to +\infty$. By observing
Proposition~\ref{prop:converge_multivar}-2) and the boundedness of
$\{\beta_k\}$, we have
\begin{eqnarray*}
\begin{array}{rl}
\sum\limits_{i=1}^n f_i(\mathbf{x}_i^{\infty})&\leq
\sum\limits_{i=1}^n f_i(\mathbf{x}_i^*)
+\sum\limits_{i=1}^n \< \mathbf{x}_i^{\infty} - \mathbf{x}_i^*,-\A_i^{\dag}(\bflambda^\infty) \> \\
&= \sum\limits_{i=1}^n f_i(\mathbf{x}_i^*)
-\sum\limits_{i=1}^n \< \A(\mathbf{x}_i^{\infty} - \mathbf{x}_i^*), \bflambda^\infty \> \\
&= \sum\limits_{i=1}^n f_i(\mathbf{x}_i^*)
-\< \sum\limits_{i=1}^n \A(\mathbf{x}_i^{\infty}) - \b, \bflambda^\infty \> \\
&= \sum\limits_{i=1}^n f_i(\mathbf{x}_i^*).
\end{array}
\end{eqnarray*}
So we conclude that
$(\mathbf{x}_1^{\infty},\cdots,\mathbf{x}_n^{\infty})$ is an
optimal solution to (\ref{eq:model_problem_multivar}).

Again by $-\sigma_i^{(k_j-1)}(\x_i^{k_j}-\u_i^{k_j-1}) \in\partial
f_i(\x_i^{k_j})$ we have
\begin{eqnarray*}
\begin{array}{rl}
f_i(\mathbf{x})&\geq f_i(\mathbf{x}_i^{k_j})
+\<\mathbf{x}-\mathbf{x}_i^{k_j},-\sigma_i^{(k_j-1)}(\mathbf{x}_i^{k_j}-\mathbf{u}_i^{k_j-1}) \>\\
&= f_i(\mathbf{x}_i^{k_j})
+\<\mathbf{x}-\mathbf{x}_i^{k_j},-\eta_i\beta_{k_j-1}(\mathbf{x}_i^{k_j}-\mathbf{x}_i^{k_j-1})-\A_i^{\dag}(\hat{\bflambda}^{k_j-1})
\>.
\end{array}
\end{eqnarray*}
Fixing $\mathbf{x}$ and letting $j\to +\infty$, we see that $$
f_i(\mathbf{x}) \geq f_i(\mathbf{x}_i^{\infty}) + \<
\mathbf{x}-\mathbf{x}_i^{\infty},
-\mathcal{A}_i^{\dag}(\bflambda^{\infty})\>,\quad\forall\mathbf{x}.$$
So $-\mathcal{A}_i^{\dag}(\bflambda^{\infty}) \in \partial
f_i(\mathbf{x}_i^{\infty})$, $i=1,\cdots,n$. Thus
$(\x_1^\infty,\cdots,\x_n^\infty,\bflambda^\infty)$ is a KKT point
of problem (\ref{eq:model_problem_multivar}).

\textbf{2.} We next prove that the whole sequence
$\{(\mathbf{x}_1^k,\cdots,\mathbf{x}_n^k,\bflambda^k)\}$ converges
to
$(\mathbf{x}_1^{\infty},\cdots,\mathbf{x}_n^{\infty},\bflambda^{\infty})$.

By choosing
$(\mathbf{x}_1^*,\cdots,\mathbf{x}_n^*,\bflambda^*)=(\mathbf{x}_1^{\infty},\cdots,\mathbf{x}_n^{\infty},\bflambda^{\infty})$
in Proposition~\ref{prop:converge_multivar}, we have
$$
\sum\limits_{i=1}^n\eta_i\|\mathbf{x}_i^{k_j}-\mathbf{x}_i^{\infty}\|^2
+\beta_{k_j}^{-2}\|\bflambda^{k_j}-\bflambda^{\infty}\|^2 \to 0.
$$ By Proposition~\ref{prop:converge_multivar}-1), we readily have
$$
\sum\limits_{i=1}^n\eta_i\|\mathbf{x}_i^{k}-\mathbf{x}_i^{\infty}\|^2
+\beta_{k}^{-2}\|\bflambda^{k}-\bflambda^{\infty}\|^2 \to 0. $$ So
$(\mathbf{x}_1^k,\cdots,\mathbf{x}_n^k,\bflambda^k) \to
(\mathbf{x}_1^{\infty},\cdots,\mathbf{x}_n^{\infty},\bflambda^{\infty})$.

As
$(\mathbf{x}_1^{\infty},\cdots,\mathbf{x}_n^{\infty},\bflambda^{\infty})$
can be an arbitrary accumulation point of
$\{(\mathbf{x}_1^k,\cdots,\mathbf{x}_n^k,\bflambda^k)\}$, we
conclude that
$\{(\mathbf{x}_1^k,\cdots,\mathbf{x}_n^k,\bflambda^k)\}$ converge
to a KKT point of problem (\ref{eq:model_problem_multivar}).\qed
\end{proof}

\section{Proof of Theorem~\ref{thm:convergence_unbounded}}
We first have the following proposition.
\begin{proposition}\label{prop:unbounded_beta}
If $\{\beta_k\}$ is non-decreasing and unbounded, $\eta_i >
n\|\A_i\|^2$ and $\partial f_i(\x)$ is bounded for $i=1,\cdots,n$, then
Proposition~\ref{prop:converge_multivar} holds and
\begin{equation}
\beta_k^{-1} \bflambda^k \rightarrow
0.\label{eq:lambda_k_unbounded}
\end{equation}
\end{proposition}

\begin{proof} As the conditions here are stricter than those in Proposition~\ref{prop:converge_multivar}, Proposition~\ref{prop:converge_multivar} holds. Then we have that
$\{\beta_k^{-1}\|\bflambda^k - \bflambda^*\|\}$ is bounded due to
Proposition~\ref{prop:converge_multivar}-1). So
$\{\beta_k^{-1}\bflambda^k\}$ is bounded due to
$\beta_k^{-1}\|\bflambda^k\|\leq \beta_k^{-1}\|\bflambda^k -
\bflambda^*\|+\beta_k^{-1}\|\bflambda^*\|$.
$\{\beta_k^{-1}\hat{\bflambda}^k\}$ is also bounded thanks to
Proposition~\ref{prop:converge_multivar}-2).

We rewrite Lemma~\ref{lem:subgradients} as
\begin{equation}
-\eta_i(\x_i^{k+1}-\x_i^{k})-\A_i^{\dag}(\beta_k^{-1}\hat{\bflambda}^k)
\in\beta_k^{-1}\partial f_i(\x_i^{k+1}), \quad
i=1,\cdots,n.\label{eq:subproblem_optimality1}
\end{equation}
Then by the boundedness of $\partial f_i(x)$, the unboundedness of
$\{\beta_k\}$ and Proposition~\ref{prop:converge_multivar}-2),
letting $k\rightarrow +\infty$, we have that
\begin{equation}
\A_i^{\dag}(\check{\bflambda}^\infty)= 0, \quad
i=1,\cdots,n.\label{eq:subproblem_optimality2}
\end{equation}
where $\check{\bflambda}^\infty$ is any accumulation point of
$\{\beta_k^{-1}\hat{\bflambda}^k\}$, which is the same as that of
$\{\beta_k^{-1}\bflambda^k\}$ due to
Proposition~\ref{prop:converge_multivar}-2).

Recall that we have assumed that the mapping
$\A(\x_1,\cdots,\x_n)\equiv\sum\limits_{i=1}^n\A_i(\x_i)$ is onto.
So $\cap_{i=1}^n null(\A_i^\dag)=0$. Therefore by
(\ref{eq:subproblem_optimality2}),
$\check{\bflambda}^\infty=0$.\qed
\end{proof}

Based on Proposition~\ref{prop:unbounded_beta}, we can prove Theorem~\ref{thm:convergence_unbounded} as follows.
\begin{proof} {\bf (of Theorem~\ref{thm:convergence_unbounded})} When $\{\beta_k\}$ is bounded, the convergence has been proven in Theorem 1. In the following, we only focus on the case that $\{\beta_k\}$ is unbounded.

By Proposition~\ref{prop:converge_multivar}-1),
$\{(\x_1^k,\cdots,\x_n^k)\}$ is bounded, hence has at least one
accumulation point $(\x_1^\infty,\cdots,\x_n^\infty)$. By
Proposition~\ref{prop:converge_multivar}-2),
$(\x_1^\infty,\cdots,\x_n^\infty)$ is a feasible solution.

Since $\sum\limits_{k=1}^{+\infty}\beta_k^{-1}=+\infty$ and
Proposition~\ref{prop:converge_multivar}-3), there exists a
subsequence $\{(\x_1^{k_j},\cdots,\x_n^{k_j})\}$ such that
\begin{equation}
\<\x_i^{k_j}-\x_i^*,-\sigma_i^{(k_j-1)}(\x_i^{k_j}-\u_i^{k_j-1})+\A_i^{\dag}(\bflambda^*)\>
\rightarrow 0,\quad i=1,\cdots,n.\label{eq:subsequence_to_zero}
\end{equation}
As $\p_i^{k_j}\equiv
-\sigma_i^{(k_j-1)}(\x_i^{k_j}-\u_i^{k_j-1})\in\partial
f_i(\x_i^{k_j})$ and $\partial f_i$ is bounded, we may assume that
$$\x_i^{k_j}\rightarrow \x_i^\infty\quad \mbox{and}\quad \p_i^{k_j}\rightarrow\p_i^\infty.$$
It can be easily proven that
$$\p_i^\infty\in\partial f_i(\x_i^\infty).$$
Then letting $j\rightarrow \infty$ in
(\ref{eq:subsequence_to_zero}), we have
\begin{equation}
\<\x_i^{\infty}-\x_i^*,\p_i^\infty+\A_i^{\dag}(\bflambda^*)\> = 0,\quad
i=1,\cdots,n.\label{eq:equal_zero}
\end{equation}
Then by $\p_i^{k_j}\in\partial f_i(\x_i^{k_j})$,
\begin{equation}
\sum\limits_{i=1}^n f_i(\x_i^{k_j}) \leq \sum\limits_{i=1}^n
f_i(\x_i^*)+\sum\limits_{i=1}^n \<\x_i^{k_j}-\x_i^*,\p_i^{k_j}\>.
\end{equation}
Letting $j\rightarrow \infty$ and making use of
(\ref{eq:equal_zero}), we have
\begin{eqnarray}
\begin{array}{rl}
\sum\limits_{i=1}^n f_i(\x_i^{\infty}) &\leq \sum\limits_{i=1}^n f_i(\x_i^*)+\sum\limits_{i=1}^n \<\x_i^{\infty}-\x_i^*,\p_i^\infty\>\\
&= \sum\limits_{i=1}^n f_i(\x_i^*)-\sum\limits_{i=1}^n \<\x_i^{\infty}-\x_i^*,\A_i^{\dag}(\bflambda^*)\>\\
&= \sum\limits_{i=1}^n f_i(\x_i^*)-\sum\limits_{i=1}^n \<\A_i(\x_i^{\infty}-\x_i^*),\bflambda^*\>\\
&=\sum\limits_{i=1}^n f_i(\x_i^*).
\end{array}
\end{eqnarray}
So together with the feasibility of
$\{(\x_1^\infty,\cdots,\x_n^\infty)\}$ we have that
$\{(\x_1^{k_j},\cdots,\x_n^{k_j})\}$ converges to an optimal
solution $\{(\x_1^\infty,\cdots,\x_n^\infty)\}$ to
(\ref{eq:model_problem_multivar}).

Finally, we set $\x_i^*=\x_i^\infty$ and $\bflambda^*$ be the
corresponding Lagrange multiplier $\bflambda^\infty$ in
Proposition~\ref{prop:converge_multivar}. By
Proposition~\ref{prop:unbounded_beta}, we have that
$$
\sum\limits_{i=1}^n\eta_i\|\mathbf{x}_i^{k_j}-\mathbf{x}_i^{\infty}\|^2
+\beta_{k_j}^{-2}\|\bflambda^{k_j}-\bflambda^{\infty}\|^2 \to 0.
$$ By Proposition~\ref{prop:converge_multivar}-1), we readily have
$$
\sum\limits_{i=1}^n\eta_i\|\mathbf{x}_i^{k}-\mathbf{x}_i^{\infty}\|^2
+\beta_{k}^{-2}\|\bflambda^{k}-\bflambda^{\infty}\|^2 \to 0. $$ So
$(\mathbf{x}_1^k,\cdots,\mathbf{x}_n^k) \to
(\mathbf{x}_1^{\infty},\cdots,\mathbf{x}_n^{\infty})$.\qed
\end{proof}

\section{Proof of Theorem~\ref{thm:convergence_unbounded_necessary}}
\begin{proof}{\bf(of Theorem~\ref{thm:convergence_unbounded_necessary})} We first prove
that there exist linear mappings $\B_i$, $i=1,\cdots,n$, such that
$\B_i$'s are not all zeros and $\sum\limits_{i=1}^n\B_i\A_i^{\dag}=0$.
Indeed, $\sum\limits_{i=1}^n\B_i\A_i^{\dag}=0$ is equivalent to
\begin{equation}
\sum\limits_{i=1}^n \mathbf{B}_i
\mathbf{A}_i^T=0,\label{eq:linear_mappings_equiv}
\end{equation}
where $\mathbf{A}_i$ and $\mathbf{B}_i$ are the matrix
representations of $\A_i$ and $\B_i$, respectively.
(\ref{eq:linear_mappings_equiv}) can be further written as
\begin{equation}
(\mathbf{A}_1 \,\,\cdots\,\, \mathbf{A}_n)\left(
\begin{array}{c}
\mathbf{B}_1^T\\
\vdots \\
\mathbf{B}_n^T
\end{array}
\right)=0. \label{eq:linear_mappings_equiv'}
\end{equation}
Recall that we have assumed that the solution to
$\sum\limits_{i=1}^n \A_i(\x_i)=\b$ is non-unique. So
$(\mathbf{A}_1 \,\,\cdots\,\, \mathbf{A}_n)$ is not full column
rank hence (\ref{eq:linear_mappings_equiv'}) has nonzero
solutions. Thus there exist $\B_i$'s such that they are not all
zeros and $\sum\limits_{i=1}^n\B_i\A_i^{\dag}=0$.

By Lemma~\ref{lem:subgradients},
\begin{equation}
-\sigma_i^{(k)}(\x_i^{k+1}-\u_i^{k})\in\partial
f_i(\x_i^{k+1}),\quad i=1,\cdots,n.
\end{equation}
As $\partial f_i$ is bounded, $i=1,\cdots,n$, so is
\begin{equation}
\sum\limits_{i=1}^n \B_i(\sigma_i^{(k)}(\x_i^{k+1}-\u_i^{k}))
=\beta_k(\v^{k+1}-\v^{k}),\label{eq:cancel_lambda}
\end{equation}
where $\v^{k}=\phi(\x_1^k,\cdots,\x_n^k)$ and
\begin{equation}
\phi(\x_1,\cdots,\x_n)=\sum\limits_{i=1}^n\eta_i\B_i(\x_i).
\end{equation}
In (\ref{eq:cancel_lambda}) we have utilized
$\sum\limits_{i=1}^n\B_i\A_i^{\dag}=0$ to cancel $\hat{\lambda}^k$,
whose boundedness is uncertain.

Then we have that there exists a constant $C>0$ such that
\begin{equation}
\|\v^{k+1}-\v^{k}\|\leq C\beta_k^{-1}.
\end{equation}

If $\sum\limits_{k=1}^{+\infty}\beta_k^{-1}<+\infty$, then
$\{\v^k\}$ is a Cauchy sequence, hence has a limit $\v^\infty$.
Define $ \v^*=\phi(\x_1^*,\cdots,\x_n^*), $ where
$(\x_1^*,\cdots,\x_n^*)$ is any optimal solution. Then
\begin{eqnarray}
\|\v^\infty - \v^*\| &=& \left\|\v^0 + \sum\limits_{k=0}^\infty
(\v^{k+1}-\v^k) - \v^*\right\| \\
&\geq & \|\v^0  - \v^*\| - \sum\limits_{k=0}^\infty
\|\v^{k+1}-\v^k\| \\
&\geq &\|\v^0 - \v^*\| - C\sum\limits_{k=0}^\infty \beta_k^{-1}.
\end{eqnarray}
So if $(\x_1^0,\cdots,\x_n^0)$ is initialized badly such that
\begin{equation}
\|\v^0 - \v^*\| > C\sum\limits_{k=0}^\infty
\beta_k^{-1},\label{eq:bad_init}
\end{equation}
then $\|\v^\infty - \v^*\|>0$, which implies that
$(\x_1^k,\cdots,\x_n^k)$ cannot converge to
$(\x_1^*,\cdots,\x_n^*)$. Note that (\ref{eq:bad_init}) is
possible because $\phi$ is not a zero mapping given the conditions
on $\B_i$.\qed
\end{proof}

\section{Proofs of Proposition~\ref{prop:optimality} and Theorem~\ref{thm:convergence_rate_2var}}

\begin{proof}{\bf (of Proposition~\ref{prop:optimality})} If $\tilde{\x}$ is optimal, it is easy to check that
\eqref{eq:constrained_optimality} holds.

Since $-\A_i^{\dag}(\lambda^*)\in\partial f_i(\x_i^*)$, we have
$$f(\tilde{\x})-f(\x^*)+\sum\limits_{i=1}^n\<\A_i^{\dag}(\bflambda^*),\tilde{\x}_i-\x_i^*\> \geq 0.$$
So if (\ref{eq:constrained_optimality}) holds, we have
\begin{eqnarray}
&f(\tilde{\x})-f(\x^*)+\sum\limits_{i=1}^n\<\A_i^{\dag}(\bflambda^*),\tilde{\x}_i-\x_i^*\>=0,&\label{eq:constrained_optimality1}\\
&\sum\limits_{i=1}^n
\A_i(\tilde{\x}_i)-\b=0.&\label{eq:constrained_optimality2}
\end{eqnarray}
With (\ref{eq:constrained_optimality2}), we have
\begin{equation}
\sum\limits_{i=1}^n\<\A_i^{\dag}(\bflambda^*),\tilde{\x}_i-\x_i^*\>=\sum\limits_{i=1}^n\<\bflambda^*,\A_i(\tilde{\x}_i-\x_i^*)\>
=\<\bflambda^*,\sum\limits_{i=1}^n\A_i(\tilde{\x}_i-\x_i^*)\>=0.
\end{equation}
So (\ref{eq:constrained_optimality1}) reduces to
$f(\tilde{\x})=f(\x^*)$. As $\tilde{\x}$ satisfies the feasibility
condition, it is an optimal solution to
(\ref{eq:model_problem_multivar}).\qed
\end{proof}

\begin{proof}{\bf(of Theorem~\ref{thm:convergence_rate_2var})} We first deduce
\begin{eqnarray}
\begin{array}{rl}
&\lbar\sum\limits_{i=1}^n\A_i(\x_i^{k+1})-\b\rbar^2\\
=&\lbar\sum\limits_{i=1}^n\A_i(\x_i^{k})-\b+\sum\limits_{i=1}^n\A_i(\x_i^{k+1}-\x_i^{k})\rbar^2\\
\leq&\left(\lbar\sum\limits_{i=1}^n\A_i(\x_i^{k})-\b\rbar+\sum\limits_{i=1}^n\lbar\A_i(\x_i^{k+1}-\x_i^{k})\rbar\right)^2\\
\leq&(n+1)\left(\lbar\sum\limits_{i=1}^n\A_i(\x_i^{k})-\b\rbar^2+\sum\limits_{i=1}^n\lbar\A_i\rbar^2\lbar\x_i^{k+1}-\x_i^{k}\rbar^2\right)\\
\leq&(n+1)\left(\beta_k^{-2}\lbar\bflambda_k-\hat{\bflambda}_k\rbar^2
+\max\left\{\dfrac{\lbar\A_i\rbar^2}{\eta_i-n\lbar\A_i\rbar^2}\right\}\sum\limits_{i=1}^n\left(\eta_i-n\lbar\A_i\rbar^2\right)\lbar\x_i^{k+1}-\x_i^{k}\rbar^2\right)\\
\leq&(n+1)\max\left\{1,\left\{\dfrac{\lbar\A_i\rbar^2}{\eta_i-n\lbar\A_i\rbar^2}\right\}\right\}\left(\beta_k^{-2}\lbar\bflambda_k-\hat{\bflambda}_k\rbar^2
+\sum\limits_{i=1}^n\left(\eta_i-n\lbar\A_i\rbar^2\right)\lbar\x_i^{k+1}-\x_i^{k}\rbar^2\right)\\
=&\alpha^{-1}\left(\beta_k^{-2}\lbar\bflambda_k-\hat{\bflambda}_k\rbar^2
+\sum\limits_{i=1}^n\left(\eta_i-n\lbar\A_i\rbar^2\right)\lbar\x_i^{k+1}-\x_i^{k}\rbar^2\right).
\end{array}\label{eq:bounding_A(x^{k+1})}
\end{eqnarray}
By Proposition~\ref{prop:basic-identity-multivar}, we have
\begin{equation}
\begin{array}{rl}
&\beta_k^{-1}\sum\limits_{i=1}^n
\<\x_i^{k+1}-\x_i^*,-\sigma_i^{(k)}(\x_i^{k+1}-\u_i^k)+\A_i^{\dag}(\bflambda^*)\>\\
&+\dfrac{1}{2}\sum\limits_{i=1}^n\left(\eta_i-n\lbar\A_i\rbar^2\right)\lbar\x_i^{k+1}-\x_i^{k}\rbar^2+\dfrac{1}{2}\beta_k^{-2}\|\bflambda^{k}-\hat{\bflambda}^k\|^2\\
\leq&
\dfrac{1}{2}\left(\sum\limits_{i=1}^n\eta_i\lbar\x_i^{k}-\x_i^*\rbar^2+\beta_k^{-2}\lbar\bflambda^{k}-\bflambda^*\rbar^2\right)\\
&-\dfrac{1}{2}\left(\sum\limits_{i=1}^n\eta_i\lbar\x_i^{k+1}-\x_i^*\rbar^2+\beta_{k+1}^{-2}\lbar\bflambda^{k+1}-\bflambda^*\rbar^2\right).
\label{eq:recursive1}
\end{array}
\end{equation}
So by Lemma~\ref{lem:subgradients} and combining the above
inequalities, we have
\begin{equation}
\begin{array}{rl}
&\beta_k^{-1}\left( f(\x^{k+1})-f(\x^*)+\sum\limits_{i=1}^n\<\x_i^{k+1}-\x_i^*,\A_i^{\dag}(\bflambda^*)\>+\dfrac{\alpha\beta_0}{2}\lbar\sum\limits_{i=1}^n\A_i(\x_i^{k+1})-\b\rbar^2\right)\\
\leq &\beta_k^{-1}\left(\sum\limits_{i=1}^n
\<\x_i^{k+1}-\x_i^*,-\sigma_i^{(k)}(\x_i^{k+1}-\u_i^k)\> +
\sum\limits_{i=1}^n
\<\x_i^{k+1}-\x_i^*,\A_i^\dag(\bflambda^*)\>\right)\\
&+\dfrac{\alpha}{2}\lbar\sum\limits_{i=1}^n\A_i(\x_i^{k+1})-\b\rbar^2\\
\leq &\beta_k^{-1}\sum\limits_{i=1}^n
\<\x_i^{k+1}-\x_i^*,-\sigma_i^{(k)}(\x_i^{k+1}-\u_i^k)+\A_i^{\dag}(\bflambda^*)\>\\
&+\dfrac{1}{2}\beta_k^{-2}\lbar\bflambda_k-\hat{\bflambda}_k\rbar^2
+\dfrac{1}{2}\sum\limits_{i=1}^n\left(\eta_i-n\lbar\A_i\rbar^2\right)\lbar\x_i^{k+1}-\x_i^{k}\rbar^2\\
\leq & \dfrac{1}{2}\left(\sum\limits_{i=1}^n\eta_i\lbar\x_i^{k}-\x_i^*\rbar^2+\beta_k^{-2}\lbar\bflambda^{k}-\bflambda^*\rbar^2\right)\\
&-\dfrac{1}{2}\left(\sum\limits_{i=1}^n\eta_i\lbar\x_i^{k+1}-\x_i^*\rbar^2+\beta_{k+1}^{-2}\lbar\bflambda^{k+1}-\bflambda^*\rbar^2\right).
\label{eq:recursive3}
\end{array}
\end{equation}
Here we use the fact that $\beta_k \geq \beta_0$, which is guaranteed by \eqref{eq:update_beta_multivar} and \eqref{eq:update_rho}.
Summing the above inequalities from $k=0$ to $K$, and dividing
both sides with $\sum\limits_{k=0}^K \beta_k^{-1}$, we have
\begin{eqnarray}
&&\sum\limits_{k=0}^K \gamma_k f(\x^{k+1}) - f(\x^*)
+\sum\limits_{i=1}^n\<\sum\limits_{k=0}^K \gamma_k\x_i^{k+1}-\x_i^*,\A_i^{\dag}(\bflambda^*)\>\label{eq:recursive4-line1}\\
&&+ \dfrac{\alpha\beta_0}{2}\sum\limits_{k=0}^K \gamma_k \lbar\sum\limits_{i=1}^n\A_i(\x_i^{k+1})-\b\rbar^2\\
&\leq& \dfrac{1}{2\sum\limits_{k=0}^K
\beta_k^{-1}}\left(\sum\limits_{i=1}^n\eta_i\lbar\x_i^{0}-\x_i^*\rbar^2+\beta_0^{-2}\lbar\bflambda^{0}-\bflambda^*\rbar^2\right).
\label{eq:recursive4}
\end{eqnarray}
Next, by the convexity of $f$ and the squared Frobenius norm
$\|\cdot\|^2$, we have
\begin{eqnarray}
&&f(\bar{\x}^K) - f(\x^*)
+\sum\limits_{i=1}^n\<\bar{\x}_i^{K}-\x_i^*,\A_i^{\dag}(\bflambda^*)\>+\dfrac{\alpha\beta_0}{2} \lbar\sum\limits_{i=1}^n\A_i(\bar{\x}_i^{K})-\b\rbar^2\label{eq:recursive5-line1}\\
&\leq& \sum\limits_{k=0}^K \gamma_k f(\x^{k+1}) - f(\x^*)
+\sum\limits_{i=1}^n\<\sum\limits_{k=0}^K \gamma_k\x_i^{k+1}-\x_i^*,\A_i^{\dag}(\bflambda^*)\>\\
&&+ \dfrac{\alpha\beta_0}{2}\sum\limits_{k=0}^K \gamma_k
\lbar\sum\limits_{i=1}^n\A_i(\x_i^{k+1})-\b\rbar^2.
\label{eq:recursive5}
\end{eqnarray}
Combining (\ref{eq:recursive4-line1})-(\ref{eq:recursive4}) and
(\ref{eq:recursive5-line1})-(\ref{eq:recursive5}), we have
\begin{eqnarray}
&&f(\bar{\x}^K) - f(\x^*)
+\sum\limits_{i=1}^n\<\bar{\x}_i^{K}-\x_i^*,\A_i^{\dag}(\bflambda^*)\>+\dfrac{\alpha\beta_0}{2} \lbar\sum\limits_{i=1}^n\A_i(\bar{\x}_i^{K})-\b\rbar^2\\
&\leq& \dfrac{1}{2\sum\limits_{k=0}^K
\beta_k^{-1}}\left(\sum\limits_{i=1}^n\eta_i\lbar\x_i^{0}-\x_i^*\rbar^2+\beta_0^{-2}\lbar\bflambda^{0}-\bflambda^*\rbar^2\right).
\label{eq:recursive6}
\end{eqnarray}\qed
\end{proof}

\section{Proof of Theorem~\ref{thm:better_eta}}
We only need to prove the following proposition. Then by the same
technique for proving Theorem~\ref{thm:converge_multivar}, we can prove Theorem~\ref{thm:better_eta}.
\begin{proposition}\label{prop:basic-identity-multivar-equiv}
For $\{(\x_1^{k},\cdots, \x_{2n}^k,\lambda^k)\}$ generated by
Algorithm~\ref{alg:LADMPSAP-multivar_equiv} and a KKT point
$(\x_1^*,\cdots,\x_{2n}^*,\bflambda^*)$ of problem
(\ref{eq:model_problem_multivar_equiv}), we have that
\begin{eqnarray}
&&\beta_k\sum\limits_{i=1}^{2n}\sigma_i^{(k)}\|\x_i^{k+1}-\x_i^*\|^2+\|\bflambda^{k+1}-\bflambda^*\|^2 \label{eq:basic-identity-multivar-equiv-line1}\\
&\leq &\beta_k\sum\limits_{i=1}^n\sigma_i^{(k)}\|\x_i^{k}-\x_i^*\|^2+\|\bflambda^{k}-\bflambda^*\|^2\\
&&-2\beta_k\sum\limits_{i=1}^{2n} \<\x_i^{k+1}-\x_i^*,-\sigma_i^{(k)}(\x_i^{k+1}-\u_i^k)+\hat{\A}_i^{\dag}(\bflambda^*)\>\\
&&-\beta_k\sum\limits_{i=1}^{n} \left(\sigma_i^{(k)}-\beta_k(n\|\A_i\|^2+2)\right)\|\x_i^{k+1}-\x_i^k\|^2\\
&&-\beta_k\sum\limits_{i=n+1}^{2n} \left(\sigma_i^{(k)}-2\beta_k\right)\|\x_i^{k+1}-\x_i^k\|^2\\
&&-\|\bflambda^k-\hat{\bflambda}^{k}\|^2.\label{eq:basic-identity-multivar-equiv}
\end{eqnarray}
\end{proposition}
\begin{proof}
We continue from (\ref{eq:basic-identity-multivar-line7}):
\begin{eqnarray}
&&-2\beta_k\sum\limits_{i=1}^{2n}\sigma_i^{(k)}\<\x_i^{k+1}-\x_i^*,\x_i^{k}-\u_i^k\>+2\<\bflambda^{k+1}-\bflambda^k,\bflambda^{k+1}\>\\
&=&\|\bflambda^{k+1}-\bflambda^k\|^2+\beta_k^2\left\|\sum\limits_{i=1}^{2n}\hat{\A}_i(\x_i^{k+1}-\x_i^k)\right\|^2-\|\bflambda^{k}-\hat{\bflambda}^k\|^2\\
&=&\|\bflambda^{k+1}-\bflambda^k\|^2+\beta_k^2\left\|\sum\limits_{i=1}^{n}\A_i(\x_i^{k+1}-\x_i^k)\right\|^2\\
&&+\beta_k^2\sum\limits_{i=1}^{n}\left\|(\x_i^{k+1}-\x_i^k)-(\x_{n+i}^{k+1}-\x_{n+i}^k)\right\|^2-\|\bflambda^{k}-\hat{\bflambda}^k\|^2\\
&\leq&\|\bflambda^{k+1}-\bflambda^k\|^2+n\beta_k^2\sum\limits_{i=1}^{n}\|\A_i\|^2\|\x_i^{k+1}-\x_i^k\|^2\\
&&+2\beta_k^2\sum\limits_{i=1}^{n}\left(\|\x_i^{k+1}-\x_i^k\|^2+\|\x_{n+i}^{k+1}-\x_{n+i}^k\|^2\right)-\|\bflambda^{k}-\hat{\bflambda}^k\|^2.
\end{eqnarray}
Then we can have
(\ref{eq:basic-identity-multivar-equiv-line1})-(\ref{eq:basic-identity-multivar-equiv}).\qed
\end{proof}

\section{Proof of Theorem~\ref{thm:general_convergence_unbounded}}
To prove Theorem~\ref{thm:general_convergence_unbounded}, we need the following
proposition:

\begin{proposition}\label{prop:inequ_para}
For $\{(\x_1^{k},\cdots,\x_n^k,\lambda^k)\}$ generated by
Algorithm~\ref{alg:general-LADMPSAP-multivar} and a KKT point
$(\x_1^*,\cdots, \x_n^*,\bflambda^*)$ of problem
(\ref{eq:model_problem_multivar}) with $f_i$ described in
Section~\ref{sec:G-LADMPSAP}, we have that
\begin{eqnarray}
&&\sum\limits_{i=1}^n \left( f_i(\x_i^{k+1})-f_i(\x_i^*)+\<\A_i^{\dag} (\lambda^*),\x_i^{k+1}-\x_i^*\> \right) \hspace*{4.9cm} \label{line_1}\\
&\leq& \frac{1}{2}\sum\limits_{i=1}^n\tau_i^{(k)}\left(\|\x_i^{k}-\x_i^*\|^2-\|\x_i^{k+1}-\x_i^*\|^2\right)+\frac{1}{2\beta_k}\left(\|\lambda^{k}-\lambda^*\|^2-\|\lambda^{k+1}-\lambda^*\|^2\right)\hspace*{1cm}\\
&&-\frac{1}{2}\sum\limits_{i=1}^n\left(\tau_i^{(k)}-L_i-n\beta_k\|\A_i\|^2\right)\|\x_i^{k+1}-\x_i^k\|^2-\frac{1}{2\beta_k}\|\hat\lambda^{k}-\lambda^k\|^2\label{line_end}
\end{eqnarray}
\end{proposition}
\begin{proof}
It can be observed that
$$
0\in\partial h_i(\x_i^{k+1})+\nabla g_i(\x_i^k)+\A_i^{\dag} (\hat
\lambda^k)+\tau_i^{(k)}(\x_i^{k+1}-\x_i^k).$$ So we have
$$h_i(\x_i)-h_i(\x_i^{k+1}) \geq \<-\nabla g_i(\x_i^k)-\A_i^{\dag}
(\hat
\lambda^k)-\tau_i^{(k)}(\x_i^{k+1}-\x_i^k),\x_i-\x_i^{k+1}\>,
\quad\forall \x_i,$$ and
\begin{eqnarray}
&&\sum\limits_{i=1}^n f_i(\x_i^{k+1})= \sum\limits_{i=1}^n \left(h_i(\x_i^{k+1})+g_i(\x_i^{k+1})\right)\notag\\
&\leq& \sum\limits_{i=1}^n \left(h_i(\x_i^{k+1})+g_i(\x_i^k)+\<\nabla g_i(\x_i^k),\x_i^{k+1}-\x_i^k\>+\frac{L_i}{2}\|\x_i^{k+1}-\x_i^k\|^2\right)\notag \\
&=& \sum\limits_{i=1}^n \Big(h_i(\x_i^{k+1})+g_i(\x_i^k)+\<\nabla g_i(\x_i^k),\x_i-\x_i^k\>+\<\nabla g_i(\x_i^k),\x_i^{k+1}-\x_i\>\notag\\
&&+\frac{L_i}{2}\|\x_i^{k+1}-\x_i^k\|^2\Big)\notag\\
&\leq& \sum\limits_{i=1}^n \left(g_i(\x_i)+h_i(\x_i)+
\<\A_i^{\dag} (\hat
\lambda^k)+\tau_i^{(k)}(\x_i^{k+1}-\x_i^k),\x_i-\x_i^{k+1}\>+\frac{L_i}{2}\|\x_i^{k+1}-\x_i^k\|^2\right).\notag
\end{eqnarray}
On the one hand,
\begin{eqnarray}
&&\sum\limits_{i=1}^n \left(f_i(\x_i^{k+1})-f_i(\x_i)+\<\A_i^{\dag} (\hat \lambda^k),\x_i^{k+1}-\x_i\>\right)-\<\sum\limits_{i=1}^n \A_i(\x_i^{k+1})-\b,\hat\lambda^k-\lambda\>\hspace*{1.5cm}\label{pop1_line1}\\
&\leq& \sum\limits_{i=1}^n \left(-\tau_i^{(k)}\<\x_i^{k+1}-\x_i^k,\x_i^{k+1}-\x_i\>+\frac{L_i}{2}\|\x_i^{k+1}-\x_i^k\|^2\right)-\<\sum\limits_{i=1}^n \A_i(\x_i^{k+1})-\b,\hat\lambda^k-\lambda\>\notag\\
&=& \sum\limits_{i=1}^n \left(-\tau_i^{(k)}\<\x_i^{k+1}-\x_i^k,\x_i^{k+1}-\x_i\>+\frac{L_i}{2}\|\x_i^{k+1}-\x_i^k\|^2\right)-\frac{1}{\beta_k}\<\lambda^{k+1}-\lambda^k,\hat\lambda^k-\lambda\>\notag\\
&=& \sum\limits_{i=1}^n \left[\frac{\tau_i^{(k)}}{2}\left(\|\x_i^k-\x_i\|^2-\|\x_i^{k+1}-\x_i\|^2-\|\x_i^{k+1}-\x_i^k\|^2\right)+\frac{L_i}{2}\|\x_i^{k+1}-\x_i^k\|^2\right]\notag\\
&&-\frac{1}{2\beta_k}\left(\|\lambda^{k+1}-\lambda\|^2-\|\lambda^{k}-\lambda\|^2+\|\hat{\lambda}^{k}-\lambda^k\|^2-\|\lambda^{k+1}-\hat{\lambda}^k\|^2\right)\notag\\
&=& \sum\limits_{i=1}^n \left[\frac{\tau_i^{(k)}}{2}\left(\|\x_i^k-\x_i\|^2-\|\x_i^{k+1}-\x_i\|^2-\|\x_i^{k+1}-\x_i^k\|^2\right)+\frac{L_i}{2}\|\x_i^{k+1}-\x_i^k\|^2\right]\notag\\
&&-\frac{1}{2\beta_k}\left(\|\lambda^{k+1}-\lambda\|^2-\|\lambda^{k}-\lambda\|^2+\|\hat{\lambda}^{k}-\lambda^k\|^2-\beta_k^2\lbar\sum\limits_{i=1}^n\A_i(\x_i^{k+1}-\x_i^k)\rbar^2\right)\notag\\
&\leq&\frac{1}{2}\sum\limits_{i=1}^n
\tau_i^{(k)}\left(\|\x_i^{k}-\x_i\|^2-\|\x_i^{k+1}-\x_i\|^2\right)-\frac{1}{2}\sum\limits_{i=1}^n\left(\tau_i^{(k)}
-L_i-n\beta_k\|\A_i\|^2\right)\|\x_i^{k+1}-\x_i^k\|^2\notag\\
&&+\frac{1}{2\beta_k}\left(\|\lambda^{k}-\lambda\|^2-\|\lambda^{k+1}-\lambda\|^2-\|\hat\lambda^{k}-\lambda^k\|^2\right).
\label{pop1_line_end}
\end{eqnarray}
On the other hand,
\begin{eqnarray}
&&\sum\limits_{i=1}^n \left(f_i(\x_i^{k+1})-f_i(\x_i)+\<\A_i^{\dag} (\lambda),\x_i^{k+1}-\x_i\>\right)-\<\sum\limits_{i=1}^n \A_i(\x_i)-\b,\hat\lambda^k-\lambda\>\notag\\
&=& \sum\limits_{i=1}^n \left(f_i(\x_i^{k+1})-f_i(\x_i)+\<\A_i^{\dag} (\hat \lambda^k),\x_i^{k+1}-\x_i\>\right)-\<\sum\limits_{i=1}^n \A_i(\x_i^{k+1})-\b,\hat\lambda^k-\lambda\>.\notag
\end{eqnarray}
So we have
\begin{eqnarray}
&&\sum\limits_{i=1}^n \left(f_i(\x_i^{k+1})-f_i(\x_i)+\<\A_i^{\dag} (\lambda),\x_i^{k+1}-\x_i\>\right)-\<\sum\limits_{i=1}^n \A_i(\x_i)-\b,\hat\lambda^k-\lambda\>\notag\\
&\leq&\frac{1}{2}\sum\limits_{i=1}^n \tau_i^{(k)}\left(\|\x_i^{k}-\x_i\|^2-\|\x_i^{k+1}-\x_i\|^2\right)\notag\\
&&-\frac{1}{2}\sum\limits_{i=1}^n\left(\tau_i^{(k)}-L_i-n\beta_k\|\A_i\|^2\right)\|\x_i^{k+1}-\x_i^k\|^2\notag\\
&&+\frac{1}{2\beta_k}\left(\|\lambda^{k}-\lambda\|^2-\|\lambda^{k+1}-\lambda\|^2-\|\hat\lambda^{k}-\lambda^k\|^2\right).\notag
\end{eqnarray}
Let $\x_i=\x_i^*$ and $\lambda=\lambda^*$, we have
\begin{eqnarray}
&&\sum\limits_{i=1}^n \left(f_i(\x_i^{k+1})-f_i(\x_i^*)+\<\A_i^{\dag} (\lambda^*),\x_i^{k+1}-\x_i^*\>\right) \notag\\
&\leq& \frac{1}{2}\sum\limits_{i=1}^n \tau_i^{(k)}\left[\|\x_i^{k}-\x_i^*\|^2-\|\x_i^{k+1}-\x_i^*\|^2\right]+\frac{1}{2\beta_k}\left(\|\lambda^{k}-\lambda^*\|^2-\|\lambda^{k+1}-\lambda^*\|^2\right)\notag\\
&&-\frac{1}{2}\sum\limits_{i=1}^n\left(\tau_i^{(k)}-L_i-n\beta_k\|\A_i\|^2\right)\|\x_i^{k+1}-\x_i^k\|^2-\frac{1}{2\beta_k}\|\hat\lambda^{k}-\lambda^k\|^2.\notag
\end{eqnarray}
\end{proof}

\begin{proof}{\bf (of Theorem~\ref{thm:general_convergence_unbounded})}
As $\x^*$ minimizes $\sum\limits_{i=1}^n
f(\x_i)+\<\lambda^*,\sum\limits_{i=1}^n \A_i(\x_i)-\b\>$, we have
$$
0\leq \sum\limits_{i=1}^n \left(f_i(\x_i^{k+1})-f_i(\x_i^*)+\<\A_i^{\dag} (\lambda^*),\x_i^{k+1}-\x_i^*\>\right).
$$
By Proposition \ref{prop:inequ_para}, we have
\begin{eqnarray}
&&\sum\limits_{i=1}^n\frac{1}{2}\left(\tau_i^{(k)}-L_i-n\beta_k\|\A_i\|^2\right)\|\x_i^{k+1}-\x_i^k\|^2+\frac{1}{2\beta_k}\|\hat\lambda^{k}-\lambda^k\|^2\notag\\
&\leq&\frac{1}{2}\sum\limits_{i=1}^n
\tau_i^{(k)}\left(\|\x_i^{k}-\x_i^*\|^2-\|\x_i^{k+1}-\x_i^*\|^2\right)+\frac{1}{2\beta_k}\left(\|\lambda^{k}-\lambda^*\|^2-\|\lambda^{k+1}-\lambda^*\|^2\right).\notag
\end{eqnarray}
Dividing both sides by $\beta_k$ and using
$\tau_i^{(k)}-L_i-n\beta_k\|\A_i\|^2\geq
\beta_k(\eta_i-n\|\A_i\|^2)$, the non-decrement of $\beta_k$ and
the non-increment of $\beta_k^{-1}\tau_i^{(k)}$, we have
\begin{eqnarray}
&&\frac{1}{2}\sum\limits_{i=1}^n\left(\eta_i-n\|\A_i\|^2\right)\|\x_i^{k+1}-\x_i^k\|^2+\frac{1}{2\beta_k^2}\|\hat\lambda^{k}-\lambda^k\|^2\label{the_line1}\\
&\leq&\frac{1}{2}\sum\limits_{i=1}^n\left(\beta_k^{-1}\tau_i^{(k)}\|\x_i^{k}-\x_i^*\|^2-\beta_{k+1}^{-1}\tau_i^{(k+1)}\|\x_i^{k+1}-\x_i^*\|^2\right)\notag\\
&&+\left(\frac{1}{2\beta_k^2}\|\lambda^{k}-\lambda^*\|^2-\frac{1}{2\beta_{k+1}^2}\|\lambda^{k+1}-\lambda^*\|^2\right).\label{the_line2}
\end{eqnarray}
It can be easily seen that $(\x_1^k,\cdots,\x_n^k,\lambda^k)$ is bounded, hence has an accumulation point, say
$(\mathbf{x}_1^{k_j},\cdots,\mathbf{x}_n^{k_j},\bflambda^{k_j}) \to
(\mathbf{x}_1^{\infty},\cdots,\mathbf{x}_n^{\infty},\bflambda^{\infty})$.

Summing \eqref{the_line1}-\eqref{the_line2} over
$k=0,\cdots,\infty$, we have
\begin{eqnarray}
&&\frac{1}{2}\sum\limits_{i=1}^n\left(\eta_i-n\|\A_i\|^2\right)\sum\limits_{k=0}^\infty\|\x_i^{k+1}-\x_i^k\|^2+\sum\limits_{k=0}^\infty\frac{1}{2\beta_k^2}\|\hat\lambda^{k}-\lambda^k\|^2\notag\\
&\leq&\frac{1}{2}\sum\limits_{i=1}^n\beta_0^{-1}\tau_i^{(0)}\|\x_i^{0}-\x_i^*\|^2+\frac{1}{2\beta_0^2}\|\lambda^{0}-\lambda^*\|^2.\notag
\end{eqnarray}
So $\|\x_i^{k+1}-\x_i^k\|\rightarrow 0$ and
$\beta_k^{-2}\|\hat\lambda^{k}-\lambda^k\|\rightarrow 0$ as
$k\rightarrow \infty$. Hence
$\lbar\sum\limits_{i=1}^n\A_i(\x_i^{k})-\b\rbar\rightarrow 0$,
which means that $\x_1^\infty,\cdots,\x_n^\infty$ is a feasible
solution.

From (\ref{pop1_line1})-(\ref{pop1_line_end}), we have
\begin{eqnarray}
&&\sum\limits_{i=1}^n \left(f_i(\x_i^{k_j+1})-f_i(\x_i)+\<\A_i^{\dag} (\hat \lambda^{k_j}),\x_i^{k_j+1}-\x_i\>\right)-\<\sum\limits_{i=1}^n \A_i(\x_i^{k_j+1})-\b,\hat\lambda^{k_j}-\lambda\>\notag\\
&\leq&\frac{1}{2}\sum\limits_{i=1}^n \tau_i^{(k_j)}\left(\|\x_i^{k_j}-\x_i\|^2-\|\x_i^{k_j+1}-\x_i\|^2\right)\notag\\
&&-\frac{1}{2}\sum\limits_{i=1}^n\left(\tau_i^{(k_j)}-L_i-n\beta_{k_j}\|\A_i\|^2\right)\|\x_i^{k_j+1}-\x_i^{k_j}\|^2\notag\\
&&+\frac{1}{2\beta_{k_j}}\left(\|\lambda^{k_j}-\lambda\|^2-\|\lambda^{k_j+1}-\lambda\|^2-\|\hat\lambda^{k_j}-\lambda^{k_j}\|^2\right).\notag
\end{eqnarray}
Let $j\rightarrow \infty$. By the boundedness of $\tau_i^{(k_j)}$
we have
$$
\sum\limits_{i=1}^n \left(f_i(\x_i^\infty)-f_i(\x_i)+\<\A_i^{\dag} (\lambda^\infty),\x_i^\infty-\x_i\>\right)\leq 0,\quad \forall \x_i.
$$
Together with the feasibility of $(\x_1^\infty,\cdots,\x_n^\infty)$, we can see that $(\x_1^\infty,\cdots,\x_n^\infty,\lambda^\infty)$ is a KKT point.

By choosing
$(\mathbf{x}_1^*,\cdots,\mathbf{x}_n^*,\bflambda^*)=(\mathbf{x}_1^{\infty},\cdots,\mathbf{x}_n^{\infty},\bflambda^{\infty})$
we have
$$
\sum\limits_{i=1}^n\eta_i\|\x_i^{k_j}-\x_i^{\infty}\|^2
+\frac{1}{\beta_{k_j}^2}\|\bflambda^{k_j}-\bflambda^{\infty}\|^2 \to 0.
$$
Using (\ref{the_line1})-(\ref{the_line2}), we have
$$
\sum\limits_{i=1}^n\eta_i\|\x_i^{k}-\x_i^{\infty}\|^2
+\frac{1}{\beta_{k}^2}\|\bflambda^{k}-\bflambda^{\infty}\|^2 \to 0. $$
So
$(\x_1^k,\cdots,\x_n^k,\bflambda^k) \to
(\x_1^{\infty},\cdots,\x_n^{\infty},\bflambda^{\infty})$.
\end{proof}

\section{Proof of Theorem~\ref{thm:general_convergence_unbounded_rate}}
\begin{proof}{\bf (of Theorem~\ref{thm:general_convergence_unbounded_rate})}
By the definition of $\alpha$ and $\tau_i^{(k)}$,
\begin{eqnarray}
&&\dfrac{1}{2}\left[\sum\limits_{i=1}^n\left(\tau_i^{(k)}-L_i-n\beta_k\|\A_i\|^2\right)\|\x_i^{k+1}-\x_i^k\|^2+\frac{1}{\beta_k}\|\hat\lambda^{k}-\lambda^k\|^2\right]\hspace*{2cm}\\
&\geq&\dfrac{\beta_k}{2}\left[\sum\limits_{i=1}^n\left(\eta_i-n\|\A_i\|^2\right)\|\x_i^{k+1}-\x_i^k\|^2+\frac{1}{\beta_k^2}\|\hat\lambda^{k}-\lambda^k\|^2\right]\notag\\
&\geq&\dfrac{\alpha\beta_k}{2}(n+1)\left(\sum\limits_{i=1}^n\|\A_i\|^2\|\x_i^{k+1}-\x_i^k\|^2+\frac{1}{\beta_k^2}\|\hat\lambda^{k}-\lambda^k\|^2\right)\notag\\
&\geq&\dfrac{\alpha\beta_k}{2}(n+1)\left(\sum\limits_{i=1}^n\|\A_i(\x_i^{k+1}-\x_i^k)\|^2+\frac{1}{\beta_k^2}\|\hat\lambda^{k}-\lambda^k\|^2\right)\notag\\
&=&\dfrac{\alpha\beta_k}{2}(n+1)\left(\sum\limits_{i=1}^n\|\A_i(\x_i^{k+1}-\x_i^k)\|^2+\lbar\sum\limits_{i=1}^n\A_i(\x_i^k)-\b\rbar^2\right)\notag\\
&\geq&\dfrac{\alpha\beta_k}{2}\lbar\sum\limits_{i=1}^n\A_i(\x_i^{k+1})-\b\rbar^2.
\end{eqnarray}
So by (\ref{line_1})-(\ref{line_end}) and the non-decrement of
$\beta_k$, we have
\begin{eqnarray}
&&\sum\limits_{i=1}^n \left(f_i(\x_i^{k+1})-f_i(\x_i^*)+\<\A_i^{\dag} (\lambda^*),\x_i^{k+1}-\x_i^*\>\right)+\frac{\alpha\beta_0}{2}\lbar\sum\limits_{i=1}^n\A_i(\x_i^{k+1})-\b\rbar^2\hspace*{2cm}\\
&\leq& \sum\limits_{i=1}^n
\left(f_i(\x_i^{k+1})-f_i(\x_i^*)+\<\A_i^{\dag}
(\lambda^*),\x_i^{k+1}-\x_i^*\>\right)+\frac{\alpha\beta_k}{2}\lbar\sum\limits_{i=1}^n\A_i(\x_i^{k+1})-\b\rbar^2\hspace*{2cm}\\
&\leq&
\frac{1}{2}\sum\limits_{i=1}^n\tau_i^{(k)}\left(\|\x_i^{k}-\x_i^*\|^2-\|\x_i^{k+1}-\x_i^*\|^2\right)+\frac{1}{2\beta_k}\left(\|\lambda^{k}-\lambda^*\|^2-\|\lambda^{k+1}-\lambda^*\|^2\right).
\end{eqnarray}
Dividing both sides by $\beta_k$ and using the non-decrement of
$\beta_k$ and the non-increment of $\beta_k^{-1}\tau_i^{(k)}$, we
have
\begin{eqnarray}
&&\frac{1}{\beta_k}\left[\sum\limits_{i=1}^n \left(f_i(\x_i^{k+1})-f_i(\x_i^*)+\<\A_i^{\dag} (\lambda^*),\x_i^{k+1}-\x_i^*\>\right)+\frac{\alpha\beta_0}{2}\lbar\sum\limits_{i=1}^n\A_i(\x_i^{k+1})-\b\rbar^2\right]\hspace*{1cm}\\
&\leq& \frac{1}{2}\sum\limits_{i=1}^n \beta_k^{-1}\tau_i^{(k)}\left(\|\x_i^{k}-\x_i^*\|^2-\|\x_i^{k+1}-\x_i^*\|^2\right)+\frac{1}{2\beta_k^2}\left(\|\lambda^{k}-\lambda^*\|^2-\|\lambda^{k+1}-\lambda^*\|^2\right)\notag\\
&\leq&\frac{1}{2}\sum\limits_{i=1}^n
\left(\beta_k^{-1}\tau_i^{(k)}\|\x_i^{k}-\x_i^*\|^2-\beta_{k+1}^{-1}\tau_i^{(k+1)}\|\x_i^{k+1}-\x_i^*\|^2\right)\notag\\
&&+\left(\frac{1}{2\beta_k^2}\|\lambda^{k}-\lambda^*\|^2-\frac{1}{2\beta_{k+1}^2}\|\lambda^{k+1}-\lambda^*\|^2\right).
\end{eqnarray}
Summing over $k=0,\cdots,K$ and dividing both sides by
$\sum\limits_{k=0}^K \beta_k^{-1}$, we have
\begin{eqnarray}
&&\sum\limits_{i=1}^n \left(\sum\limits_{k=0}^K\gamma^k f_i(\x_i^{k+1})-f_i(\x_i^*)+\<\A_i^{\dag} (\lambda^*),\sum\limits_{k=0}^K\gamma^k\x_i^{k+1}-\x_i^*\>\right)\hspace*{2cm}\\
&&+\frac{\alpha\beta_0}{2}\sum\limits_{k=0}^K\gamma^k\lbar\sum\limits_{i=1}^n\A_i(\x_i^{k+1})-\b\rbar^2\\
&\leq&\left(\sum\limits_{i=1}^n
\beta_0^{-1}\tau_i^{(0)}\|\x_i^{0}-\x_i^*\|^2+\beta_0^{-2}\|\lambda^{0}-\lambda^*\|^2\right)/\sum\limits_{k=0}^K
2\beta_k^{-1}.
\end{eqnarray}
Using the convexity of $f_i$ and $\|\cdot\|^2$, we have
\begin{eqnarray}
&&\sum\limits_{i=1}^n \left(f_i(\bar\x_i^{K})-f_i(\x_i^*)+\<\A_i^{\dag} (\lambda^*),\bar\x_i^{K}-\x_i^*\>\right)+\frac{\alpha\beta_0}{2}\lbar\sum\limits_{i=1}^n\A_i(\bar\x_i^{K})-\b\rbar^2\\
&\leq&\sum\limits_{i=1}^n \left(\sum\limits_{k=0}^K\gamma^k f_i(\x_i^{k+1})-f_i(\x_i^*)+\<\A_i^{\dag} (\lambda^*),\sum\limits_{k=0}^K\gamma^k\x_i^{k+1}-\x_i^*\>\right)\\
&&+\frac{\alpha\beta_0}{2}\sum\limits_{k=0}^K\gamma^k\lbar\sum\limits_{i=1}^n\A_i(\x_i^{k+1})-\b\rbar^2.
\end{eqnarray}
So we have
\begin{eqnarray}
&&\sum\limits_{i=1}^n \left(f_i(\bar\x_i^{K})-f_i(\x_i^*)+\<\A_i^{\dag} (\lambda^*),\bar\x_i^{K}-\x_i^*\>\right)+\frac{\alpha\beta_0}{2}\lbar\sum\limits_{i=1}^n\A_i(\bar\x_i^{K})-\b\rbar^2\\
&\leq&\left(\sum\limits_{i=1}^n
\beta_0^{-1}\tau_i^{(0)}\|\x_i^{0}-\x_i^*\|^2+\beta_0^{-2}\|\lambda^{0}-\lambda^*\|^2\right)/\sum\limits_{k=0}^K
2\beta_k^{-1}.
\end{eqnarray}
\end{proof}

\bibliographystyle{spbasic}      

\begin{thebibliography}{0}
\providecommand{\natexlab}[1]{#1}
\providecommand{\url}[1]{{#1}}
\providecommand{\urlprefix}{URL }
\expandafter\ifx\csname urlstyle\endcsname\relax
  \providecommand{\doi}[1]{DOI~\discretionary{}{}{}#1}\else
  \providecommand{\doi}{DOI~\discretionary{}{}{}\begingroup
  \urlstyle{rm}\Url}\fi
\providecommand{\eprint}[2][]{\url{#2}}

\end{thebibliography}


\begin{thebibliography}{}
\bibitem[{Beck and Teboulle(2009)}]{Beck2009}
Beck A, Teboulle M (2009) A fast iterative shrinkage-thresholding algorithm for
  linear inverse problems. {SIAM} J Imaging Sciences 2(1):183--202

\bibitem[{Boyd and Vandenberghe(2004)}]{Boyd-Convex-Optimization}
Boyd S, Vandenberghe L (2004) Convex optimization. Cambridge University Press

\bibitem[{Boyd et~al(2011)Boyd, Parikh, Chu, Peleato, and
  Eckstein}]{Boyd-2011-Distributed}
Boyd S, Parikh N, Chu E, Peleato B, Eckstein J (2011) Distributed
optimization
  and statistical learning via the alternating direction method of multipliers.
  In: Jordan M (ed) Foundations and Trends in Machine Learning

\bibitem[{Cai et~al(2010)Cai, Cand\`{e}s, and Shen}]{Cai-2008-SVT}
Cai J, Cand\`{e}s E, Shen Z (2010) A singular value thresholding algorithm for
  matrix completion. SIAM J Optimization 20(4):1956--1982

\bibitem[{Cand\`{e}s and Recht(2009)}]{Candes-2009-matrix}
Cand\`{e}s E, Recht B (2009) Exact matrix completion via convex optimization.
  Foundations of Computational Mathematics 9(6):717--772

\bibitem[{Cand\`{e}s et~al(2011)Cand\`{e}s, Li, Ma, and
  Wright}]{Candes-2011-RPCA}
Cand\`{e}s E, Li X, Ma Y, Wright J (2011) Robust principal component analysis?
  J ACM 58(3):No.11

\bibitem[{Chandrasekaran et~al(2012)Chandrasekaran, Parrilo, and
  Willsky}]{Chand-2012-GMS}
Chandrasekaran V, Parrilo P, Willsky A (2012) Latent variable graphical model
  selection via convex optimization. The Annals of Statistics 40(4):1935--1967

\bibitem[{Chang(2011)}]{Chang-FoLSMM}
Chang E (2011) Foundations of Large-Scale Multimedia Information Management and
  Retrieval: Mathematics of Perception. Springer-Verlag New York Inc

\bibitem[{Chang et~al(2007)Chang, Zhu, Wang, Bai, Li, Qiu, and
  Cui}]{Chang-2007-Psvm}
Chang E, Zhu K, Wang H, Bai H, Li J, Qiu Z, Cui H (2007) Psvm: Parallelizing
  support vector machines on distributed computers. In: NIPS

\bibitem[{Chen et~al(2013)Chen, He, Ye, and Yuan}]{Chen-2013-Divergent}
Chen C, He B, Ye Y, Yuan X (2013) The direct extension of {ADMM} for
  multi-block convex minimization problems is not necessarily convergent.
  Preprint

\bibitem[{Deng and Yin(2012)}]{deng2012global}
Deng W, Yin W (2012) On the global and linear convergence of the generalized
  alternating direction method of multipliers. Tech. rep., DTIC Document

\bibitem[{Deng et~al(2011)Deng, Yin, and Zhang}]{Deng-2011-ADM}
Deng W, Yin W, Zhang Y (2011) Group sparse optimization by alternating
  direction method. TR11-06, Department of Computational and Applied
  Mathematics, Rice University

\bibitem[{Esser(2009)}]{Esser-2009-SB}
Esser E (2009) Applications of {L}agrangian-based alternating direction methods
  and connections to split {B}regman. CAM Report 09-31, UCLA

\bibitem[{Favaro et~al(2011)Favaro, Vidal, and
  Ravichandran}]{Favaro-2011-Closed}
Favaro P, Vidal R, Ravichandran A (2011) A closed form solution to robust
  subspace estimation and clustering. In: CVPR

\bibitem[{Fazel(2002)}]{Fazel-2002-nuclear}
Fazel M (2002) Matrix rank minimization with applications. PhD thesis

\bibitem[{Fortin and Glowinski(1983)}]{Fortin-1983-ADM}
Fortin M, Glowinski R (1983) Augmented Lagrangian methods. North-Holland

\bibitem[{Goldfarb and Ma(2012)}]{Goldfarb-10-fast}
Goldfarb D, Ma S (2012) Fast multiple splitting algorithms for convex
  optimization. SIAM J Optimization 22(2):533--556

\bibitem[{Goldstein and Osher(2008)}]{Goldstein-2009-SB}
Goldstein T, Osher S (2008) The split {B}regman method for
$\ell_1$ regularized
  problems. SIAM J Imaging Sciences 2(2):323--343

\bibitem[{He and Yuan(2012)}]{He-2012-Rate}
He B, Yuan X (2012) On the $O(1/n)$ convergence rate of the
{D}ouglas-{R}achford alternating direction method. SIAM J
Numerical Analysus 50(2):700--709

\bibitem[{He and Yuan(2013)}]{He-11-LADM_Gauss}
He B, Yuan X (2013) Linearized alternating direction method with {G}aussian
  back substitution for separable convex programming. Numerical Algebra,
  Control and Optimization 3(2):247--260

\bibitem[{He et~al(2012)He, Tao, and Yuan}]{He-2011-Gaussian}
He B, Tao M, Yuan X (2012) Alternating direction method with {G}aussian back
  substitution for separable convex programming. SIAM J Optimization
  22(2):313--340

\bibitem[{Hong and Luo(2012)}]{Luo-2012-LinearADM}
Hong M, Luo ZQ (2012) On the linear convergence of the alternating direction
  method of multipliers. Preprint, arXiv:12083922

\bibitem[{Jacob et~al(2009)Jacob, Obozinski, and Vert}]{Jacob-2009-Group-LASSO}
Jacob L, Obozinski G, Vert J (2009) Group {L}asso with overlap and graph
  {L}asso. In: ICML 

\bibitem[{Ji et~al(2010)Ji, Liu, Shen, and Xu}]{Ji-2010-VideoDenois}
Ji H, Liu C, Shen Z, Xu Y (2010) Robust video denoising using low rank matrix
  completion. In: CVPR

\bibitem[{Lin et~al(2009{\natexlab{a}})Lin, Chen, and Ma}]{Lin09}
Lin Z, Chen M, Ma Y (2009{\natexlab{a}}) The augmented {L}agrange multiplier
  method for exact recovery of corrupted low-rank matrices. UIUC Technical
  Report UILU-ENG-09-2215

\bibitem[{Lin et~al(2009{\natexlab{b}})Lin, Ganesh, Wright, Wu, Chen, and
  Ma}]{Lin-09-APG}
Lin Z, Ganesh A, Wright J, Wu L, Chen M, Ma Y (2009{\natexlab{b}}) Fast convex
  optimization algorithms for exact recovery of a corrupted low-rank matrix.
  UIUC Technical Report UILU-ENG-09-2214

\bibitem[{Lin et~al(2011)Lin, Liu, and Su}]{Lin-2011-LADMAP}
Lin Z, Liu R, Su Z (2011) Linearized alternating direction method with adaptive
  penalty for low-rank representation. In: NIPS

\bibitem[{Liu and Yan(2011)}]{Liu-2011-LLRR}
Liu G, Yan S (2011) Latent low-rank representation for subspace segmentation
  and feature extraction. In: ICCV

\bibitem[{Liu et~al(2010)Liu, Lin, and Yu}]{Liu-2010-LRR}
Liu G, Lin Z, Yu Y (2010) Robust subspace segmentation by low-rank
  representation. In: ICML

\bibitem[{Liu et~al(2012)Liu, Lin, Yan, Sun, Yu, and Ma}]{Liu-2012-LRR}
Liu G, Lin Z, Yan S, Sun J, Yu Y, Ma Y (2012) Robust recovery of subspace
  structures by low-rank representation. IEEE Trans on PAMI 35(1):171--184

\bibitem[{Liu et~al(2013, oral presentation)Liu, Lin, and
  Su}]{Liu-2013-LADMPSAP}
Liu R, Lin Z, Su Z (2013, oral presentation) Linearized alternating direction
  method with parallel splitting and adaptive penalty for separable convex
  programs in machine learning. In: ACML

\bibitem[{Ma et~al(2011)Ma, Goldfarb, and Chen}]{Ma-2008-FPC}
Ma S, Goldfarb D, Chen L (2011) Fixed point and bregman iterative methods for
  matrix rank minimization. Mathematical Programming 128(1-2):321--359

\bibitem[{Meier et~al(2008)Meier, Geer, and
  B\"{u}hlmann}]{Meier-2008-Group-Logistic}
Meier L, Geer SVD, B\"{u}hlmann P (2008) The group {L}asso for logistic
  regression. Journal of the Royal Statistical Society: Series B (Statistical
  Methodology) 70(1):53--71

\bibitem[{Rockafellar(1970)}]{Rockafellar}
Rockafellar R (1970) Convex Analysis. Princeton University Press

\bibitem[{Shen and Wu(2012)}]{Wu-2012-saliency}
Shen X, Wu Y (2012) A unified approach to salient object detection via low rank
  matrix recovery. In: CVPR

\bibitem[{Subramanian et~al(2005)Subramanian, Tamayo, Mootha, Mukherjee, and
  {et al.}}]{Subramanian-2005-MSigDB}
Subramanian A, Tamayo P, Mootha V, Mukherjee S, {et al} (2005) Gene set
  enrichment analysis: A knowledge-based approach for interpreting genome-wide
  expression profiles. Proceedings of the National Academy of Sciences
  102(43):267--288

\bibitem[{Tao(2014)}]{Tao-14-ADM_Parallel}
Tao M (2014) Some parallel splitting methods for separable convex
programming
  with ${O}(1/t)$ convergence rate. Pacific J. Optimization 10(2):359--384

\bibitem[{Toh and Yun(2010)}]{Toh-2009-APG}
Toh K, Yun S (2010) An accelerated proximal gradient algorithm for nuclear norm
  regularized least squares problems. Pacific J Optimization 6(15):615--640

\bibitem[{Tron and Vidal(2007)}]{Tron-2007-Hopkins}
Tron R, Vidal R (2007) A benchmark for the comparison of 3D
montion segmentation algorithms. In: CVPR

\bibitem[{van~de Vijver et~al(2002)van~de Vijver, He, van't Veer, Dai, and {et
  al.}}]{Van-2002-Breast}
van~de Vijver M, He Y, van't Veer L, Dai H, {et al} (2002) A gene-expression
  signature as a predictor of survival in breast cancer. The New England
  Journal of Medicine 347(25):1999--2009

\bibitem[{Wright et~al(2009)Wright, Yang, Ganesh, Sastry, and
  Ma}]{Wright-2009-face}
Wright J, Yang A, Ganesh A, Sastry S, Ma Y (2009) Robust face recognition via
  sparse representation. IEEE Trans on PAMI 31(2):210--227

\bibitem[{Xu et~al(2011)Xu, Yin, and Wen}]{Xu-2011-NMF}
Xu Y, Yin W, Wen Z (2011) An alternating direction algorithm for matrix
  completion with nonnegative factors. CAAM Technical Report TR11-03

\bibitem[{Yang and Yuan(2013)}]{Yang-2011-LADM}
Yang J, Yuan X (2013) Linearized augmented {L}agrangian and alternating
  direction methods for nuclear norm minimization. Mathematics of Computation
  82(281):301--329

\bibitem[{Ye et~al(2008)Ye, Ji, and Chen}]{Ye-2008-MDKL}
Ye J, Ji S, Chen J (2008) Multi-class discriminant kernel learning via convex
  programming. JMLR 9:719--758

\bibitem[{Zhang et~al(2011)Zhang, Burger, and Osher}]{zhang2011unified}
Zhang X, Burger M, Osher S (2011) A unified primal-dual algorithm framework
  based on {B}regman iteration. Journal of Scientific Computing 46(1):20--46

\end{thebibliography}

%
\end{document}